\documentclass[12pt,reqno]{amsart}

\usepackage{amsmath}
\PassOptionsToPackage{hyphens}{url}\usepackage[hidelinks,hypertexnames=false]{hyperref}

\usepackage[utf8]{inputenc}
\usepackage[british]{babel}
\usepackage[noabbrev, nameinlink]{cleveref}
\usepackage{amsthm}
\usepackage{amssymb}
\usepackage{mathrsfs}
\usepackage{indentfirst}
\usepackage{stmaryrd}
\usepackage{bbm}
\usepackage{xspace, xcolor}
\usepackage{thmtools}
\usepackage{thm-restate}
\usepackage{calc}
\usepackage{fullpage}

\makeatletter
\newcommand*{\ie}{i.e\@ifnextchar.{}{.\@\xspace}}
\makeatother
\makeatletter
\newcommand*{\whp}{\emph{w.h.p}\@ifnextchar.{}{.\@\xspace}}
\makeatother
\newcommand*{\Freiman}{Fre\u{\i}man\xspace}
\usepackage[shortlabels]{enumitem}

\newcommand{\E}{\mathbb{E}}
\newcommand{\F}{\mathbb{F}}
\newcommand{\R}{\mathbb{R}}
\newcommand{\Z}{\mathbb{Z}}
\newcommand{\N}{\mathbb{N}}
\newcommand{\Zmodn}{\Z_n}
\newcommand{\eps}{\varepsilon}
\newcommand{\cX}{\mathcal{X}}
\newcommand{\cF}{\mathcal{F}}
\newcommand{\cH}{\mathcal{H}}
\newcommand{\cZ}{\mathcal{Z}}
\newcommand{\cS}{\mathcal{S}}
\newcommand{\cP}{\mathcal{P}}
\newcommand{\cC}{\mathcal{C}}
\newcommand{\cQ}{\mathcal{Q}}
\newcommand{\cI}{\mathcal{I}}

\newcommand{\ccY}{\mathcal{Y}}
\newcommand{\pr}{\mathbb{P}}
\newcommand{\rsum}{\mathbin{\hat{+}}}
\newcommand{\dimF}{\dim_\mathrm{F}}
\newcommand{\rank}{\operatorname{rank}}
\newcommand{\Var}{\operatorname{Var}}
\newcommand{\Span}{\operatorname{span}}
\newcommand{\conv}{\operatorname{conv}}
\newcommand{\diam}{\operatorname{diam}}

\newcommand{\inner}[1]{\left\langle#1\right\rangle}
\newcommand{\equivc}[1]{\llbracket#1\rrbracket}
\newcommand{\oN}[1]{\overline{N}(#1)}
\renewcommand{\complement}[1]{#1^\mathsf{c}}

\crefname{prop}{proposition}{propositions}
\Crefname{prop}{Proposition}{Propositions}
\crefname{obs}{observation}{observations}
\Crefname{obs}{Observation}{Observations}
\crefname{lem}{lemma}{lemmas}
\Crefname{lem}{Lemma}{Lemmas}
\crefname{thm}{theorem}{theorems}
\Crefname{thm}{Theorem}{Theorems}
\crefname{defi}{definition}{definitions}
\Crefname{defi}{Definition}{Definitions}
\crefname{cor}{corollary}{corollaries}
\Crefname{cor}{Corollary}{Corollaries}

\title{On the independence number of \\ sparser random Cayley graphs}

\author{Marcelo Campos \and Gabriel Dahia \and Jo\~ao Pedro Marciano}

\address{IMPA, Estrada Dona Castorina 110, Jardim Bot\^anico, Rio de Janeiro, 22460-320, Brasil}
\email{\{gabriel.dahia, joao.marciano\}@impa.br}

\address{Department of Pure Mathematics and Mathematical Statistics, University of Cambridge, Wilberforce Road, Cambridge, CB3 0WA, UK}
\email{mc2482@cam.ac.uk}


\usepackage[numbers,longnamesfirst]{natbib}
\usepackage{graphicx}

\newtheorem{thm}{Theorem}[section]
\newtheorem{lem}[thm]{Lemma}
\newtheorem{defi}[thm]{Definition}

\newtheorem{prop}[thm]{Proposition}
\newtheorem{conjecture}[thm]{Conjecture}
\newtheorem{obs}[thm]{Observation}

\newtheoremstyle{named}{}{}{\itshape}{}{\bfseries}{.}{.5em}{\thmnote{#3}}
\theoremstyle{named}

\begin{document}

\begin{abstract}
  The Cayley sum graph $\Gamma_A$ of a set $A \subseteq \Zmodn$ is defined to have vertex set $\Zmodn$ and an edge between two distinct vertices $x, y \in \Zmodn$ if $x + y \in A$.
  \citeauthor{GM16} proved that if the set $A$ is a $p$-random subset of $\Zmodn$ with $p = 1/2$, then the independence number of $\Gamma_A$ is asymptotically equal to $\alpha(G(n, 1/2))$ with high probability.
  Our main theorem is the first extension of their result to $p = o(1)$: we show that, with high probability, $$\alpha(\Gamma_A) = (1 + o(1)) \alpha(G(n, p))$$ as long as $p \ge (\log n)^{-1/80}$.

  One of the tools in our proof is a geometric-flavoured theorem that generalises \Freiman's lemma, the classical lower bound on the size of high dimensional sumsets.
  We also give a short proof of this result up to a constant factor; this version yields a much simpler proof of our main theorem at the expense of a worse constant.
\end{abstract}

\maketitle

\section{Introduction}

Let $n \in \mathbb{N}$ be a large prime. The Cayley sum graph $\Gamma_A$ of a set $A \subseteq \Zmodn$ is defined to have $\Zmodn$ as its vertices and an edge between two distinct $x, y \in \Zmodn$ if $x + y \in A$.
\citet[Conjecture 4.4]{Alo07} conjectured that for each $t \in \{1, \dots, n\}$, there exists a set $A \subseteq \Zmodn$ such that $|A| = t$ and $\Gamma_A$ has independence number at most $\frac{n}{t} (\log n)^{O(1)}$, and moreover remarked that ``the conjecture may well hold for a random choice of $A$.''

\citeauthor{Gre05} more generally asked (see \cite[Question 2]{Gre05}) which parameters of $\Gamma_A$ match those of the random graph $G(n, p)$ when $A$ is a random subset of $\Zmodn$.
In that same work, he showed that if $A$ is chosen uniformly at random over all subsets of $\Zmodn$, then with high probability the size of the largest independent set\footnote{\citeauthor{Gre05} actually states this for the clique number, but this is equivalent to what we state since a ${1/2}$-random subset of $\Zmodn$ has the same distribution of its complement.} of $\Gamma_A$ is at most $\alpha(G(n, 1/2))$, up to a constant factor.

This result is in fact a consequence of a theorem in additive combinatorics.
Recall that the sumset $X + Y$ and the restricted sumset $X \rsum X$ are defined as
\begin{equation*}
  X + Y = \{x + y : x \in X, y \in Y\}, \qquad X \rsum X = \{x_1 + x_2 : x_1, x_2 \in X, x_1 \neq x_2\},
\end{equation*}
and observe that $X \subseteq \Zmodn$ is independent in $\Gamma_A$ exactly when $X \rsum X \subseteq \complement{A}$.
\citeauthor{Gre05} deduced his bound on the typical size of $\alpha(\Gamma_A)$ from estimates on the number of sets $X \subseteq \Zmodn$ such that $|X| = k$ and $|X + X| \le \sigma k$, where $\sigma$ is some upper bound on the doubling of $X$.
When referring to the doubling of the set itself, we will use $\sigma[X]$, \ie $\sigma[X] = |X + X| / |X|$.

Later, \citet*{GM16} showed that better estimates on the number of such sets are possible if one handles ranges of $\sigma[X]$ differently.
When $\sigma[X]$ is constant, they used an arithmetic regularity lemma to obtain a tighter bound.
Leveraging the isoperimetric inequality on $\Z^d$, they obtain a bound on the number of ``quasi-random'' elements in sets with doubling close to maximum, thus refining \citeauthor{Gre05}'s estimate for the number of such sets.
As a consequence, they determined the correct constant for the independence number of $\Gamma_A$ when $A \subseteq \Zmodn$ is uniform random, showing that $\alpha(\Gamma_A)$ asymptotically matches $\alpha(G(n, 1/2))$ with high probability~\cite{GM16}.

Our main \namecref{stmt:main-result} extends the above result of \citeauthor*{GM16} to the more challenging setting when $A$ is a sparse random subset of $\Zmodn$, and can be seen as progress towards \citeauthor{Alo07}'s aforementioned conjecture.
To state our result, it is useful to denote by $\Gamma_p$ the (random) Cayley sum graph $\Gamma_A$ when $A$ is a $p$-random subset of $\Zmodn$ for $p = p(n)$.

\begin{thm}\label{stmt:main-result}
  Let $n$ be a prime number and let $p = p(n)$ satisfy $p \ge (\log n)^{-1/80}$.
  The random Cayley sum graph $\Gamma_p$ of $\Zmodn$ satisfies
  \begin{equation}\label{eq:main-result}
    \alpha(\Gamma_p) = \big(2 + o(1)\big) \log_{\frac{1}{1 - p}} n
  \end{equation}
  with high probability as $n \to \infty$.
\end{thm}

If $p = o(1)$, then this bound is asymptotically equal to $(2 + o(1)) (\log n) / p$; it also matches $\alpha(G(n, p))$ in the corresponding range of $p$~\cite{Bol85, JLR00}.
We remark that we do not believe the lower bound on $p$ in \Cref{stmt:main-result} to be close to optimal, so we do not optimize the constant in the exponent of the $\log$.
Nonetheless, there is a natural barrier for our methods when proving upper bounds at around $p \approx (\log n)^{-1}$ (see \Cref{sec:concluding-remarks}).

The lower bound in \eqref{eq:main-result} follows from a second moment computation together with a simple combinatorial argument (see \Cref{sec:lower-bound}).
The proof of the upper bound, on the other hand, is much more challenging, even though it has the same broad outline pioneered by \citet*{Gre05}.
In fact, we use his bounds on the number of sets with given doubling for the ``middle'' range $k^\beta < \sigma[X] = o(k)$, and a minor extension (that follows from the same proof) of the theorem of \citet{GM16} for the ``upper'' range $\sigma[X] = \Theta(k)$.
Our contribution is therefore in the ``lower'' doubling range, $\sigma[X] \le k^\beta$ for a small constant $\beta > 0$, where even the correct count of such sets would not be enough to prove our main \namecref{stmt:main-result}.

To overcome this obstacle, we show that each of those sets contains a much smaller subset $F$ -- we call it a ``fingerprint'' of $X$ -- so that, after determining that $F$ possesses some special properties, it suffices to count the fingerprints to deduce \Cref{stmt:main-result}.
This idea is inspired by the application of the asymmetric container method of \citet*{MSS24} to the problem of counting sets with small sumset by the first author~\cite{Cam20}.
Indeed, we could also use asymmetric containers as a tool to prove our main \namecref{stmt:main-result}, but we prefer this simpler approach as it results in better bounds and a more self-contained proof.

The special property we require of each fingerprint $F \subseteq X$ is having a sufficiently large sumset.
Ideally, we would like $F + F$ to be as large as $X + X$; this obviously imposes a lower bound of $|F| \ge |X + X|^{1/2}$.
\citeauthor{BLT22+b} recently studied the question of finding such sets for general Abelian groups: they obtain $F \subseteq X$ satisfying $|F + F| \approx |X + X|$ when $X$ has bounded doubling~\cite[Theorem 2]{BLT22+b}.
However, their result does not serve our needs because it does not handle the case of $\sigma[X]$ being $k^\beta$ for fixed, small $\beta > 0$.

We pursue a different approach to obtain our collection of fingerprints, one that is able to handle sets with doubling that is polynomial in their size.
This approach yields fingerprints such that $|F| \approx |X + X|^{1/2} \log |X|$, which is only a logarithmic factor away from the ideal bound, but $|F + F|$ is in general far from $|X + X|$.
What we can show is that the size of the sumset is as large as in \Freiman's lemma, which is just enough for our application.

It is therefore useful to recall \Freiman's lemma, and to do so we need a definition.
We will write $\rank(X)$ for the minimum dimension of a hyperplane that contains a set $X \subseteq \R^d$; if $\rank(X) = d$, then we say $X$ has full rank.
The statement of the lemma is then:
\begin{lem}[\citet{Fre99}]\label{stmt:freimans-lemma}
  Let $X \subseteq \R^d$ be a finite set of full rank.
  Then,
  \begin{equation*}
    |X + X| \ge (d + 1)|X| - \binom{d + 1}{2}.
  \end{equation*}
\end{lem}

Observe that \Cref{stmt:freimans-lemma} is a statement about subsets of $\R^d$ rather than $\Zmodn$.
To amend this, we can use the (by now) standard machinery of \Freiman isomorphisms (see, for example, \cite[Chapter 5.3]{TV06}) to reach similar conclusions for subsets of the integers modulo $n$, changing rank for \Freiman dimension when appropriate.
We therefore change the setting to $\R^d$ for the rest of this informal discussion.

Like applications of the hypergraph container method, the way to show that $F + F$ is almost as large as $(d + 1)|X|$ is by proving a suitable supersaturation result (see \Cref{stmt:supsat}, below).
In fact, from this result we will be able to show that picking a fingerprint randomly inside $X$ works.
\Cref{stmt:supsat} informally says the following: if a set $Y$ approximates $X + X$ in terms of its popular sums, then it is (almost) as large as the bound in \Freiman's lemma.

This key step in the proof of \Cref{stmt:supsat} is a result stating that we can find some small $T \subseteq X$ such that $X + T$ almost attains the bound in \Freiman's lemma.
Proving this \namecref{stmt:freiman-lemma-via-few-translates} is the most difficult part of our proof, and we consider it to be of independent interest.
\begin{restatable}{thm}{freimanlemmafewtrans}
  \label{stmt:freiman-lemma-via-few-translates}
  Let $d, r \in \N$, let $\gamma > r^{-1/3}$ and let $X \subseteq \R^d$ be a finite set with $\rank(X) \ge r$.
  There exists $T \subseteq X$ such that $|T| \le 4(r + 1)/\gamma$ and
  \begin{equation*}
    |X + T| \ge (1 - 5\gamma) \frac{(r + 1)|X|}{2}.
  \end{equation*}
\end{restatable}

The proof of this theorem is technical; however, we will also show that if we were satisfied with $1/6$ being the leading constant instead of $(1 - 5\gamma)/2$, then a much simpler approach would suffice.
This weaker version would also be enough to prove the upper bound in \Cref{stmt:main-result} with 6 being the leading constant instead of 2.

Recently, \citet{JM23+} proved the following \namecref{stmt:jing-mudgal-freiman} that is closely related to \Cref{stmt:freiman-lemma-via-few-translates}.
It obtains the correct leading constant, \ie without the $(1 - 5\gamma)/2$ or $1/6$ loss in our bound, at the expense of a worse relationship between the number of translates and the dimension of the set:
\begin{thm}[{\cite[Theorem 1.2]{JM23+}}]\label{stmt:jing-mudgal-freiman}
  Given $d \in \N$, there exists a constant $C = C(d) > 0$ such that, for every finite set of full rank $X \subseteq \R^d$, there exist $x_1, \ldots, x_C \in X$ satisfying
  \begin{equation*}
    |X + \{x_1, . . . , x_C\}| \ge (d + 1)|X| - 5(d + 1)^3.
  \end{equation*}
\end{thm}

\Cref{stmt:jing-mudgal-freiman} is part of a recent line of work~\cite{BLT22+a, FLPZ24} that relies on a beautiful theorem of \citet*[Theorem 8]{BLT22+a} to obtain sumset lower bounds via few-translates.
Unfortunately, the proof of this theorem uses \citeauthor{Sha19}'s~\cite{Sha19} almost-all Balog--Szemer\'edi--Gowers theorem, and, as a consequence, the dependency of the number of translates $C$ on the dimension $d$ is of tower-type~\cite{BLT22+a, Sha19}.

A related result, which avoids super-polynomial dependencies between its parameters, is the following elegant \namecref{stmt:flpz} proved by \citet*{FLPZ24}.
Its proof relies on a clever path counting argument akin to Gowers' proof of the Balog--Szemer\'edi theorem.
Note that we state a specialized version of their much more general result.
\begin{thm}[{\cite[Theorem 1.1]{FLPZ24}}]\label{stmt:flpz}
  There exists $c > 0$ such that the following holds.
  Let $X \subseteq \R^d$ with $|X| = k$.
  For every $s \in \{1, \ldots, k\}$, there exists $T \subseteq X$ such that $|T| \le s$ and
  \begin{equation}\label{eq:flpz}
    |X + T| \ge c \min\big\{\sigma[X]^{1/3}, s\big\}|X|.
  \end{equation}
\end{thm}

This result would work for our needs if we could ensure $\sigma[X] > c^{-1} d^3$, but we cannot.
Alas, \citeauthor*{FLPZ24} also exhibit a construction~\cite[Proposition 2.4]{FLPZ24} which shows that we cannot even replace $\sigma[X]^{1/3}$ in \eqref{eq:flpz} by $\sigma[X]^{1/1.29}$; this means that our requirement for the rank of the set in \Cref{stmt:freiman-lemma-via-few-translates} is essential.

In the next \namecref{sec:overview}, we give a detailed sketch of our proof strategy, and a proof of our fingerprint \namecref{stmt:fingerprints}.
Then, in \Cref{sec:freiman-via-few-translates}, we give a simple proof of a weaker version of our ``\Freiman's lemma via few translates'' theorem that nevertheless contains some of the main ideas required for the full proof of \Cref{stmt:freiman-lemma-via-few-translates}.
\Cref{sec:improved-freiman-via-few-translates} is dedicated to improving the constant to $(1 - 5\gamma)/2$, assuming two additional technical lemmas, which we prove in \Cref{sec:offsetting-the-loss-of-the-zero-fibre,sec:weighted-freiman}.
In \Cref{sec:supsat}, we derive our supersaturation result from \Cref{stmt:freiman-lemma-via-few-translates}, and in \Cref{sec:main-theorem,sec:lower-bound} we complete the proof of our main \namecref{stmt:main-result}.
Finally, in \Cref{sec:concluding-remarks}, we discuss future work and open problems.

\section{Overview of the proof}\label{sec:overview}

Throughout, $n \in \N$ will always be a sufficiently large prime number; we will also adopt the standard convention of omitting floors and ceilings whenever they are not essential.
Let $k \in \N$ be the bound that we want to show for the independence number and let $A_p$ be a $p$-random subset of $\Zmodn$.
Denoting $\cX := {\Zmodn \choose k}$, we will show that
\begin{equation*}
  \pr\left(\exists \, X \in \cX : X \rsum X \subseteq \complement{A_p}\right) \to 0
\end{equation*}
as $n \to \infty$, which is equivalent to proving that $\alpha(\Gamma_p) < k$ \whp.

We will follow \citeauthor{Gre05}'s general approach of partitioning $\cX = \cX_1 \cup \cX_2 \cup \cX_3$ where each sub-collection is defined based on the doubling $\sigma[X]$,
\begin{align*}
  \cX_1 & = \left \{ X \in \cX : \sigma[X] \le k^\beta \right \}, \\
  \cX_2 & = \left \{ X \in \cX : k^\beta < \sigma[X] \le \delta k/10 \right \}, \text{ and} \\
  \cX_3 & = \left \{ X \in \cX : \delta k/10 < \sigma[X] \right \},
\end{align*}
for $\beta = 1/40$ and some $\delta = o(1)$, and handle each one differently.
Explicitly, we use the union bound to deduce
\begin{equation*}
  \pr\left(\exists \, X \in \cX : X \rsum X \subseteq \complement{A}\right) \le \sum_{i = 1}^3 \pr\left(\exists \, X \in \cX_i : X \rsum X \subseteq \complement{A}\right).
\end{equation*}

It turns out that we can bound the terms related to $\cX_2$ and $\cX_3$ using the same techniques of \citet{Gre05} and \citet*{GM16}; we therefore defer working out the details regarding these terms to \Cref{sec:main-theorem}.
For the remaining term, we could try to follow the methods in \cite{GM16} and take a union bound over the sets in $\cX_1$:
\begin{equation}\label{eq:the-sum}
  \pr(\exists \, X \in \cX_1 : X \rsum X \subseteq \complement{A_p}) \le \sum_{X \in \cX_1} \pr(X \rsum X \subseteq \complement{A_p}).
\end{equation}
Observe that, for each $X \subseteq \Zmodn$, the probability in the summand is
\begin{equation*}
  \pr(X \rsum X \subseteq \complement{A_p}) = (1 - p)^{|X \rsum X|}.
\end{equation*}
Since there are more than ${m / 2 \choose k}$ sets $X$ such that $|X \rsum X| = m$, the right-hand side of \eqref{eq:the-sum} is at least
\begin{equation*}
  \sum_{m = 2k - 1}^{k^{1 + \beta}} {m / 2 \choose k} (1 - p)^m \ge {2k \choose k} \exp(-5 k p) \to \infty,
\end{equation*}
as $n \to \infty$, whenever $p = o(1)$ and $k \to \infty$.
The conclusion is that any approach that uses the union bound over all sets with small doubling cannot give the optimal upper bound on the independence number for $p$ smaller than some explicit constant, like $1/5$.

As we mentioned in the introduction, one of the key new ideas in this paper is to show that there exists a family $\cF$ of fingerprints such that it suffices to consider only the events $\{F \rsum F \subseteq \complement{A_p}\}_{F \in \cF}$ instead of the collection $\{X \rsum X \subseteq \complement{A_p}\}_{X \in \cX_1}$.
Trivially, $\cF = \cX_1$ is such a family, so, to make this strategy work and improve upon taking the union bound over all $X$, we must choose $\cF$ in a more clever way.

The first property of these fingerprints that we require is that $|\cF|$ is sufficiently small, where small vaguely means that a union bound over all $F \in \cF$ works.
One direct way to achieve that is picking each fingerprint $F$ to be small, but there is a subtle trade-off between the size of $F$ and the upper bound on the probability term $(1 - p)^{|F \rsum F|}$.
To circumvent this issue, we can use the fact that we have a bound (even if a polynomially large one) on the doubling of each $X \in \cX_1$.
In this setting, \Freiman's theorem~\cite{Fre59} says that $X$ is contained in a generalised arithmetic progression $P$ such that both the size $s$ and dimension $d$ of $P$ are small.
Recall that a $d$-dimensional generalised arithmetic progression is a set of the form
\begin{equation*}
  \bigg\{ a_0 + \sum_{i = 1}^d w_i a_i : w_i \in \Z, \, 0 \le w_i < \ell_i \bigg\}
\end{equation*}
for some differences $a_0, \ldots, a_d \in \Zmodn$ and side-lengths $\ell_1, \ldots, \ell_d \in \N$, and $P$ is proper when every possible choice of $\{w_1, \ldots, w_d\}$ corresponds to a distinct element of $P$.
Therefore, instead of choosing $F$ directly from $\Zmodn$, we first choose $P$ and then select $F$ inside $P$.

Now we can state the first requirement for $F \in \cF$ in more detail.
Let $P$ be the generalised arithmetic progression given by \Freiman's theorem for $X$.
Then, $F$ should be small enough compared to $|F \rsum F|$ for a union bound over choices of $F \subseteq P$ to work:
\begin{equation}\label{eq:dealt-with}
  \pr \big(\exists \, F \subseteq P : F \rsum F \subseteq \complement{A_p} \big) \le {s \choose |F|} (1 - p)^{|F \rsum F|} \to 0
\end{equation}
as $n \to \infty$.
To get to this point, however, we must also choose this generalised arithmetic progression in a previous union bound.
Hence, our second requirement for each $F$ is that $F \rsum F$ is sufficiently large to pay for the number of choices for the generalised arithmetic progression $P$.

As we can count generalised arithmetic progressions with dimension $d$ and size $s$ by choosing the $a_0, a_1, \ldots, a_d$ and $l_1, \ldots, l_d$, there are at most $(n s)^{d + 1}$ of them.
Temporarily ignoring the number of choices for the fingerprint inside the progression (which we already dealt with in \eqref{eq:dealt-with}), this amounts to requiring that $F$ satisfies
\begin{equation*}
  (n s)^{d + 1} (1 - p)^{|F \rsum F|} \to 0
\end{equation*}
as $n \to \infty$.
While these conditions may seem too strict -- small $F$ with large $F \rsum F$ for every $X$ -- this is exactly what we are able to do.

To state the actual \namecref{stmt:fingerprints}, we need to relate the dimension of the progression to some notion of dimension for $X$.
The notion that we use is the \Freiman dimension of $X$, but to state its definition, we must first define what are \Freiman homomorphisms and \Freiman isomorphisms.
A \Freiman homomorphism is a function ${\phi : X \to Y}$ such that for every $a_1, a_2, b_1, b_2 \in X$, $$a_1 + b_1 = a_2 + b_2 \quad \text{ implies } \quad \phi(a_1) + \phi(b_1) = \phi(a_2) + \phi(b_2). $$
A \Freiman isomorphism is a bijection $\phi$ such that both $\phi$ and its inverse $\phi^{-1}$ are \Freiman homomorphisms.
The \Freiman dimension of $X$, $\dimF(X)$, is then defined to be the largest $d \in \N$ for which there is a full rank subset of $\Z^d$ that is \Freiman isomorphic to $X$.

It is furthermore useful to define the robustness of the \Freiman dimension of $X$: we say that $X$ has $\eps$-robust \Freiman dimension $d$ if $\dimF(X) \ge d$ and there is no $X' \subseteq X$ such that $|X'| \ge (1 - \eps)|X|$ and $\dimF(X') < d$.
In words, this means that the \Freiman dimension of $X$ is at least $d$ even if we remove an $\eps$ proportion of its elements.

With these definitions, we can state our fingerprint \namecref{stmt:fingerprints}:

\begin{restatable}{thm}{fingerprints}\label{stmt:fingerprints}
  Let $n$ be a large enough prime and let $k, d \in \N$.
  For every $0 < \gamma < 1/2$, there exists $C = C(\gamma) > 0$ such that the following holds for all $m \ge (d + 1)k/2$ and $C/k < \eps < \gamma$.
  For each $d$-dimensional generalised arithmetic progression $P \subseteq \Z_n$, there exists a collection $\cF = \cF_{k, m, \eps}(P)$ of subsets of $P$ satisfying:
  \begin{enumerate}[(a)]
    \item For every $F \in \cF$, we have
    \begin{equation}\label{eq:fingerprint-reqs}
      |F| \le C \eps^{-1} \sqrt{ m \log m } \quad \text{ and } \quad |F \rsum F| \geq \frac{(1 - \gamma) (d + 1) k}{2}.
    \end{equation}
    \item For all $X \in \binom{P}{k}$ with $|X \rsum X| \leq m$ and $\eps$-robust \Freiman dimension $d$, there exists $F \in \cF$ such that $F \subseteq X$. \label{item:fingerprint-is-contained-in-X}
  \end{enumerate}
\end{restatable}

We will deduce this \namecref{stmt:fingerprints} from the following supersaturation result.

\begin{restatable}{thm}{supsat}\label{stmt:supsat}
  For every $0 < \gamma < 1$, there exists a constant ${c = c(\gamma) > 0}$ such that, for every sufficiently large set $X \subseteq \Zmodn$, every $d \in \N$ and every $0 < \eps < \gamma$, the following holds.
  If $X$ has $\eps$-robust \Freiman dimension $d$ and $Y \subseteq X + X$ satisfies
  \begin{equation}\label{eq:pairs}
    \big| \big\{(x_1, x_2) \in X^2 : x_1 + x_2 \not \in Y \big\} \big| \le c \eps |X|^2,
  \end{equation}
  then $|Y| \ge (1 - \gamma)(d + 1)|X|/2$.
\end{restatable}

In words, whenever $Y \approx X + X$ in the sense of \eqref{eq:pairs}, then it also (almost) satisfies the lower bound given by \Freiman's lemma (\Cref{stmt:freimans-lemma}), up to a factor of 1/2.
Once we have this \namecref{stmt:supsat}, the proof of \Cref{stmt:fingerprints} is substantially simpler than usual for a similarly strong container theorem.
In fact, it is also the only part of the proof that we could omit by using the container theorem for sumsets~\cite[Theorem 4.2]{Cam20}.
The self-contained proof is so simple that we include it here in the overview to motivate the usefulness of \Cref{stmt:supsat}.

\begin{proof}[Proof of \Cref{stmt:fingerprints} assuming \Cref{stmt:supsat}]
  Our aim is to show that, for every $X \in \binom{P}{k}$ with $|X \rsum X| \le m$ and $\eps$-robust \Freiman dimension $d$, there exists a subset $F \subseteq X$ such that \eqref{eq:fingerprint-reqs} holds.
  We will show that a $q$-random subset $F$ of $X$ satisfies the first inequality by a suitable choice of $q$, and the second one via an application of \Cref{stmt:supsat} with $Y = F \rsum F$.

  In order to apply \Cref{stmt:supsat} to $Y = F \rsum F$, we need to show that $Y$ satisfies \eqref{eq:pairs}.
  To this end, we will first prove that a random choice of $F$ makes it unlikely for $F \rsum F$ to miss any $y$ such that
  \begin{equation}\label{eq:popular}
    \rho_{X \rsum X}(y) \ge \frac{2 \eps k^2}{C m},
  \end{equation}
  where $$\rho_{X \rsum X}(y) = \big| \big \{(x_1, x_2) \in X^2 : x_1 \neq x_2,\, x_1 + x_2 = y \big \} \big|$$ is the number of pairs of distinct elements of $X$ that sum to a given $y \in X \rsum X$.
  Hereon, we will refer to the $y$ satisfying \eqref{eq:popular} as the popular elements of $X \rsum X$.
  We will take a $q$-random subset where either
  \begin{equation}\label{eq:def-q}
    q = \frac{C \sqrt{ m \log m }}{2 \eps k}
  \end{equation}
  if the right-hand side is at most 1, and $q = 1$ otherwise (a trivial case which we will ignore).

  Notice that we can upper bound the number of missing pairs by:
  \begin{equation}\label{eq:missing-pairs}
    \big| \{(x_1, x_2) \in X^2 : x_1 + x_2 \not \in F \rsum F\}\big| \le |X| + \sum_{y \in (X \rsum X) \setminus (F \rsum F)} \rho_{X \rsum X}(y)
  \end{equation}
  where the term $|X|$ comes from the pairs $(x, x)$ for $x \in X$.
  Once we have proved that $F \rsum F$ contains all popular $y \in X \rsum X$, we will have an upper bound on $\rho_{X \rsum X}(y)$ for every $y \in (X \rsum X) \setminus (F \rsum F)$.
  Inserting this into \eqref{eq:missing-pairs}, we deduce that the number of missing pairs is at most
  \begin{equation}\label{eq:missing-pairs-for-supsat}
    |X| + \sum_{y \in (X \rsum X) \setminus (F \rsum F)} \rho_{X \rsum X}(y) < |X| + \frac{2 \eps k^2}{C m}|X \rsum X| \le c \eps |X|^2,
  \end{equation}
  if we take $C \ge 3/c$, where $c = c(\gamma)$ is the constant in \Cref{stmt:supsat}, since $|X| = k$, $|X \rsum X| \le m$ and $\eps k > C$.
  The upper bound in \eqref{eq:missing-pairs-for-supsat} is what we need to apply \Cref{stmt:supsat}; doing so gives
  \begin{equation*}
    |F \rsum F| \ge \frac{(1 - \gamma)(\dimF(X) + 1)k}{2},
  \end{equation*}
  from which we can use our assumption that $\dimF(X) \ge d$ to complete the proof.

  It therefore only remains to prove our claim that with positive probability $F$ contains all popular elements of $X \rsum X$ while also being sufficiently small.
  Notice that our choice of $F$ as a $q$-random subset of $X$ implies, for each $y \in X \rsum X$, that
  \begin{equation}\label{eq:prob-missing}
    \pr(y \not \in F \rsum F) = (1 - q^2)^{\rho_{X \rsum X}(y)/2}
  \end{equation}
  because (i) each pair $(x_1, x_2) \in X^2$ that satisfies $x_1 + x_2 = y$ and $x_1 \neq x_2$ is counted twice in $\rho_{X \rsum X}(y)$ and (ii) the probability that such a pair is chosen in $F$ is $q^2$.

  Take $B_F$ to be the random variable counting the number of $y \in X \rsum X$ such that
  \begin{equation*}
    \rho_{X \rsum X}(y) \ge \frac{2 \eps k^2}{C m} \quad \text{ and } \quad y \not \in F \rsum F.
  \end{equation*}
  By linearity of expectation and \eqref{eq:prob-missing}, we have
  \begin{align*}
    \E[B_F] = \sum_{\substack{y \in X \rsum X \\ \rho_{X \rsum X}(y) \ge 2 \eps k^2 / C m}} \pr(y \not \in F \rsum F) \le m (1 - q^2)^{\eps k^2 / C m} \le m \exp\left(- \frac{\eps q^2 k^2}{C m}\right),
  \end{align*}
  where we used $|X \rsum X| \le m$ to bound the number of terms in the sum.

  Since $\eps q^2 k^2/(C m) = (C/4 \eps) \log m$ by our choice~\eqref{eq:def-q} of $q$, it follows by Markov's inequality that
  \begin{equation*}
     \pr(B_F > 0) \le m^{-1},
  \end{equation*}
  where we also used that $\eps < 1$ and $C \ge 8$.
  Using Chernoff's inequality to bound the probability that $|F| > 2 q k$, we deduce that
  \begin{equation*}
    \pr(|F| > 2 q k) + \pr(B_F > 0) < 1,
  \end{equation*}
  which proves that there exists a fingerprint $F \subseteq X$ satisfying \eqref{eq:fingerprint-reqs}, as required.
\end{proof}

Before moving on with the overview, let us briefly discuss the robustness property in \Cref{stmt:supsat}.
This condition may at first seem unnatural, but the following simple construction shows that some form of robustness is necessary: take $X$ to be the union of $d - 1$ random points with a progression $P$, and define $Y = P + P$.
We have simultaneously \whp that $\dimF(X) = d$, $|Y| \approx 2|X|$ and the sum of almost all pairs of elements of $X$ are in $Y$.

With \Cref{stmt:fingerprints} in hand, we can now check that for the family $\cF$ of all sets satisfying \eqref{eq:fingerprint-reqs}, our requirements for the fingerprints $F$ are satisfied.
Recall that what we need from the size of the sumset is
\begin{equation*}
  (ns)^{d + 1} (1 - p)^{|F \rsum F|} \to 0,
\end{equation*}
where $s = |P|$ and $d = \dim(P)$.
Modern formulations of \Freiman's theorem tells us that we can take $s \le \exp(\sigma^{1 + o(1)})k$, which in our range of $\sigma$ and $k$ corresponds to $n^{o(1)}$.
However, we can only apply \Cref{stmt:fingerprints} to sets $X \subseteq P$ such that $\dimF(X) \ge \dim(P)$ (ignoring the robustness for the moment).
Standard formulations of \Freiman's theorem only guarantee that $\dim(P)$ is at most $\sigma[X]$, which would not be good enough for us.

Fortunately, \citet{Cha02} proved a version of \Freiman's theorem that guarantees that $\dim(P)$ is at most $\dimF(X)$, at the cost of a weaker bound on the size of $P$ as compared to the more recent results of \citet{Sch11} and \citet{San13}.
The impact of the suboptimal size of the progression is a slight reduction in the range of $p$ for which our proof works.

\begin{thm}[\citet{Cha02}, see~{\cite[Proposition 1.3]{GT06}}]\label{stmt:chang-freiman-ruzsa}
  There exists $C' > 0$ such that for all finite subsets $X \subseteq \Z$ with $|X| \ge 2$ and $\sigma[X] \le \sigma$, there is a $d$-dimensional generalised arithmetic progression $P$ such that $X \subseteq P$,
  \begin{equation*}
    |P| \le \exp\hspace{-2pt}\big(C' \sigma^2 (\log \sigma)^3 \big) |X|
  \end{equation*}
  and ${d \le \dimF(X)}$.
\end{thm}

Now, we can use the lower bound on $|F \rsum F|$ given by \Cref{stmt:fingerprints} to obtain
\begin{equation*}
  (1 - p)^{|F \rsum F|} \le \exp\hspace{-2pt}\big( -(1 - \gamma) (d + 1) k p/2 \big),
\end{equation*}
which, choosing $k = (2 + o(1)) \log_{1 / (1 - p)} n$ (a little larger than $(2 \log n)/p$), is at most
\begin{equation}\label{eq:probability-of-fingerprint}
  \exp\big( - (1 + \gamma) (d + 1) \log n \big).
\end{equation}
Replacing this back in the previous equation, and recalling that $s = n^{o(1)}$, thus yields
\begin{equation*}
  (ns)^{d + 1} (1 - p)^{|F \rsum F|} \le n^{-\gamma} \to 0.
\end{equation*}

The attentive reader may have noticed that \citeauthor{Cha02}'s \namecref{stmt:chang-freiman-ruzsa} is stated for subsets of $\Z$ instead of $\Zmodn$.
To use it for subsets of $\Zmodn$, we will use instead a version of Green--Ruzsa's theorem (\Freiman's theorem for Abelian groups) due to \citet{CS13}.
It does not bound $\dim(P) \le \dimF(X)$ directly, though, but it yields a proper progression, which we can combine with a \namecref{stmt:discrete-john} by \citet{TV08} to obtain what we need (see \Cref{stmt:green-ruzsa-with-dim(P)<=dimF(X)}).

Moreover, it is not true that every $X$ has $\eps$-robust \Freiman dimension equal to $\dimF(X)$.
This is not a problem, however, since it is straightforward to prove (see \Cref{stmt:robust-subset}) that every set $X$ has a large enough subset $X'$ with $\eps$-robust \Freiman dimension $d' \in \N$.

Finally, let us check that the size of $F$ given to us by \Cref{stmt:fingerprints} is compatible with the range of $p$ in \Cref{stmt:main-result}.
To do so, we need to show that, as $n \to \infty$,
\begin{equation*}
  {s \choose |F|} (1 - p)^{|F \rsum F|} \to 0.
\end{equation*}
As we already know from \eqref{eq:probability-of-fingerprint} that the second term is at most $n^{-d}$, we need
\begin{equation}\label{eq:choices-for-fingerprint}
  {s \choose |F|} \le s^{|F|} \le \exp\big( C \eps^{-1} ( m \log m )^{1/2} \log s \big) = n^{o(1)}.
\end{equation}
The second inequality in \eqref{eq:choices-for-fingerprint} follows from \eqref{eq:fingerprint-reqs} in \Cref{stmt:fingerprints}, while the third is a consequence of our choice of $k$ and the bound on $s$ given by \Cref{stmt:chang-freiman-ruzsa}.
Indeed, we have
\begin{equation}\label{eq:overview-bound}
  C \eps^{-1} ( m \log m )^{1/2} \log s \ll k^{3/4} = o(\log n),
\end{equation}
because $m \le k^{1 + \beta}$ for $X \in \cX_1$, which implies that $\log s \le k^{3 \beta}$, and we can take $\eps = k^{-2\beta}$ and $C$ to be a constant.
Our choice of $\beta = 1/40$ is therefore more than sufficient to prove \eqref{eq:overview-bound}.

At this point, one might ask why we decided to keep track of the constant $C$ up to the final computation.
Note that it is crucial that the value of $C$ does not increase too quickly with $d$ growing, since otherwise \eqref{eq:overview-bound} would not hold for large $d$.
Recall that in the proof of \Cref{stmt:fingerprints}, we took $C \approx c^{-1}$.
The constant $c$ is essentially given to us by our supersaturation result and its value is tied to how many translates we need in our approximate bound for \Freiman's lemma (\Cref{stmt:freiman-lemma-via-few-translates}).

To see where the dependency of $c$ on the number of translates comes from, consider the contrapositive of \Cref{stmt:supsat}: if the set $Y$ is small, then it misses many pairs $(x_1, x_2) \in X^2$.
By \Cref{stmt:freiman-lemma-via-few-translates}, we can find a small $T \subseteq X$ such that
\begin{equation*}
  |X + T| - |Y| \ge \gamma (d + 1)|X|.
\end{equation*}
The pigeonhole principle then ensures us that there is some $x \in T$ satisfying
\begin{equation*}
  \big| (X + x) \setminus Y \big| \ge \frac{\gamma (d + 1)|X|}{|T|} = c |X|
\end{equation*}
for $c = \gamma (d + 1)/|T|$.
Now, if we add the $c |X|$ pairs determined by $(X + x) \setminus Y$ to our collection of missing pairs $(x_1, x_2) \in X^2$ such that $x_1 + x_2 \not \in Y$, remove $x$ from $X$ and repeat this procedure $t$ times, we would have at least $c |X| t$ missing pairs in total.
Recalling that $X$ has $\eps$-robust \Freiman dimension $d$, we can take $t = \eps |X|$ while maintaining $\dimF(X) \ge d$, and hence obtain $c \eps |X|^2$ missing pairs, as required.

The above sketch proof of \Cref{stmt:supsat} shows that to prove our supersaturation result with $c$ being an absolute constant, we really need the size of $T$ to be a linear function of $d$, as in \Cref{stmt:freiman-lemma-via-few-translates}.
In fact, for \eqref{eq:overview-bound}, a bound of the form $d^{O(1)}$ would suffice.
Unfortunately, the \namecref{stmt:jing-mudgal-freiman} of \citeauthor{JM23+}~(\Cref{stmt:jing-mudgal-freiman}) gives a super-polynomial dependency on $d$.
Nevertheless, we use their result to prove \Cref{stmt:supsat} in the range $d = O(1)$.

The only missing part in this overview is a proof of \Cref{stmt:freiman-lemma-via-few-translates} itself.
Instead of sketching it, we complete the picture with the (short) proof of its weaker version in the next section.

\section{A simple proof of a weaker \Freiman's lemma via few translates}\label{sec:freiman-via-few-translates}

This \namecref{sec:freiman-via-few-translates} is dedicated to proving a weaker form of \Cref{stmt:freiman-lemma-via-few-translates}.
Although it is not strong enough to prove the upper bound in our main \namecref{stmt:main-result}, it does imply a weaker version where the constant $2$ in \eqref{eq:main-result} is replaced by a $6$.
The deduction of this weaker result is the same as that of \Cref{stmt:main-result} (see \Cref{sec:supsat,sec:main-theorem}) simply replacing \Cref{stmt:freiman-lemma-via-few-translates} by \Cref{stmt:weak-freiman}.

\begin{thm}\label{stmt:weak-freiman}
  Let $d, r \in \N$ and let $X \subseteq \R^d$ be a finite set with $\rank(X) \ge r$.
  There exists a set $T \subseteq X$ such that $|T| \le r/2 + 1$ and
  \begin{equation*}
    |X + T| \ge \frac{r|X|}{6}.
  \end{equation*}
\end{thm}

The idea of the proof is to add elements of $X$ to $T$ one by one, picking in each step a new translate that adds a substantial number of new elements to the sumset.
This suggests a greedy argument, picking $x \in X \setminus T$ that increases the size of the sumset the most.
However, it is not clear how to show that we can make substantial progress for a sufficient number of steps.
Previous arguments, such as the one by \citet*[Theorem 1.1]{FLPZ24}, show that if progress stops, then $X$ must have small doubling; as we must handle polynomially large $\sigma[X]$, such strategies do not work in our case.
Instead, we adopt a variation of this idea that incorporates the geometry of the set, allowing us to reach conclusions that do not rely on the doubling of the set and depend only on its rank.

Roughly speaking, in the proof of \Cref{stmt:greedy-phase} below, we will show that we can add a new element to $T$ so that it increases the size of $X + T$ by a factor proportional to $|X \setminus \Span(T)|$.
Notice that, as long as $|T| < \rank(X)$, we can take a non-trivial step.

Now, if we have enough steps that add $|X|/3$ elements to the sumset, then after $r/2$ steps we will have both the bound for the sumset and for $|T|$ in \Cref{stmt:weak-freiman}.
Otherwise, as we will show that every step makes at least $|X \setminus \Span(T)|/2$ ``progress'', it follows that, after the last ``good'' step, we must have $|X \cap \Span(T)| \ge |X|/3$.
In this scenario, we define
\begin{equation*}
  X^* = X \cap \Span(T) \qquad \text{ and } \qquad X' = X \setminus \Span(T),
\end{equation*}
and observe that
\begin{equation}\label{eq:rank-X'}
  \rank(X') \ge \rank(X) - \rank(T) \ge r - r/2 = r/2,
\end{equation}
since $|T| \le r/2$.
At this point, we now discard our old $T$ and greedily choose elements of a new set of translates $Z \subseteq X'$ each increasing the rank of $X^* \cup Z$ by one.
Each new element that we add yields a disjoint, translated copy of $X^*$ in the sumset.
We then obtain
\begin{equation*}
  |X + Z| \ge |X^* + Z| = \frac{r|X^*|}{2} \ge \frac{r|X|}{6},
\end{equation*}
because, by \eqref{eq:rank-X'}, we can greedily add $r/2$ elements to $Z$, and $|X^*| \ge |X|/3$.

We proceed by formalizing the notion that either the greedy argument suffices, or a significant part of $X$ lies in the span of the elements already in $T$.
The version we state below is more general than we need to prove \Cref{stmt:weak-freiman} because we will reuse it when proving \Cref{stmt:freiman-lemma-via-few-translates}.

\begin{prop}\label{stmt:greedy-phase}
  Let $d, r \in \N$, let $\gamma > 0$ and let $X \subseteq \R^d$ be a finite set with $\rank(X) \ge r$.
  If $T \subseteq \R^d$ satisfies $0 \in T$, $\rank(T) < r$ and $|X \cap \Span(T)| \le \gamma |X|$, then there exists an element $x_* \in X$ such that
  \begin{equation*}
    \big|X + \big(T \cup \{x_*\}\big)\big| \ge |X + T| + \frac{(1 - \gamma)|X|}{2}.
  \end{equation*}
\end{prop}

The key idea here is to find a co-dimension 1 hyperplane $\mathcal{H}$ which intersects $X$ exactly in $X \cap \Span(T)$.
We can find such a hyperplane because $X$ is finite and $\rank(T) < \rank(X)$.
Note that the two open half-spaces defined by $\mathcal{H}$ induce a partition $X' = X_+ \cup X_-$.
Without loss of generality, we assume that $|X_+| \ge |X_-|$, and our goal is now to add, for each point in $X_+$, a new element to the sumset.
To achieve this, we let $u$ be a normal vector of $\mathcal{H}$, choose $y_+ \in X_+$ to maximise $\inner{y_+, u}$, and observe that $y_+ + X_+$ is disjoint from $X + T$.

\begin{proof}
  First, choose a vector $u \in \R^d$ to be the normal of the hyperplane $\mathcal{H}$ discussed above.
  It should satisfy the following properties
  \begin{enumerate}
    \item $u \neq 0$, \label{item:non-null-u}
    \item ${\inner{z, u} = 0}$, for all $z \in \Span(T)$, and \label{item:contain-span-of-translates}
    \item $\inner{z, u} \neq 0$, for all $z \in X \setminus \Span(T)$. \label{item:non-null-remainder-of-x}
  \end{enumerate}
  There is a $u$ satisfying \cref{item:non-null-u,item:contain-span-of-translates} since $0 \in T$ and $\rank(T) < r$.
  We can furthermore find a $u$ for which \cref{item:non-null-remainder-of-x} also holds because there are only finitely many elements in $X \setminus \Span(T)$, as $X$ itself is finite.

  We partition $X = X_+ \cup X_* \cup X_-$ according to the position of its elements relative to the hyperplane defined by $\inner{x, u} = 0$, \ie
  \begin{align*}
    X_+ = \{x \in X : \inner{x, u} > 0\}, \\
    X_* = \{x \in X : \inner{x, u} = 0\}, \\
    X_- = \{x \in X : \inner{x, u} < 0\}.
  \end{align*}
  Assume without loss of generality that $|X_+| \ge |X_-|$ by swapping the sign of $u$ if necessary, and take $y_+ \in X_+$ to be a maximiser of $f(y) = \inner{y, u}$.

  We claim that $X_+ + y_+$ is disjoint from $X + T$.
  In fact, if we let $a \in X_+, b \in X$ and $c \in T$, then:
  \begin{align*}
    \inner{a + y_+, u} & > \inner{y_+, u},     &  & \text{as } a \in X_+,              \\
                       & \ge \inner{b, u},     &  & \text{by the maximality of } y_+,  \\
                       & = \inner{b + c, u}, &  & \text{as we chose } u \text{ with } c \perp u.
  \end{align*}
  Therefore, if $|X_+| \ge (1 - \gamma) |X|/2$, we can pick $x_* = y_+$ and prove the \namecref{stmt:greedy-phase}:
  \begin{align*}
    |X + (T \cup \{x_*\})| & \ge |X + T| + |X_+ + x_*| \\
                           & \ge |X + T| + \frac{(1 - \gamma) |X|}{2}.
  \end{align*}

  Otherwise, since we took $|X_+| \ge |X_-|$ and $X_+ \cup X_* \cup X_-$ partition $X$, we have
  \begin{equation}\label{eq:size-of-X*}
    |X_*| = |X| - |X_-| - |X_+| > \gamma |X|.
  \end{equation}
  By our choice of $u$ and the definition of $X_*$, we have $X_* = X \cap \Span(T)$, which, by \eqref{eq:size-of-X*}, contradicts our assumption that $|X \cap \Span(T)| \le \gamma |X|$.
\end{proof}

The 1-dimensional perspective we took in this proof suggests that instead of adding a single maximiser $y_+$ to $T$ at each step, we should pick both the maximiser $y_+$ and the minimiser $y_-$.
Unfortunately, if we add $\{y_+, y_-\}$ to $T$, then we could increase $\rank(T)$ by two instead of one, causing the greedy argument to run for only half as many steps.

We need the following simple lemma to handle the case when $|X \cap \Span(T)|$ is large.
\begin{lem}\label{stmt:disjoint-translates-with-X*}
  Let $d, r, s \in \N$, let $\gamma > 0$ and let $X \subseteq \R^d$ be a finite set with $\rank(X) \ge r$.
  If $X_* \subseteq X$ with $\rank(X_*) < s$, then there exists a set $Z \subseteq X \setminus X_*$ such that
  \begin{equation*}
    |X_* + Z| \ge (r - s)|X_*|
  \end{equation*}
  and $|Z| \le r - s$.
\end{lem}

\begin{proof}
  First, note that for any set $S \subseteq \R^d$ such that $\rank(S) < s$, we have $\rank(\Span(S)) \le s$.
  That is, taking the $\Span$ of a set that does not contain $0$ may increase its rank by 1.

  We will construct $Z = \{z_1, \ldots, z_{r - s}\} \subseteq X \setminus X_*$ one element at a time, also defining $$Z_i = \{z_1, \ldots, z_i\} \quad \text{ and } \quad W_i = \Span(X_* \cup Z_i)$$ for $i \in \{0, \ldots, r - s\}$.
  Notice that if $i < r - s$, then $\rank(W_i) \le s + i < r$, by our assumption on $X^*$ and the definition of $W_i$.
  Therefore, there exists $z_{i + 1} \in X \setminus W_i$, as $\rank(X) \ge r$.
  Since $X^* \subseteq W_i$, this implies that $X^* + z_{i + 1}$ and $W_i$ are disjoint, and hence that $X^* + z_{i + 1}$ is disjoint from $X^* + Z_i$.
  This readily implies the \namecref{stmt:disjoint-translates-with-X*} because
  \begin{equation*}
    |X_* + Z| = \sum_{i = 1}^{r - s} |X_* + z_i| = (r - s) |X_*|,
  \end{equation*}
  where in the first equality we repeatedly used that $X_* + z_{i + 1}$ and $X_* + Z_i$ are disjoint.
\end{proof}

The proof of \Cref{stmt:weak-freiman} now follows easily from \Cref{stmt:greedy-phase} and \Cref{stmt:disjoint-translates-with-X*}:

\begin{proof}[Proof of \Cref{stmt:weak-freiman}]
  We may assume that $0 \in X$; notice that this translation does not change $\rank(X)$.
  First, we construct sets $T_i = \{x_0, \ldots, x_i\} \subseteq X$ by choosing $x_0 = 0$ and $x_{i + 1}$ as given by \Cref{stmt:greedy-phase}.
  In details, let $t$ be the first index for which
  \begin{equation}\label{eq:greedy-condition-for-weaker-freiman}
    |X \cap \Span(T_t)| > \frac{|X|}{3}.
  \end{equation}
  Note that while \eqref{eq:greedy-condition-for-weaker-freiman} does not hold, we have $\rank(T_i) < \rank(X)$ since $T_i \subseteq X$.
  We may therefore apply \Cref{stmt:greedy-phase} to $T_i$ with $\gamma = 1/3$ to define $x_i$, which then implies that
  \begin{equation}\label{eq:greedy-phase-of-weaker-freiman}
    |X + T_i| \ge \frac{|T_i|\,|X|}{3} = \frac{(i + 1)|X|}{3}
  \end{equation}
  for all $i \le t$.
  Hence, if $t \ge r/2$, then, by \eqref{eq:greedy-phase-of-weaker-freiman}, we have
  \begin{equation*}
    |X + T_t| \ge \frac{(t + 1)|X|}{3} \ge \frac{r|X|}{6},
  \end{equation*}
  and taking $T = T_t$ concludes the proof.
  We may therefore assume that $t < r/2$.

  In this case, we want to apply \Cref{stmt:disjoint-translates-with-X*} with $X_* = X \cap \Span(T_t)$ and $s = r/2$.
  Note that, as $\rank(T_t) \le t$ and $0 \in T_t$, we have $$\rank(X_*) \le \rank(\Span(T_t)) \le t < r/2,$$ where in the last inequality we used our assumption that $t < r/2$.
  This application yields a set $Z \subseteq X \setminus X_*$ such that $|Z| \le r - s = r/2$ and
  \begin{align*}
    |X + Z| & \ge |X_* + Z|        & & \text{ because } X_* \subseteq X, \\
            & \ge \frac{r|X_*|}{2} & & \text{ by  \Cref{stmt:disjoint-translates-with-X*}}, \\
            & \ge \frac{r|X|}{6}   & & \text{ since } |X_*| \ge |X|/3 \text{ by \eqref{eq:greedy-condition-for-weaker-freiman}}.
  \end{align*}
  Taking $T = Z$ concludes the proof.
\end{proof}

\section{Improving the bound for \Freiman's lemma via few-translates}\label{sec:improved-freiman-via-few-translates}

To obtain the sharp upper bound for $\alpha(\Gamma_p)$, the bound $|X + T| \ge r|X|/6$ is not enough; we need at least $|X + T| \ge (1 - 5\gamma)(r + 1)|X| / 2$ for small $\gamma > 0$.
In this section, we describe the overall approach and state the intermediate results we require to improve the bound.
Once we have stated these results, we will show that assuming them is sufficient to prove \Cref{stmt:freiman-lemma-via-few-translates} -- this will motivate their statements, since neither they nor their proofs are straightforward.
We will prove that these results hold in subsequent sections.

To discuss the methods that we will use to improve the bound, it will be convenient to first make some definitions.

\begin{defi}
  For finite sets $X, W \subseteq \R^d$, we define
  \begin{equation*}
    Z(X, W) = \Pi_{W^\perp}(X),
  \end{equation*}
  where $\Pi_U(V)$ denotes the orthogonal projection of $V$ onto $U$.
  We also partition $\R^d$ into equivalence classes in that projection, denoting these\footnote{To avoid confusion between the equivalence class $[z]$ for $z \in \R^d$ and $[m] = \{1, \ldots, m\}$ for $m \in \N$, we avoid using the latter notation.} by
  \begin{equation*}
    [z]_W = \big\{x \in \R^d : \Pi_{W^\perp}(x) = z\big\},
  \end{equation*}
  and partition X into equivalence classes in the same way:
  \begin{equation*}
    \equivc{z}_{W, X} = [z]_W \cap X.
  \end{equation*}
\end{defi}

It will be convenient to omit the dependency of those definitions on $X$ and $W$ whenever these sets are clear from context, leaving us with the notation $Z, [z], \equivc{z}$.
We also refer to $[z]$ as a ``fibre'', and say that a fibre $[z]$ is ``empty'' if $z \not \in Z$.

The start of the proof follows the same idea as in \Cref{sec:freiman-via-few-translates}: to obtain $T_{i + 1}$ from $T_i$, at each step we add the element provided by \Cref{stmt:greedy-phase}.
We do this for $t$ steps, where $t$ is the first index for which ${|X \cap \Span(T_t)| > \gamma|X|}$.
If $t \ge r$, then, by \Cref{stmt:greedy-phase}, we have $|X + T_t| \ge (1 - \gamma)(t + 1)|X|/2$, and we are done.
Otherwise, we define
\begin{equation}\label{eq:T*-def}
  T^* = T_{t - 1} \qquad \text{ and } \qquad W = \Span(T_t),
\end{equation}
and we look into the set of non-empty fibres $Z = Z(X, W)$.
Notice that $W$ is neither $\R^d$ nor $\{0\}$ since $0 < t < r$.
The rest of the argument is divided into two different cases.

If there are many distinct non-empty fibres in $Z = \{z_1, \ldots, z_m\}$, then we use the following generalisation of \Cref{stmt:disjoint-translates-with-X*}.
Whereas for that previous result we needed one point per dimension, in \Cref{stmt:many-fibres-are-easier} we take one point $y_i \in \equivc{z_i}$ per non-empty fibre to be our set $T$, and show that such a choice yields disjoint translates $y_i + X_*$ for $X_* = X \cap W$.

\begin{prop}\label{stmt:many-fibres-are-easier}
  Let $d, r \in \N$, let $\gamma > 0$, and let $X \subseteq \R^d$ be a finite set with $\rank(X) \ge r$.
  If $W \subseteq \R^d$ is such that $|X \cap W| \ge \gamma |X|$ and $|Z| \ge (r + 1)/\gamma$, where $Z = Z(X, W)$, then there exists $T \subseteq X$ such that
  \begin{equation*}
    |X + T| \ge (r + 1)|X|
  \end{equation*}
  and $|T| = (r + 1)/\gamma$.
\end{prop}

\begin{proof}
  Define $m = (r + 1)/\gamma$. Take $\{z_1, \ldots, z_m\} \subseteq Z$ and, for each $z_i$, pick some arbitrary $y_i \in \equivc{z_i}$.
  Note that $X \cap W \subseteq \equivc{0}$.
  As we have chosen $y_i \in [z_i]$, we have $\equivc{0} + y_i \subseteq [z_i].$
  Doing the same with $i \neq j$ shows that
  \begin{equation}\label{eq:y_i-y_j-disjoint}
    \big(\equivc{0} + y_i\big) \cap \big(\equivc{0} + y_j\big) \subseteq [z_i] \cap [z_j] = \emptyset.
  \end{equation}

  We can therefore conclude that taking $T = \{y_1, \ldots, y_m\} \subseteq X$ satisfies
  \begin{align*}
    |X + T| & \ge \big|\equivc{0} + T\big| & & \text{as } \equivc{0} \subseteq X \text{ by definition}, \\
            & \ge \sum_{i = 1}^m \big| \equivc{0} + y_i \big| & & \text{by \eqref{eq:y_i-y_j-disjoint}}, \\
            & \ge m \gamma |X| & & \text{as } \big| \equivc{0} \big| \ge \gamma|X|, \text{ by assumption,} \\[0.3ex]
            & = (r + 1)|X|,
  \end{align*}
  as required.
\end{proof}

The case when there are few non-empty fibres, \ie $Z = Z(X, W)$ is small, is more complex.
Letting $r_W = \rank(W)$ and recalling \eqref{eq:T*-def}, we have
\begin{equation*}
  |X + T^*| \ge (1 - \gamma)\frac{(r_W + 1)|X|}{2}.
\end{equation*}
Therefore, to obtain a final set of translates $T$ such that $|X + T| \ge (1 - 5\gamma)(r + 1)|X|/2$, we need to find a set $T'$ that roughly satisfies
\begin{equation}\label{eq:optimistic-freiman-lemma}
  |X + T'| \ge \frac{(r - r_W)|X|}{2}.
\end{equation}
To combine $T^*$ and $T'$, though, we must take care not to count overlaps between the sumsets $X + T'$ and $X + T^*$ more than once.

The content of the next \namecref{stmt:phase-3} shows that choosing $T'$ as discussed above is possible.
It says that we can choose a $T'$ which almost attains \eqref{eq:optimistic-freiman-lemma} while avoiding not only $X + T^*$, but the whole of $X + \equivc{0}$ -- recall that $T^* \subseteq W \cap X \subseteq \equivc{0}$.
\begin{prop}\label{stmt:phase-3}
  Let $d, r, r_W \in \N$, $\eta > 0$, and let $X \subseteq \R^d$ be a finite set with $\rank(X) \ge r$.
  Let also $W \subseteq \R^d$ be a subspace with dimension $r_W$, and $Z = Z(X, W)$.
  If $\big|\equivc{z}\big| \leq \eta |X|$ for every $z \in Z \setminus \{0\}$, then there exists $T' \subseteq X$ such that
  \begin{equation}\label{eq:phase-3}
    \Big|\big(X + T'\big) \setminus \big(X + \equivc{0}\big)\Big| \ge \frac{r - r_W}{2} \Big(\big|X\big| - \big|\equivc{0}\big|\Big) - \eta|Z||X|
  \end{equation}
  and $|T'| \le |Z|$.
\end{prop}

To gain some intuition for why \Cref{stmt:phase-3} is true, consider the following.
Applying \Freiman's lemma in the projected world $W^\perp$, if we could obtain a lower bound depending on $|X|$, instead of $|Z|$ -- which is what we actually get -- then we would be done.
Notice, however, that this naive application considers every non-empty fibre $\equivc{z}$ as a single element to avoid repeated counts.
That is, for each $z_1 + z_2 \in Z + Z$, this approach counts only $x + y$ for a single choice of $x \in \equivc{z_1}$ and $y \in T' \cap \equivc{z_2}$.

In the proof of \Cref{stmt:phase-3}, we will show that we can instead consider $\equivc{z_1} + y$ and not overcount.
To achieve that, we assign to each $z \in Z$ a weight that encodes the size of the corresponding fibre $\equivc{z}$, and incorporate this weight into the proof of \Freiman's lemma in the projected world.
Although considering $\equivc{z_1} + y$ is still not enough to replace $|Z|$ with $|X|$ in the bound, we can crucially choose $z_1, z_2 \in Z$ whichever way we want, as long as $y \in \equivc{z_2}$.
We therefore choose the representation that maximises $\big|\equivc{z_1}\big|$ and show that this modification is enough to obtain the bound in \Cref{stmt:phase-3}.

The above discussion overlooks the removal of the zero fibre from the sumset, \ie it gives a lower bound for $|X + T'|$ instead of one for $\big|(X + T') \setminus (X + \equivc{0})\big|$.
To incorporate it, we must take into account our choice maximizing $\big|\equivc{z_1}\big|$, which imposes the (technical) requirement that every non-zero fibre has size at most $\eta |X|$.
As this is not always the case, in order to apply \Cref{stmt:phase-3}, we will change $W$ slightly to make it so: we will ``add'' to $W$ the large non-zero fibres until we have $\big| \equivc{z} \big| \le \eta |X|$ for all remaining $z \in Z$.
As we will do this economically, using this new $W$ will not hurt our bound too much.

Nevertheless, combining $T'$ and $T^*$ had another unanticipated cost: it incurred a negative $(r - r_W) \big| \equivc{0} \big| / 2$ term in the bound of \Cref{stmt:phase-3}.
The following \namecref{stmt:phase-2} shows that a simple change suffices to offset this loss: we add two extra points from each non-empty fibre to our final $T$.

\begin{prop}\label{stmt:phase-2}
  Let $d \in \N$, and let $X \subseteq \R^d$ be a finite set.
  For every $W \subseteq \R^d$ and $u^* \in \R^d$, there is $ T'' \subseteq X$ satisfying $|T''| \le 2|Z|$ and
  \begin{equation*}
    \Big|\big(\equivc{0} + T''\big) \setminus \big(X + \equivc{0}_*\big)\Big| \ge |Z| \Big(\big|\equivc{0}\big| - \big|\equivc{0}_*\big|\Big),
  \end{equation*}
  where $Z = Z(X, W)$ and $\equivc{0}_* = \{x \in \equivc{0} : \inner{x, u^*} = 0\}$.
\end{prop}

It is not immediate that \Cref{stmt:phase-2} is really enough to make up for what we lost in \Cref{stmt:phase-3}, as we are removing $X + \equivc{0}_*$ instead of $X + T^*$, and $\equivc{0}_*$ is determined by $W$ and $u^*$.
What we need is that, for our choice of $u^*$, both $\big| \equivc{0}_* \big| \le \gamma |X|$ and $T_{t - 1} = T^* \subseteq \equivc{0}_*$.
This condition, combined with the following \namecref{stmt:gluing-phase-2-and-phase-3}, is enough to prove \Cref{stmt:freiman-lemma-via-few-translates}.
The proof of the \namecref{stmt:gluing-phase-2-and-phase-3} is a simple (but slightly tedious) manipulation of set relations.

\begin{obs}\label{stmt:gluing-phase-2-and-phase-3}
  For every $T', T'' \subseteq X$, we have
  \begin{equation}\label{eq:gluing-inequality}
    \big|\big(X + \hat{T}\big) \setminus \big(X + \equivc{0}_*\big)\big| \ge \big|\big(X + T'\big) \setminus \big(X + \equivc{0}\big)\big| + \big|\big(\equivc{0} + T''\big) \setminus \big(X + \equivc{0}_*\big)\big|
  \end{equation}
  where $\hat{T} = T' \cup T''$.
\end{obs}

\begin{proof}
  Note first that $\equivc{0} + T''$ is a subset of $X + \hat{T}$: this follows easily from
  \begin{equation*}
    \equivc{0} \subseteq X \quad \text{ and } \quad T'' \subseteq \hat{T}.
  \end{equation*}
  We therefore have
  \begin{equation*}
    \big|X + \hat{T}\big| \ge |S| + \big|\equivc{0} + T''\big|
  \end{equation*}
  where $S = \big(X + T'\big) \setminus \big(\equivc{0} + T''\big)$, and moreover,
  \begin{equation}\label{eq:disjoint-union-with-removal}
    \big|(X + \hat{T}) \setminus (X + \equivc{0}_*)\big| \ge \big|S \setminus (X + \equivc{0}_*)\big| + \big|(\equivc{0} + T'') \setminus (X + \equivc{0}_*)\big|.
  \end{equation}
  As the last term in \eqref{eq:disjoint-union-with-removal} already matches what appears in \eqref{eq:gluing-inequality}, the proof is reduced to showing that
  \begin{equation*}
    \big|S \setminus (X + \equivc{0}_*)\big| \ge \big|\big(X + T'\big) \setminus \big(X + \equivc{0}\big)\big|,
  \end{equation*}
  where, recall,
  \begin{equation*}
    S \setminus (X + \equivc{0}_*) = \big(X + T'\big) \setminus \big((\equivc{0} + T'') \cup (X + \equivc{0}_*)\big).
  \end{equation*}
  It is enough, therefore, to prove that
  \begin{equation*}
    \equivc{0} + T'' \subseteq \equivc{0} + X \quad \text{ and } \quad X + \equivc{0}_* \subseteq X + \equivc{0}.
  \end{equation*}
  Both inclusions are trivial, as $T'' \subseteq X$ by assumption, and $\equivc{0}_* \subseteq \equivc{0}$ by definition.
\end{proof}

We are now ready to put the pieces together and prove \Cref{stmt:freiman-lemma-via-few-translates}.

\begin{proof}[Proof of \Cref{stmt:freiman-lemma-via-few-translates} assuming \Cref{stmt:phase-2,stmt:phase-3}]
  By translating if necessary, we may assume that $0 \in X$; note that this does not change $\rank(X)$.
  As in the proof of the weaker version, \Cref{stmt:weak-freiman}, we start by defining $T_0 = \{0\}$: observe that $\rank(T_0) = 0$ and
  \begin{equation}\label{eq:first-step-of-greedy}
    |X + 0| = |X|.
  \end{equation}

  The next steps consist of taking $T_{i + 1} = T_i \cup \{x_{i + 1}\}$, where $x_{i + 1} \in X$ is defined to be the $x^*$ given by \Cref{stmt:greedy-phase} applied to $T_i$.
  We can do so as long as $|X \cap \Span(T_i)| < \gamma|X|$ and $i < r$, because $|T_i| = i + 1$ implies that $\rank(T_i) \le i$.
  We also stop this process for the first $t$ such that
  \begin{equation*}
    |X \cap \Span(T_t)| \ge \gamma |X|.
  \end{equation*}
  By \eqref{eq:first-step-of-greedy} and our choice of $x_{i + 1}$ via \Cref{stmt:greedy-phase}, we have, for all $i \le t$,
  \begin{equation}\label{eq:result-of-greedy}
    |X + T_i| \ge |X| + (1 - \gamma) \frac{i|X|}{2} \ge (1 - \gamma) \frac{(i + 2)|X|}{2},
  \end{equation}
  and $|T_i| \le i + 1$.

  Now, if $t \ge r$, then we can take $T = T_t$, and we have completed the proof by \eqref{eq:result-of-greedy}.
  Otherwise, we may assume that $t < r$ and define $W_1 = \Span(T_t).$
  We would like to use $W_1$ for the rest of the proof, but \Cref{stmt:phase-3} requires that all non-zero fibres have size at most $\eta |X|$.
  To continue, then, we need to define a subspace $W$ such that for every $z \in Z(W, X) \setminus \{0\}$ and given $\eta > 0$, $$\big| \equivc{z}_{W,X} \big| \le \eta |X|.$$
  We will do so by iteratively projecting these large fibres onto the 0 fibre until none remain.
  With foresight, we set $\eta = \gamma^2$.
  Formally, our process is:
  \begin{enumerate}
    \item Start with $\ell = 1$ and $W_1 = \Span(T_t)$.
    \item If there exists $z_\ell \in Z(W_\ell, X) \setminus \{0\}$ such that $\big| \equivc{z_\ell} \big| \ge \eta |X|$, then we let $$W_{\ell + 1} = \Span(W_\ell \cup \{z_\ell\}).$$
    \item If there is no such $z_\ell$, we stop with output $W_\ell$.
  \end{enumerate}
  A simple and important observation is that $\ell \le 1/\eta$, since $\eta < \gamma$ and $$ |X| \ge |W_\ell \cap X| \ge \ell \eta |X|. $$
  Noting that $\dim(W_1) \le t$, since $|T_t| \le t + 1$ and $0 \in T_t$, it follows that
  \begin{equation}\label{eq:rank-of-W}
    r_W := \dim(W_\ell) \leq \dim(W_1) + 1/\eta \le t + 1/\eta.
  \end{equation}

  Take $W$ to be the output of the above process.
  Moreover, let $Z = Z(W, X)$ and divide the rest of the proof into two cases depending on $|Z|$.
  The first case is if $|Z| \geq (r + 1)/\gamma$.
  Here, we claim that applying \Cref{stmt:many-fibres-are-easier} completes the proof.
  The set $T \subseteq X$ provided by this application satisfies
  \begin{equation*}
    |X + T| \ge (r + 1)|X| \ge (1 - 5\gamma)\frac{(r + 1)|X|}{2},
  \end{equation*}
  and
  \begin{equation*}
    |T| = \frac{r + 1}{\gamma} \le \frac{4(r + 1)}{\gamma},
  \end{equation*}
  as required.

  It remains to deal with the other case, and we may thus assume that $|Z| < (r + 1)/\gamma$.
  In this scenario, our set $T$ will be the union of three sets: the second to last set $T_{t - 1}$, the set $T'$ given by \Cref{stmt:phase-3}, and $T''$, the output of \Cref{stmt:phase-2} for a suitable choice of $u^* \in \R^d$.
  Note that it consists of few translates:
  \begin{equation*}
    |T| \le |T_{t - 1}| + |T'| + |T''| \le t + |Z| + 2|Z| \le \frac{4(r + 1)}{\gamma},
  \end{equation*}
  where the last inequality follows from our assumptions that $t < r$ and $|Z| < (r + 1)/\gamma$.
  It remains to show that $X + T$ has the appropriate size.

  First, we separate the contributions of $T_{t - 1}$ and $\hat{T} = T' \cup T''$ to the sumset
  \begin{equation}\label{eq:first-splitting-sumset}
    |X + T| \geq |X + T_{t - 1}| + \big| \big(X + \hat{T}\big) \setminus \big(X + T_{t - 1}\big)\big|,
  \end{equation}
  and we would like to apply \Cref{stmt:gluing-phase-2-and-phase-3} to bound the second term in the right-hand side.
  To do this, we need to show that $T_{t - 1} \subseteq \equivc{0}_*$ for some $u^* \in \R^d$, as that would imply
  \begin{equation}\label{eq:inclusion-of-T_t-1-in-null-part-of-zero-fibre}
    \big(X + \hat{T}\big) \setminus \big(X + \equivc{0}_*\big) \subseteq \big(X + \hat{T}\big) \setminus \big(X + T_{t - 1}\big).
  \end{equation}

  Besides $T_{t - 1} \subseteq \equivc{0}_*$, our choice of $u^*$ will also need to satisfy
  \begin{equation}\label{eq:small-null-part-of-zero-fibre}
    \big| \equivc{0}_* \big| < \gamma |X|.
  \end{equation}
  To achieve that, we define the subspace $W_0 = \Span(T_{t - 1})$ and claim that we can pick $u^* \in \R^d$ such that
  \begin{equation}\label{eq:choice-of-x*}
    \equivc{0}_* = \{x \in \equivc{0} : \inner{x, u^*} = 0\} = X \cap W_0.
  \end{equation}
  As we have stopped the greedy process at $t$, we must have $|X \cap W_0| < \gamma |X|$, so this choice also satisfies \eqref{eq:small-null-part-of-zero-fibre}.
  We can choose $u^* \in W_0^\perp$ satisfying \eqref{eq:choice-of-x*} because $X$ is finite and
  \begin{equation*}
    \dim(W_0) \le t < r \le d,
  \end{equation*}
  since $|T_{t - 1}| \le t$ and $0 \in T_{t - 1}$.
  Notice that this choice of $u^*$ mimics that of $u$ in the proof of \Cref{stmt:greedy-phase}.

  With this choice of $u^*$, we have $T_{t - 1} \subseteq \equivc{0}_*$, as required, and so combining \eqref{eq:inclusion-of-T_t-1-in-null-part-of-zero-fibre} with \Cref{stmt:gluing-phase-2-and-phase-3} yields, in \eqref{eq:first-splitting-sumset},
  \begin{equation}\label{eq:splitting-sumset}
    |X + T| \ge |X + T_{t - 1}| + \big|\big(X + T'\big) \setminus \big(X + \equivc{0}\big)\big| + \big|\big(\equivc{0} + T''\big) \setminus \big(X + \equivc{0}_*\big)\big|
  \end{equation}
  for $T = T_{t - 1} \cup T' \cup T''$, and $T'$ and $T''$ as given by \Cref{stmt:phase-3} and \Cref{stmt:phase-2}, respectively.
  As $T'$ originates from \Cref{stmt:phase-3}, we have
  \begin{equation}\label{eq:bound-from-phase-3}
    \big|\big(X + T'\big) \setminus \big(X + \equivc{0}\big)\big| \ge \frac{(r - r_W)}{2}\big(|X| - \big|\equivc{0}\big|\big) - \eta|Z||X|,
  \end{equation}
  where, recall, $r_W \leq t + 1/\eta$, by \eqref{eq:rank-of-W}.
  Moreover, by \Cref{stmt:phase-2}, we have
  \begin{equation*}
    \big| \big(\equivc{0} + T'' \big) \setminus \big(X + \equivc{0}_*\big)\big| \ge |Z| \big( \big|\equivc{0}\big| - \big|\equivc{0}_*\big|\big),
  \end{equation*}
  which, by \eqref{eq:small-null-part-of-zero-fibre}, implies that
  \begin{equation}\label{eq:bound-from-phase-2}
    \big| \big(\equivc{0} + T'' \big) \setminus \big(X + \equivc{0}_*\big)\big| \ge |Z| \big( \big|\equivc{0}\big| - \gamma|X|\big).
  \end{equation}

  Recall that, by \eqref{eq:result-of-greedy}, we have
  \begin{equation*}
    |X + T_{t - 1}| \ge (1 - \gamma) \frac{(t + 1)|X|}{2},
  \end{equation*}
  which, replaced alongside \eqref{eq:bound-from-phase-3} and \eqref{eq:bound-from-phase-2} in \eqref{eq:splitting-sumset}, yields
  \begin{align}
    |X + T| & \ge (1 - \gamma) \frac{(t + 1)}{2}|X| \nonumber \\
                & + \frac{(r - r_W)}{2}\Big(|X| - \big|\equivc{0}\big|\Big) \label{eq:negative-part-of-weighted-freiman-term-in-opaque-but-good-enough-bound} \\
                & - \eta|Z||X| \nonumber \\
                & + |Z| \Big( \big|\equivc{0}\big| - \gamma|X|\Big). \label{eq:phase-2-term-in-opaque-but-good-enough-bound}
  \end{align}
  The rest of the proof is dedicated to showing that this expression is at least the bound we need.

  To this end, observe that
  \begin{equation}\label{eq:lower-bound-on-size-of-Z}
    |Z| \ge \rank(Z) \ge \rank(X) - \rank(W) \ge r - r_W.
  \end{equation}
  Another important observation is that
  \begin{equation}\label{eq:lower-bound-on-size-of-null-part-of-zero-fibre}
    \big|\equivc{0}\big| \ge |X \cap W_1| \ge \gamma |X|,
  \end{equation}
  since $X \cap W_1 \subseteq \equivc{0}$.

  It follows from \eqref{eq:lower-bound-on-size-of-Z} and \eqref{eq:lower-bound-on-size-of-null-part-of-zero-fibre} that the sum of \eqref{eq:negative-part-of-weighted-freiman-term-in-opaque-but-good-enough-bound} and \eqref{eq:phase-2-term-in-opaque-but-good-enough-bound} is at least
  \begin{equation}\label{eq:lower-bound-on-middle-term}
    \frac{r - r_W}{2}\Big(|X| - \big|\equivc{0}\big|\Big) + |Z| \Big( \big|\equivc{0}\big| - \gamma|X|\Big) \ge (1 - \gamma) \frac{(r - r_W)}{2}|X|.
  \end{equation}
  Replacing \eqref{eq:lower-bound-on-middle-term} in our lower bound for the size of $X + T$, we obtain
  \begin{equation*}
    |X + T| \ge (1 - \gamma) \frac{(t + 1)}{2}|X| + (1 - \gamma) \frac{(r - r_W)}{2}|X| - \eta |Z| |X|
  \end{equation*}
  and using that $r_W \le t + 1/\eta$ by \eqref{eq:rank-of-W}, yields
  \begin{equation}\label{eq:final-bound}
    |X + T| \ge (1 - \gamma) \frac{(r + 1)}{2}|X| - \frac{(1 - \gamma)}{2\eta}|X| - \eta|Z||X|.
  \end{equation}

  It is now enough to determine that
  \begin{equation}\label{eq:final-bound-inequality}
    (1 -\gamma) \frac{1}{2\eta} + \eta |Z| \le 2 \gamma (r + 1)
  \end{equation}
  as replacing it in \eqref{eq:final-bound} gives the desired bound:
  \begin{equation*}
    |X + T| \ge (1 - 5 \gamma) \frac{r + 1}{2}|X|.
  \end{equation*}
  To obtain \eqref{eq:final-bound-inequality}, observe that
  \begin{equation}\label{eq:bound-on-eta-Z}
    \eta |Z| \le \gamma (r + 1)
  \end{equation}
  follows from our choice of $\eta = \gamma^2$ and our assumption that, in the current case, ${|Z| \leq (r + 1)/\gamma}$.
  Moreover, the following holds
  \begin{equation}\label{eq:bound-on-1-over-2-eta}
    \frac{1}{2 \eta} < \gamma (r + 1)
  \end{equation}
  since $\gamma^3 r \ge 1$ by assumption.
  The proof is complete by substituting \eqref{eq:bound-on-eta-Z} and \eqref{eq:bound-on-1-over-2-eta} in \eqref{eq:final-bound-inequality}.
\end{proof}

\section{Offsetting the loss of the zero fibre: proof of \texorpdfstring{\Cref{stmt:phase-2}}{Proposition~\protect\ref{stmt:phase-2}}}\label{sec:offsetting-the-loss-of-the-zero-fibre}

As we remarked in the previous section, we offset the loss of the zero fibre by adding two carefully selected points from each non-empty fibre to the final set of translates.
These points are the maximiser and minimiser of the linear function $x \mapsto \inner{x, u^*}$ in each fibre.
Let $Z = \{z_1, \ldots, z_m\}$.
For each $z_i \in Z$, we choose $y_i^+, y_i^- \in \equivc{z_i}$ to be, respectively, a maximiser and a minimiser of $y \mapsto \inner{y, u^*}$ in $\equivc{z_i}$, and set $Y_i := \{y_i^+, y_i^-\}$.

Recall that $Z = Z(X, W)$ is now defined as the projection of $X$ onto $W^\perp$, where $X$ and $W$ come from the statement of the \namecref{stmt:phase-2}.
Similarly to the proof of \Cref{stmt:greedy-phase}, we will define ``positive'' and ``negative'' parts of $\equivc{0}$ as
\begin{align*}
  \equivc{0}_+ = \big\{x \in \equivc{0} : \inner{x, u^*} > 0\big\} \quad \text{and} \quad \equivc{0}_- = \big\{x \in \equivc{0} : \inner{x, u^*} < 0\big\},
\end{align*}
which complement the ``null part'' $\equivc{0}_* = \{x \in \equivc{0} : \inner{x, u^*} = 0\}$ defined in the statement and complete a partition of the zero fibre.

Each pair of minimiser and maximiser we put in the set of translates will add to the sumset a translated copy of the sets $\equivc{0}_+$ and $\equivc{0}_-$.
As they form a partition of $\equivc{0} \setminus \equivc{0}_*$, this will correspond to adding $\big|\equivc{0}\big| - \big|\equivc{0}_*\big|$ elements to the sumset for each of the $m = |Z|$ non-empty fibres.
Showing this only requires the following simple geometric observations.
The first says that the translates $\equivc{0} + Y_i$ and $\equivc{0} + Y_j$ are disjoint if $Y_i$ and $Y_j$ lie in distinct fibres.
\begin{obs}\label{obs:disjointness}
  For all $i, j \in \{1, \ldots, m\}$, if $i \neq j$, then
  \begin{equation*}
    \big(\equivc{0} + Y_i\big) \cap \big(\equivc{0} + Y_j\big) = \emptyset.
  \end{equation*}
\end{obs}

\begin{proof}
  Note that $\equivc{0} + Y_i \subseteq [z_i]$, and recall from \eqref{eq:y_i-y_j-disjoint} that $[z_i] \cap [z_j] = \emptyset$.
\end{proof}

The second \namecref{obs:pos-and-neg-disjointness} says that the positive part of $\equivc{0}$ translated by the fibre maximiser $y_i^+$ is disjoint from its negative part translated by the corresponding fibre minimiser.
Its proof is essentially contained in the proof of \Cref{stmt:greedy-phase}.

\begin{obs}\label{obs:pos-and-neg-disjointness}
  For all $i \in \{1, \ldots, m\}$, we have
  \begin{equation}\label{eq:pos-and-neg-disjointness}
    \big(y_i^+ + \equivc{0}_+\big) \cap \big(y_i^- + \equivc{0}_-\big) = \emptyset.
  \end{equation}
\end{obs}

\begin{proof}
  Note that for any $a \in \equivc{0}_+, b \in \equivc{0}_-$, we have
  \begin{align*}
    \inner{y_i^+ + a, u^*} > \inner{y_i^+, u^*} \ge \inner{y_i^-, u^*} > \inner{y_i^- + b, u^*}.
  \end{align*}
  Hence, $y_i^+ + a \neq y_i^- + b$ for every $a \in \equivc{0}_+$ and $b \in \equivc{0}_-$, which implies \eqref{eq:pos-and-neg-disjointness}.
\end{proof}

The final \namecref{obs:empty-intersections} says that the original set translated by the ``null part'' of $\equivc{0}$ does not intersect the positive part of $\equivc{0}$ translated by the fibre maximiser $y_i^+$, and that the analogous statement holds for the negative part and the fibre minimiser $y_i^-$.
\begin{obs}\label{obs:empty-intersections}
  For all $i \in \{1, \ldots, m\}$, we have
  \begin{equation*}
    \big(X + \equivc{0}_*\big) \cap \big(y_i^+ + \equivc{0}_+\big) = \big(X + \equivc{0}_*\big) \cap \big(y_i^- + \equivc{0}_-\big) = \emptyset.
  \end{equation*}
\end{obs}

\begin{proof}
  Recall that $X = \bigcup_{j = 1}^m \equivc{z_j}$ is a partition.
  Whenever $\equivc{z_j}$ and $y_i^+$ are in distinct fibres, we have that $\equivc{z_j} + \equivc{0}_*$ and $y_i^+ + \equivc{0}_+$ are disjoint by \eqref{eq:y_i-y_j-disjoint},
  \begin{equation*}
    \equivc{z_j} + \equivc{0}_* \subseteq [z_j] \quad \text{and} \quad y_i^+ + \equivc{0}_+ \subseteq [z_i].
  \end{equation*}

  We now consider the case when they are in the same fibre.
  Take $a \in \equivc{0}_+$, $c \in \equivc{0}_*$ and $y \in \equivc{z_i}$.
  Note that $\inner{a, u^*} > 0$ since $a \in \equivc{0}_+$, and $\inner{c, u^*} = 0$ as $c \in \equivc{0}_*$.
  Then $c + y \neq a + y_i^+$, because
  \begin{align*}
    \inner{a + y_i^+, u^*} > \inner{c + y_i^+, u^*} \ge \inner{c + y, u^*},
  \end{align*}
  and this completes the proof in the $y_i^+$ case.
  The proof in the $y_i^-$ case is analogous.
\end{proof}

We are now ready to prove \Cref{stmt:phase-2}.

\begin{proof}[Proof of \Cref{stmt:phase-2}]
  Recall that $m = |Z|$, where $Z = Z(X, W)$, and $Y_i = \{y_i^+, y_i^-\}$ consists of a maximiser and minimiser of $y \mapsto \inner{y, u^*}$ in $\equivc{z_i}$ for each $z_i \in Z$.
  $Y_i$ is well-defined because $Z$ is the set of non-empty fibres, although it may be a singleton for example if $\big|\equivc{z_i}\big| = 1$.
  We define $T''$ to be the following set with (trivially) at most $2m$ elements
  \begin{equation*}
    T'' = \bigcup_{i = 1}^m Y_i.
  \end{equation*}

  To avoid cumbersome notation, we will first prove the weaker inequality
  \begin{equation}\label{eq:weaker-phase-2}
    \big|\equivc{0} + T''\big| \ge \sum_{i = 1}^m \big|y_i^+ + \equivc{0}_+ \big| + \big| y_i^- + \equivc{0}_- \big|,
  \end{equation}
  and then show that the same steps can be applied removing the set $X + \equivc{0}_*$ in the left-hand side to obtain the desired bound.

  As $T''$ is the union of the $Y_i$, we use \Cref{obs:disjointness} to obtain
  \begin{equation}\label{eq:T-sumset-as-sum-of-Y_i}
    \big| \equivc{0} + T'' \big| = \sum_{i = 1}^m \big| \equivc{0} + Y_i \big|.
  \end{equation}

  Now, because $\equivc{0}_+$, $\equivc{0}_*$ and $\equivc{0}_-$ partition $\equivc{0}$, we may decompose, for each $Y_i$, the term in \eqref{eq:T-sumset-as-sum-of-Y_i} as
  \begin{equation}\label{eq:Y_i-decomposition}
    \big| \equivc{0} + Y_i \big| = \big| \big( \equivc{0}_+ + Y_i \big) \cup \big( \equivc{0}_* + Y_i \big ) \cup \big( \equivc{0}_- + Y_i \big) \big|,
  \end{equation}
  which, ignoring the term corresponding to $\equivc{0}_*$ and the ``mixed-sign'' terms, $\equivc{0}_- + y_+$ and $\equivc{0}_+ + y_-$, yields
  \begin{equation}\label{eq:Y_i-lowerbound-from-discarding-sets}
    \big| \equivc{0} + Y_i \big| \ge \big| \big( \equivc{0}_+ + y_i^+ \big) \cup \big( \equivc{0}_- + y_i^- \big) \big|.
  \end{equation}
  We can now apply \Cref{obs:pos-and-neg-disjointness}, which implies that the right-hand side of \eqref{eq:Y_i-lowerbound-from-discarding-sets} is at least
  \begin{equation}\label{eq:pos-and-neg-decomposition}
    \big| \big( \equivc{0}_+ + y_i^+ \big) \cup \big( \equivc{0}_- + y_i^- \big) \big| = \big| \equivc{0}_+ + y_i^+ \big| + \big| \equivc{0}_- + y_i^- \big|,
  \end{equation}
  hence establishing \eqref{eq:weaker-phase-2} via \eqref{eq:T-sumset-as-sum-of-Y_i}.

  To obtain the desired bound, we now show that the same steps can be applied removing the set $X + \equivc{0}_*$ in the left-hand side of \eqref{eq:weaker-phase-2}.
  First, \eqref{eq:T-sumset-as-sum-of-Y_i} holds if we remove $X + \equivc{0}_*$ from the set in the left-hand side and those in the sum in the right-hand side because removing a set cannot create intersections in disjoint sets.
  We then repeat the steps in \eqref{eq:Y_i-decomposition} and \eqref{eq:Y_i-lowerbound-from-discarding-sets} -- they are also possible regardless of the set removal -- and, upon reaching \eqref{eq:pos-and-neg-decomposition}, we again use that disjointness is preserved under set removal.
  This gives us
  \begin{equation*}
    \Big| \big(\equivc{0} + T''\big) \setminus \big( X + \equivc{0}_* \big) \Big| \ge \sum_{i = 1}^m \Big| \big( y_i^+ + \equivc{0}_+ \big) \setminus \big( X + \equivc{0}_* \big) \Big| + \Big| \big( y_i^- + \equivc{0}_- \big) \setminus \big( X + \equivc{0}_* \big) \Big|.
  \end{equation*}

  We are now in a position to use \Cref{obs:empty-intersections} to deduce that the set we removed is disjoint from the ones in the right-hand side above, and recover the lower bound prior to the removal
  \begin{align*}
    \Big| \big(\equivc{0} + T''\big) \setminus \big( X + \equivc{0}_* \big) \Big|
    & \ge \sum_{i = 1}^m \big|y_i^+ + \equivc{0}_+ \big| + \big| y_i^- + \equivc{0}_- \big| \\
    &= m \Big( \big| \equivc{0}_+ \big| + \big| \equivc{0}_- \big| \Big).
  \end{align*}

  Again using that $\equivc{0}_- \cup \equivc{0}_* \cup \equivc{0}_+$ is a partition of $\equivc{0}$, we get
  \begin{equation*}
    m \Big( \big|\equivc{0}_+\big| + \big|\equivc{0}_-\big|\Big) = m \Big( \big|\equivc{0}\big| - \big|\equivc{0}_*\big|\Big),
  \end{equation*}
  as required to complete the proof.
\end{proof}

\section{Weighted \Freiman's lemma: proof of \texorpdfstring{\Cref{stmt:phase-3}}{Proposition~\protect\ref{stmt:phase-3}}}\label{sec:weighted-freiman}

The final piece we need to prove \Cref{stmt:freiman-lemma-via-few-translates} is detailing how to choose $T' \subseteq X$ such that $|T'| \le |Z|$ and \eqref{eq:phase-3} holds, and how to prove a variant of \Freiman's lemma where the size of the original, unprojected set $X$ appears in the lower bound, instead of simply $|Z|$.

We will first prove a weaker, insufficient statement to make the reader comfortable with the notation and ideas in the proof of \Cref{stmt:phase-3}.
Our goal here is to show that we can choose $T' \subseteq X$ such that $|T'| \le |Z|$ and\footnote{Notice that $(w - Z) \cap Z = \{z \in Z : \exists z' \in Z \, \text{ such that } z + z' = w\}$.}
\begin{equation}\label{eq:weighted-freiman-warm-up}
  |X + T'| \ge \sum_{w \in Z + Z} \max_{z \in (w - Z) \cap Z} \big| \equivc{z} \big|.
\end{equation}

Before proceeding, it will be useful to let $Z = \{z_1, \ldots, z_m\}$.
Here (and in the proof itself), we will take
\begin{equation*}
  T' = \{y_1, \ldots, y_m\},
\end{equation*}
where, for each $i \in \{1, \ldots, m\}$, $y_i \in \equivc{z_i}$ is arbitrary.
It is immediate that the size of $T'$ is appropriate, that is,
\begin{equation*}
  |T'| \le |Z|,
\end{equation*}
so we must show that it also satisfies \eqref{eq:phase-3}.

Having defined $T'$, we start this warm-up by partitioning $X + T'$ into fibres of $Z + Z$:
\begin{equation}\label{eq:simple-obs}
  X + T' = \bigcup_{w \in Z + Z} [w] \cap (X + T').
\end{equation}
As \eqref{eq:simple-obs} defines a partition, we also have
\begin{equation*}
  \left| \bigcup_{w \in Z + Z} \Big( [w] \cap (X + T') \Big) \right| = \sum_{w \in Z + Z} \left| \Big( [w] \cap (X + T') \Big) \right|.
\end{equation*}
Note also that if $w = z_i + z_j \in Z + Z$, then
\begin{equation}\label{eq:fibre-is-contained}
  \equivc{z_i} + y_j \subseteq \big( [w] \cap (X + T') \big)
\end{equation}
since $\equivc{z_i} \subseteq X$ and $y_j \in \equivc{z_j} \cap T'$.

Our approach will be to count only the elements in a single (translated) fibre $\equivc{z_i} + y_j$ and ignore the rest of $[z_i + z_j] \cap (X + T')$ -- by \eqref{eq:simple-obs} and \eqref{eq:fibre-is-contained}, this is a valid lower bound for $|X + T'|$.
As the union in \eqref{eq:simple-obs} ranges over all $w \in Z + Z$, we can pick any possible representation $z_i + z_j$ for $w$.
Our choice will be the one for which $\big|\equivc{z_i}\big|$ is as large as possible, resulting in
\begin{equation*}
  \sum_{w \in Z + Z} \left| \Big( [w] \cap (X + T') \Big) \right| \ge \sum_{w \in Z + Z} \max_{z \in (w - Z) \cap Z} \big| \equivc{z} \big|,
\end{equation*}
since $\big| \equivc{z} + y \big| = \big| \equivc{z} \big|$.
This completes the proof of \eqref{eq:weighted-freiman-warm-up}.

In the proof of \Cref{stmt:phase-3}, the statement that is analogous to \eqref{eq:weighted-freiman-warm-up} is
\begin{equation}\label{eq:phase-3-first-step}
  \Big|\big(X + T'\big) \setminus \big(X + \equivc{0}\big) \Big| \ge \sum_{w \in (Z^* + Z^*) \setminus Z} \max_{z \in (w - Z^*) \cap Z^*} \big| \equivc{z} \big|,
\end{equation}
where $Z^* = Z \setminus \{0\}$ -- in words, we ``remove the $0$ from $Z$'' before doing the sumset.
The proof of \eqref{eq:phase-3-first-step} is essentially the same as the warm-up above, but that is still not enough.
We must have a good lower bound for its right-hand side in order to complete a proof of \Cref{stmt:phase-3}.

\Cref{stmt:weighted-freiman}, below, is the lower bound we need for the right-hand side of \eqref{eq:phase-3-first-step}: when applying it to prove \Cref{stmt:phase-3}, we will set $U = Z^*$ and $f(z) = \big| \equivc{z} \big|$.
The proof of the \namecref{stmt:weighted-freiman} is similar to traditional proofs of \Freiman's lemma, but we must take into account the weight $f(u)$ of each $u \in U$ when selecting vertices of $\conv(U)$, the convex hull of $U$, instead of choosing an arbitrary vertex\footnote{Hence the name ``weighted \Freiman's lemma''.}.
\begin{lem}\label{stmt:weighted-freiman}
  Let $d \in \N$.
  For every finite $U \subseteq \R^d$ and $f: U \to \R_+$, we have
  \begin{equation}\label{eq:key-weighted-freiman}
    \sum_{w \in U + U} \max_{u \in (w - U) \cap U} f(u) \ge \frac{\rank(U) + 1}{2} \sum_{u \in U} f(u).
  \end{equation}
\end{lem}

We will also need a simple observation to repeat one of the steps in the proof of \Freiman's lemma and prove \Cref{stmt:weighted-freiman}.

\begin{obs}\label{obs:facet-intersection}
  Let $U \subseteq \R^d$ be a finite set.
  If $v \in U$ is a vertex of $\conv(U)$ and $U' = U \setminus \{v\}$, then there exists a hyperplane $\cH$ such that $v$ and $U' \setminus \cH$ are on different sides of $\cH$.
  Moreover, $|\cH \cap U'| \ge \rank(U')$.
\end{obs}

\begin{proof}
  Write $\conv(U') = \bigcap_{\cH \in \mathbf{H}} \cH^-$, where $\mathbf{H}$ is the collection of hyperplanes supporting the facets of $\conv(U')$ and $\cH^-$ denotes a closed half-space defined by $\cH$.
  Since $v \not \in \conv(U')$, there exists $\cH \in \mathbf{H}$ such that $v \not \in \cH^-$, so $v$ and $U' \setminus \cH$ are on different sides of $\cH$.
  The second part of the statement follows from $\cH$ intersecting a facet of $\conv(U')$ and the fact that every facet contains at least $\rank(U')$ vertices.
\end{proof}

Now, we can prove \Cref{stmt:weighted-freiman}.

\begin{proof}[Proof of \Cref{stmt:weighted-freiman}]
  We prove the \namecref{stmt:weighted-freiman} by induction on $|U|$.
  In the base case, we have an empty set, so the left-hand side of \eqref{eq:key-weighted-freiman} is equal to 0, as is the right-hand side.
  We can therefore assume that for every $U' \subsetneq U$, we have \eqref{eq:key-weighted-freiman}.

  For the induction step, we proceed similarly to the standard proof of \Freiman's lemma.
  The main difference is that instead of choosing an arbitrary element from $V$, the vertices of $\conv(U)$, we must choose a vertex considering the weight function $f$.
  To be precise, we choose a vertex $v \in V$ such that
  \begin{equation}\label{eq:weighted-freiman-criterion-for-choice-of-v}
    f(v) \le \frac{1}{s + 1} \sum_{u \in U} f(u),
  \end{equation}
  where $s = \rank(U)$, noting that such a choice exists by the pigeonhole principle,
  \begin{equation*}
    (s + 1) \min_{v \in V} f(v) \le |V| \min_{v \in V} f(v) \le \sum_{v \in V} f(v) \le \sum_{u \in U} f(u).
  \end{equation*}
  We fix one such $v \in V$, define $U' = U \setminus \{v\}$ and divide the remainder of the proof based on whether ${\rank(U') = s}$.

  Consider first the case $\rank(U') = s$.
  The induction hypothesis implies that
  \begin{equation}\label{eq:weighted-freiman-first-case-inductive-hypothesis-application}
    \sum_{w \in U' + U'} \max_{u \in (w - U) \cap U} f(u) \ge \sum_{w \in U' + U'} \max_{u \in (w - U') \cap U'} f(u) \ge \frac{s + 1}{2} \sum_{u' \in U'} f(u')
  \end{equation}
  since $U' \subseteq U$.
  Therefore, by \eqref{eq:weighted-freiman-first-case-inductive-hypothesis-application}, we can accomplish our goal by finding a set $S \subseteq (U + U) \setminus (U' + U')$ such that
  \begin{equation}\label{eq:weighted-freiman-requirement-for-set-S}
   \sum_{w \in S} \max_{u \in (w - U) \cap U} f(u) \ge \frac{s + 1}{2} f(v).
  \end{equation}

  In order to define our candidate set for $S$, let $\cH$ be the hyperplane given by \Cref{obs:facet-intersection} for $U$ and $v$, and take $\oN{v} = \big(\cH \cap \conv(U')\big) \cup \{v\}$.
  We claim that $\oN{v} + v$ is a suitable choice for $S$ because it is disjoint from $U' + U'$ and satisfies \eqref{eq:weighted-freiman-requirement-for-set-S}.

  To see that $\oN{v} + v$ and $U' + U'$ are disjoint, first notice that $2v \not \in U' + U'$ follows from $v$ being a vertex of $\conv(U)$.
  For the remaining elements of $\oN{v}$, \ie $v' \in \cH \cap \conv(U')$, $v' + v$ is not in $U' + U'$ because $(v' + v)/2$ is a midpoint of the segment connecting $v$ and $v'$, and this midpoint clearly lies outside $\conv(U')$ by our choice of $\cH$.

  It remains to show that \eqref{eq:weighted-freiman-requirement-for-set-S} holds with $S = \oN{v} + v$.
  Observe that if $w = v' + v$ for some $v' \in \oN{v}$, then $v \in (w - U) \cap U$.
  Hence,
  \begin{equation}\label{eq:weighted-freiman-lower-bound-for-S-surplus}
    \sum_{w \in \oN{v} + v} \max_{u \in (w - U) \cap U} f(u) \ge \sum_{w \in \oN{v} + v} f(v) = \big|\oN{v}\big| \, f(v) \ge (s + 1) f(v),
  \end{equation}
  where in the last inequality we used that $$|\oN{v}| = |\cH \cap U'| + 1 \ge \rank(U') + 1 = s + 1$$ which is due to our choice of $\cH$ by \Cref{obs:facet-intersection} and our assumption that $\rank(U') = s$.
  Combining \eqref{eq:weighted-freiman-first-case-inductive-hypothesis-application} and \eqref{eq:weighted-freiman-lower-bound-for-S-surplus}, this completes the induction step when $\rank(U') = \rank(U)$.

  We may therefore assume that $\rank(U') = s - 1$; that is, the removal of the vertex $v$ decreases the rank of $U$.
  It will be useful for this case to recall \eqref{eq:weighted-freiman-criterion-for-choice-of-v}, our criterion for the choice of $v$, in the following (trivially) equivalent form:
  \begin{equation}\label{eq:weight-of-v}
    \sum_{u \in U} f(u) \ge (s + 1) f(v).
  \end{equation}
  Applying our induction hypothesis to $U'$ yields
  \begin{equation}\label{eq:weighted-freiman-second-case-inductive-hypothesis}
    \sum_{w \in U' + U'} \max_{u \in (w - U') \cap U'} f(u) \ge \frac{s}{2} \sum_{u' \in U'} f(u'),
  \end{equation}
  so, to deduce the bound \eqref{eq:key-weighted-freiman}, it is enough to add to \eqref{eq:weighted-freiman-second-case-inductive-hypothesis} the term
  \begin{equation}\label{eq:weighted-freiman-second-case-reduced-goal}
    \Big(\frac{1}{2} \sum_{u' \in U'} f(u')\Big) + \frac{s + 1}{2} f(v)
  \end{equation}
  using elements of $(U + U) \setminus (U' + U')$.

  In order to add \eqref{eq:weighted-freiman-second-case-reduced-goal} to \eqref{eq:weighted-freiman-second-case-inductive-hypothesis}, we argue that $U' + U'$, $\{2v\}$ and $U' + v$ are disjoint.
  This follows from our assumption that $\rank(U') < \rank(U)$, which implies that $v \not \in \mathcal{U}$, where $\mathcal{U}$ is an affine subspace with dimension $\rank(U')$ containing $U'$.
  We can therefore conclude that
  \begin{equation}\label{eq:weighted-freiman-second-case-intermediate-lower-bound}
    \sum_{w \in U + U} \max_{u \in (w - U) \cap U} f(u) \ge \frac{s}{2} \sum_{u' \in U'} f(u') + f(v) + \sum_{w \in U' + v} \max_{u \in (w - U) \cap U} f(u).
  \end{equation}

  Our first observation towards bounding the right-hand side of \eqref{eq:weighted-freiman-second-case-intermediate-lower-bound} is that
  \begin{equation*}
    \sum_{w \in U' + v} \max_{u \in (w - U) \cap U} f(u) \ge \sum_{u' \in U'} f(u')
  \end{equation*}
  because if $w = u' + v$ for some $u' \in U'$, then $u' \in (w - U) \cap U$.
  It will therefore suffice to show that
  \begin{equation*}
    f(v) + \sum_{u' \in U'} f(u') \ge \Big(\frac{1}{2} \sum_{u' \in U'} f(u')\Big) + \frac{s + 1}{2} f(v),
  \end{equation*}
  or, equivalently,
  \begin{equation*}
    2f(v) + \sum_{u' \in U'} f(u') \ge (s + 1) f(v).
  \end{equation*}
  We now use that our choice of $v$ satisfies \eqref{eq:weight-of-v}, which implies that
  \begin{equation*}
    2f(v) + \sum_{u' \in U'} f(u') \ge \sum_{u \in U} f(u) \ge (s + 1) f(v),
  \end{equation*}
  and hence completes the proof of the induction step.
\end{proof}

With \Cref{stmt:weighted-freiman}, we are now ready to prove \Cref{stmt:phase-3}.
\begin{proof}[Proof of \Cref{stmt:phase-3}]
  Our first claim is that finding a $T' \subseteq X$ such that $|T'| \le |Z|$ and
  \begin{equation}\label{eq:phase-3-alternative-goal}
    \Big|\big(X + T'\big) \setminus \big(X + \equivc{0}\big) \Big| \ge \sum_{w \in (Z^* + Z^*) \setminus Z} \max_{z \in (w - Z^*) \cap Z^*} \big| \equivc{z} \big|,
  \end{equation}
  where $Z^* = Z \setminus \{0\}$, is enough to complete the proof.
  To see that, apply \Cref{stmt:weighted-freiman} with $U = Z^*$ and $f(z) = \big| \equivc{z} \big|$ to get the lower bound
  \begin{equation}\label{eq:phase-3-lower-bound-with-weighted-freiman}
    \sum_{w \in Z^* + Z^*} \max_{z \in (w - Z^*) \cap Z^*} \big| \equivc{z} \big| \ge \frac{\rank(Z^*) + 1}{2} \sum_{z \in Z^*} \big| \equivc{z} \big| \ge \frac{r - r_W}{2} \big(|X| - \big| \equivc{0} \big|\big),
  \end{equation}
  where in the last equality we have used that
  \begin{equation*}
    Z^* = Z \setminus \{0\}, \quad  X = \bigcup_{z \in Z} \equivc{z} \quad \text{and} \quad \rank(Z^*) \ge r - r_W - 1,
  \end{equation*}
  by our assumptions that $Z = \Pi_{W^\perp}(X)$, $\rank(X) \ge r$ and $r_W = \rank(W)$.

  However, the left-hand side of \eqref{eq:phase-3-lower-bound-with-weighted-freiman} is considering elements $w \in Z$ that are not in the sum in \eqref{eq:phase-3-alternative-goal}.
  This is not an issue because
  \begin{equation}\label{eq:phase-3-upper-bound-on-zero-fibre-term}
    \sum_{w \in Z} \max_{z \in (w - Z^*) \cap Z^*} \big| \equivc{z} \big| \le \eta |Z| |X|
  \end{equation}
  follows from our assumption that $\big| \equivc{z} \big| \le \eta |X|$ for all $z \in Z^*$.
  We obtain the claim that $T'$ as specified finishes the proof by substituting \eqref{eq:phase-3-upper-bound-on-zero-fibre-term} and \eqref{eq:phase-3-lower-bound-with-weighted-freiman} into \eqref{eq:phase-3-alternative-goal}:
  \begin{equation*}
    \sum_{w \in (Z + Z) \setminus Z} \max_{z \in (w - Z^*) \cap Z^*} \big| \equivc{z} \big| \ge \frac{r - r_W}{2} \big(|X| - \big| \equivc{0} \big|\big) - \eta |Z| |X|.
  \end{equation*}

  As it now suffices to find $T' \subseteq X$ such that $|T'| \le |Z|$ and \eqref{eq:phase-3-alternative-goal} holds, we simply need to repeat the proof of \eqref{eq:weighted-freiman-warm-up} given in the warm-up, but with the set $X + \equivc{0}$ removed.
  That is, let $Z^* = \{z_1, \ldots, z_m\}$, and define
  \begin{equation*}
    T' = \{y_1, \ldots, y_m\}
  \end{equation*}
  where each $y_i$ is an arbitrary element of $\equivc{z_i}$ for $i \in \{1, \ldots, m\}$.
  The first step to show that $T'$ satisfies \eqref{eq:phase-3-alternative-goal} is partitioning $X + T'$ into fibres of $Z + Z$ to obtain
  \begin{equation*}
    \Big| \big(X + T'\big) \setminus \big(X + \equivc{0}\big) \Big| = \sum_{w \in Z + Z} \Big| \big( [w] \cap (X + T') \big) \setminus \big( X + \equivc{0} \big) \Big|.
  \end{equation*}

  In order to handle the set removal, we claim that if $w \not \in Z$, then the sets $[w]$ and $X + \equivc{0}$ are disjoint, and therefore
  \begin{equation}\label{eq:phase-3-lower-bound-from-ignoring-w-in-Z}
    \sum_{w \in Z + Z} \Big| \big( [w] \cap (X + T') \big) \setminus \big( X + \equivc{0} \big) \Big| \ge \sum_{w \in (Z + Z) \setminus Z} \big| [w] \cap (X + T') \big|.
  \end{equation}
  To see this, simply note that if $x \in X$ and $x' \in \equivc{0}$, then
  \begin{equation*}
    \Pi_{W^\perp}(x + x') = \Pi_{W^\perp}(x) + \Pi_{W^\perp}(x') \in Z + 0 = Z
  \end{equation*}
  by the definitions of $\equivc{0}$ and $Z$, and therefore $\Pi_{W^\perp}(X + \equivc{0}) \subseteq Z$.

  Having established \eqref{eq:phase-3-lower-bound-from-ignoring-w-in-Z}, we can proceed (almost) like in the warm-up.
  Restricting our attention to $Z^* + Z^* \subseteq Z + Z$, observe that every $w \in Z^* + Z^*$ can be written as
  \begin{equation}\label{eq:phase-3-ways-to-write-w-in-Z*+Z*}
    w = z + z' \quad \text{where} \quad z, z' \in (w - Z^*) \cap Z^*.
  \end{equation}
  For each $w \in Z^* + Z^*$, then, we pick the pair $(z, z') \in (Z^*)^2$ satisfying \eqref{eq:phase-3-ways-to-write-w-in-Z*+Z*} that maximises $\big| \equivc{z} \big|$.
  Let $y$ be the (unique) element of $T' \cap [z']$, and note that, as $\equivc{z} + y \subseteq [w] \cap (X + T')$, we can conclude that
  \begin{equation}\label{eq:phase-3-claim-for-fibre}
    \big|[w] \cap (X + T')\big| \ge \max_{z \in (w - Z^*) \cap Z^*} \big| \equivc{z} \big|.
  \end{equation}
  Replacing \eqref{eq:phase-3-claim-for-fibre} into \eqref{eq:phase-3-lower-bound-from-ignoring-w-in-Z} yields that \eqref{eq:phase-3-alternative-goal} holds for $T'$
  \begin{equation*}
    \sum_{w \in (Z + Z) \setminus Z} \big| [w] \cap (X + T') \big| \ge \sum_{w \in (Z^* + Z^*) \setminus Z} \max_{z \in (w - Z^*) \cap Z^*} \big| \equivc{z} \big|,
  \end{equation*}
  and the proof therefore follows from our first claim.
\end{proof}

\section{The supersaturation result}\label{sec:supsat}

This \namecref{sec:supsat} is dedicated to the proof of the following theorem, restated for convenience:
\supsat*

Now that we have \Cref{stmt:freiman-lemma-via-few-translates}, proving \Cref{stmt:supsat} is simple.
Assume that $Y$ is small, that is, $|Y| \le (1 - \gamma)(d + 1)|X|/2$; our goal is to show that it misses many pairs of $X^2$.
If we can find $\eps|X|$ elements $x^{(i)} \in X$ such that $Y$ contains at most a $1 - c$ proportion of $X + x^{(i)}$, we are done.
To do that, we will use \Cref{stmt:freiman-lemma-via-few-translates} to find a small set $T \subseteq X$ such that $|X + T| - |Y| \gtrsim \gamma (d + 1)|X|$.
By the pigeonhole principle, it follows that there exists an $x^{(i)} \in T$ such that $Y$ misses many elements of $X + x^{(i)}$.
We can then remove $x^{(i)}$ from $X$ and repeat the process until the dimension of $X$ drops below its original value (removing the $x^{(i)}$ is a simple way to avoid using the same element twice while keeping our working set large).
As we assumed that $X$ has $\eps$-robust \Freiman dimension $d$, this can only happen after we have removed $\eps|X|$ translates.

This is essentially the proof except for the fact that \Cref{stmt:freiman-lemma-via-few-translates} requires $X \subseteq \R^d$, and the set in \Cref{stmt:supsat} is a subset of $\Zmodn$.
To handle this, we use \Freiman isomorphisms, relying on the fact that if $\phi$ is a \Freiman isomorphism, then $|X_1 + X_2| = |\phi(X_1) + \phi(X_2)|$.

\begin{proof}[Proof of \Cref{stmt:supsat}]
  Assume that\footnote{We use $\gamma'$ instead of $\gamma$ because its value is (slightly) less than the $\gamma$ in the application of \Cref{stmt:freiman-lemma-via-few-translates}.} $|Y| \le (1 - \gamma')(d + 1)|X|/2$.
  We claim that there exists a sequence of distinct elements $x^{(1)}, \dots, x^{(t)} \in X$, where $t = \eps |X| / 4$, such that the sets $X_i = X \setminus \{x^{(1)}, \dots, x^{(i)}\}$ have the following two properties
  \begin{equation}\label{eq:xi-requirements}
    d_i = \dimF(X_i) \ge d \quad \text{ and } \quad \big|(X_i + x^{(i + 1)}) \setminus Y\big| \ge 4 c|X|.
  \end{equation}
  The first property holds because, for all $i \le t$,
  \begin{equation}\label{eq:size-of-Xi}
    |X_i| \ge |X| - i \ge |X| - t = \left( 1 - \frac{\eps}{4} \right ) |X|,
  \end{equation}
  so we still have $\dimF(X_i) \ge d$, since $X$ has $\eps$-robust \Freiman dimension $d$.

  To prove that $X_i$ satisfies the second property in \eqref{eq:xi-requirements}, we will show how to select each $x^{(i + 1)}$.
  That is, assume that we have distinct translates $\{x^{(1)}, \dots, x^{(i)}\}$ such that the set $X_i$ satisfies \eqref{eq:xi-requirements}.
  Since $\dimF(X_i) = d_i \ge d$, there exists a \Freiman isomorphism  ${\phi_i : X_i \to X_i'}$ such that ${X_i' \subseteq \R^{d_i}}$ has full rank.
  If $\gamma' > 2^6 d^{-1/3}$, then we can apply \Cref{stmt:freiman-lemma-via-few-translates} to the set $X_i'$ with $r = d$ and $\gamma = 2^{-6} \gamma'$, to obtain a set $T_i' \subseteq X_i'$ such that
  \begin{equation*}
    |T_i'| \le \frac{2^6 (d + 1)}{\gamma'} \quad \text{ and } \quad |X_i' + T_i'| \ge \left(1 - \frac{\gamma'}{4}\right) \frac{(d + 1)|X_i'|}{2}.
  \end{equation*}
  Otherwise, we have $\gamma' \le 2^6 d^{-1/3}$, \ie $d \le 2^{18}\gamma'^{-3}$.
  In this case, we can\footnote{As it is stated, \Cref{stmt:jing-mudgal-freiman} can only be used for $d = d_i$, but that could result in too many translates.
  To circumvent this issue, we can randomly project the set to $\R^d$, and apply \Cref{stmt:jing-mudgal-freiman} to the projected set instead, using the randomness to avoid collisions.} apply \Cref{stmt:jing-mudgal-freiman} and obtain $T_i' = \{x_1, \ldots, x_C\} \subseteq X_i'$ such that
   \begin{equation*}
    |X_i' + T_i'| \ge (d + 1)|X_i'| - 5(d + 1)^3 \ge \frac{(d + 1)|X_i'|}{2},
  \end{equation*}
  for some constant $C = C(\gamma')$, since $X_i$ satisfies \eqref{eq:size-of-Xi}, $X$ is sufficiently large and $d \le 2^{18}\gamma'^{-3}$.
  Therefore, it follows that we have a set of translates $T_i'$ of size at most $2^{-6} \gamma' (d + 1)/c$ for some constant $c = c(\gamma') > 0$, in either case.

  As $\phi_i$ is a \Freiman isomorphism, we know that the preimage $T_i = \phi_i^{-1}(T_i')$ satisfies
  \begin{equation*}
    |X_i + T_i| \ge \left( 1 - \frac{\gamma'}{4} \right) \frac{(d + 1)|X_i|}{2} = \left( 1 - \frac{\gamma'}{4} \right ) \frac{(d + 1)(|X| - i)}{2}
  \end{equation*}
  and ${|T_i| \le 2^{-6} \gamma' (d + 1)/c}$.
  Since $i \le t = \eps |X| / 4 < \gamma' |X| / 4,$ this is at least
  \begin{equation*}
    |X_i + T_i| \ge \left(1 - \frac{\gamma'}{4}\right)^2 \frac{(d + 1)|X|}{2}
    \ge \left(1 - \frac{\gamma'}{2}\right) \frac{(d + 1) |X|}{2}.
  \end{equation*}
  Recalling our assumption $|Y| \le (1 - \gamma') (d + 1)|X|/2$, it follows that
  \begin{equation*}
    |(X_i + T_i) \setminus Y| \ge |X_i + T_i| - |Y| \ge \frac{\gamma' (d + 1) |X|}{4}.
  \end{equation*}
  Thus, by the pigeonhole principle, there exists $x^{(i + 1)} \in T_i$ such that
  \begin{equation*}
    |(X_i + x^{(i + 1)}) \setminus Y| \ge \frac{\gamma' (d + 1) |X|}{4} \left ( \frac{\gamma' (d + 1)}{2^6 c} \right)^{-1} > 4 c|X|,
  \end{equation*}
  since $|T_i| \le 2^{-6} \gamma' (d + 1)/c$ and $T_i \subseteq X_i$.
  Repeating this for each $i$ from 1 to $t$ thus yields the sets $X_i$ and $\{x^{(1)}, \ldots, x^{(t)}\}$ satisfying \eqref{eq:xi-requirements}.

  Now that we have the sets $X_i$, note that each $X + x^{(i)}$ contributes $4 c |X|$ ordered pairs whose sum are not in $Y$.
  This gives us a total of
  \begin{equation*}
    4 c |X| t \ge c \eps |X|^2
  \end{equation*}
  such pairs, because $t = \eps |X| /4$.
\end{proof}

\section{An upper bound for the independence number}\label{sec:main-theorem}

In this section, we prove the upper bound part of our main result, \Cref{stmt:main-result}:
\begin{thm}\label{stmt:upper-bound}
  Let $n$ be a prime number and let $p = p(n)$ satisfy $p \ge (\log n)^{-1/80}$.
  The random Cayley sum graph $\Gamma_p$ of $\Zmodn$ satisfies
  \begin{equation*}
    \alpha(\Gamma_p) \le \big(2 + o(1)\big) \log_{\frac{1}{1 - p}} n
  \end{equation*}
  with high probability as $n \to \infty$.
\end{thm}

Throughout, we fix a small enough $\delta > 0$ and $k = (2 + 4\delta) \log_{\frac{1}{1 - p}}n$.
We will also follow the outline presented in \Cref{sec:overview} and use the notation defined there.
Each sub-collection $\cX_i$ requires different techniques to bound the probability that $\alpha(\Gamma_p) > k$, so we handle each separately and show all three go to 0 as $n \to \infty$.

\subsection{Bounding the probability over choices in $\cX_1$}

A brief recap: in this doubling range, there are far too many choices for $X \in \cX_1$ for a union bound to work.
The key observation is that we do not need to count every such $X$: if somehow going to a smaller subset $\Lambda \subseteq X$ reduces our choices considerably, this would also be enough, since the event $\{\Lambda \rsum \Lambda \subseteq \complement{A_p}\}$ contains $\{X \rsum X \subseteq \complement{A_p}\}$.
As a matter of fact, we will use this idea twice.

The first time we use this idea is in order to replace each set $X \in \cX_1$ by a large subset $X' \subseteq X$ whose \Freiman dimension is robust in the sense required by \Cref{stmt:supsat}, our supersaturation result.
We build the set $X'$ from $X$ by greedily removing small, ``bad'' subsets $B \subseteq X$ such that removing them from $X$ reduces its \Freiman dimension.
Each such removal reduces $\dimF(X)$ by at least one, so we finish in at most $\dimF(X) - 1$ steps.
A step reduces the size of the working set by at most $\eps |X|$, which would already give $|X'| \ge (1 - \dimF(X) \eps) |X|$.
To get a bound that depends on $\sigma$ instead of on $\dimF(X)$, we use the following trivial consequence of \Freiman's lemma, \Cref{stmt:freimans-lemma}, which we record for ease of reference later.

\begin{obs}\label{stmt:dim-at-most-twice-doubling}
  Let $X \subseteq \Zmodn$ with $\sigma[X] \le \sigma$.
  Then, $\dimF(X) + 1 \le 2 \sigma$.
\end{obs}
\begin{proof}
  Let $d = \dimF(X)$, and take $X' \subseteq \Z^d$ of full rank to be the image of $X$ under a \Freiman isomorphism $\phi$.
  Applying \Freiman's lemma to $X'$ yields
  \begin{equation*}
    |X' + X'| \ge (d + 1)|X'| - \binom{d + 1}{2}
  \end{equation*}
  but $|X' + X'| \le \sigma |X'|$ and $|X'| = |X|$ by our choice of $\phi$.
  Dividing both sides by $|X|$ yields
  \begin{equation*}
    \sigma \ge d + 1 - \frac{d + 1}{2}
  \end{equation*}
  where we used $|X| \ge d$ to simplify the right-hand side.
  The proof follows by rearranging.
\end{proof}

The proof of the following \namecref{stmt:robust-subset} is just formalizing the sketch using \Cref{stmt:dim-at-most-twice-doubling}.

\begin{prop}\label{stmt:robust-subset}
  Let $X \subseteq \Zmodn$ with $\sigma[X] \le \sigma$ and let $\eps < 1/2\sigma$.
  There exists $d \in \N$ and $X' \subseteq X$ such that $|X'| > (1 - 2 \eps \sigma) |X|$ and $X'$ has $\eps$-robust \Freiman dimension $d$.
\end{prop}
\begin{proof}
  We start an iterative process with $X_0 = X$ and $i = 0$.
  While there exists $B \subseteq X_i$ such that $$|B| \le \eps |X| \quad \text{ and } \quad \dimF(X_i \setminus B) < \dimF(X_i),$$ we define $X_{i + 1} = X_i \setminus B$.
  Let $t$ be the number of steps in this process, and let $X' = X_t$.

  First, notice that $X'$ has $\eps$-robust \Freiman dimension because there is no $B \subseteq X'$ to continue the process.
  Moreover, $t \le \dimF(X) - 1$ since each step reduces $\dimF(X_i)$ by at least one.
  Since $|X_i| \ge |X| - i \eps |X|$, it follows that
  \begin{equation*}
    |X'| \ge (1 - \eps t) |X| > (1 - \eps \dimF(X)) |X| > (1 - 2 \eps \sigma) |X|,
  \end{equation*}
  where in the last inequality we used \Cref{stmt:dim-at-most-twice-doubling}.
\end{proof}

The second time we employ the idea of going to smaller subsets is when we use the fingerprints given by \Cref{stmt:fingerprints}, whose statement we repeat here for convenience.
\hypertarget{stmt:fingerprints-repeated}{\fingerprints*}

As we remarked in the overview, applying the previous \namecref{stmt:fingerprints} to the (trivial) generalised arithmetic progression $\Zmodn$ would result in too many fingerprints.
To overcome this, we will use \Cref{stmt:green-ruzsa-with-dim(P)<=dimF(X)} below.
We will obtain said \namecref{stmt:green-ruzsa-with-dim(P)<=dimF(X)} by combining a strengthening of the Green--Ruzsa theorem due to \citet{CS13}, \Cref{stmt:green-ruzsa-linear-bound-dim(P)}, with the discrete John's theorem of \citet{TV08}, \Cref{stmt:discrete-john}.
We defer the details to \Cref{sec:changs-theorem-for-Zn}.

\begin{restatable}{cor}{changgreenruzsa}
\label{stmt:green-ruzsa-with-dim(P)<=dimF(X)}
  There exists $C > 0$ such that the following holds.
  Let $n \in \N$ be a prime, and let $\kappa \ge 2$.
  If $X \subseteq \Zmodn$ satisfies
  \begin{equation}\label{eq:assumptions-in-green-ruzsa}
    \sigma[X] \le \kappa \qquad \text{ and } \qquad C \kappa^3 (\log \kappa)^2 < |X| < \exp(-C \kappa^4 (\log \kappa)^2) n,
  \end{equation}
  then there is a generalised arithmetic progression $P \subseteq \Zmodn$ such that
  \begin{equation}\label{eq:green-ruzsa-properties-of-P}
    X \subseteq P, \qquad |P| \le \exp(C \kappa^4 (\log \kappa)^2 ) |X| \qquad \text{ and } \qquad \dim(P) \le \dimF(X).
  \end{equation}
\end{restatable}

We are now ready to prove that
\begin{equation}\label{eq:prob-of-X1}
  \pr(\exists \, X \in \cX_1 : X \rsum X \subseteq \complement{A_p}) \to 0 \quad \text{ as } \quad n \to \infty.
\end{equation}
\begin{proof}[Proof of \eqref{eq:prob-of-X1}]
  Recall that $k = (2 + 4\delta) \log_{\frac{1}{1 - p}}n$, and that for all $X \in \cX_1$, we know that $|X| = k$ and $\sigma[X] \le k^{1/40}$.
  Fixing $\eps = k^{-1/20}$, we can apply \Cref{stmt:robust-subset} to $X$ with $\eps$ and $\sigma = k^{1/40}$ to conclude that every such set contains a $X'$ of size at least
  \begin{equation}\label{eq:ub-lower-range-size-of-X'}
    |X'| \ge (1 - 2 \eps \sigma)k \ge (1 - 2 k^{-1/40})k
  \end{equation}
  with $\eps$-robust \Freiman dimension $d_X$, for some $d_X \in \N$.
  Observe that \eqref{eq:ub-lower-range-size-of-X'} implies that the doubling of $X'$ is at most $2\sigma$:
  \begin{equation}\label{eq:ub-lower-range-doubling-of-X'}
    \sigma[X'] = \frac{|X' + X'|}{|X'|} \le \frac{|X + X|}{|X'|} \le 2 \sigma,
  \end{equation}
  for all sufficiently large $k$.
  We fix\footnote{Abusing notation to denote such a mapping via the $'$ symbol.} one $X'$ for each $X \in \cX_1$, and denote by $\cX_1'$ the collection of all such $X'$.

  As we remarked before, $X' \rsum X' \subseteq \complement{A_p}$ is implied by $X \rsum X \subseteq \complement{A_p}$ for each $X \in \cX_1$.
  Therefore, we have the bound
  \begin{equation*}
    \pr\big( \exists \, X \in \cX_1 : X \rsum X \subseteq \complement{A_p} \big) \le \pr\big( \exists \, X' \in \cX_1' : X' \rsum X' \subseteq \complement{A_p} \big).
  \end{equation*}
  Moreover, let $\ccY(P)$ be the collection of subsets $Y \subseteq P$ with
  \begin{equation}\label{eq:ub-lower-range-properties-of-Y}
    (1 - 2 k^{-1/40})k \le |Y| \le k \qquad \text{and} \qquad \sigma[Y] \le 2 \sigma
  \end{equation}
  such that $Y$ has $\eps$-robust \Freiman dimension $d_Y$ for some $d_Y \ge \dim(P)$.
  We claim that we can take another union bound:
  \begin{equation}\label{eq:union-bound-over-gaps}
    \pr( \exists \, X' \in \cX_1' : X' \rsum X' \subseteq \complement{A_p} ) \le \sum_{P \in \cP(\Zmodn)} \pr(\exists \, Y \in \ccY(P) : Y \rsum Y \subseteq \complement{A_p} )
  \end{equation}
  where $\cP(\Zmodn)$ is the collection of generalised arithmetic progressions $P \subseteq \Zmodn$ such that
  \begin{equation}\label{eq:definition-of-cal(P)}
    |P| \le \exp(k^{1/5}).
  \end{equation}

  In order to prove \eqref{eq:union-bound-over-gaps}, it is enough to show that, for every $X' \in \cX_1'$, there exists a generalised arithmetic progression $P \in \cP(\Zmodn)$ such that $X' \in \ccY(P)$; we will do so by applying \Cref{stmt:green-ruzsa-with-dim(P)<=dimF(X)} with $\kappa = 2\sigma$.
  From this application, we will obtain a generalised arithmetic progression $P \in \cP(\Zmodn)$ such that
  \begin{equation}\label{eq:ub-lower-range-properties-of-P-via-rect-chang}
    \dim(P) \le \dimF(X') = d_X \qquad \text{ and } \qquad X' \subseteq P.
  \end{equation}
  The property $P \in \cP(\Zmodn)$ follows from \eqref{eq:green-ruzsa-properties-of-P}, as
  \begin{equation}\label{eq:ub-lower-range-size-of-P-from-k}
    |P| \le \exp(C' \sigma^4 (\log \sigma)^2) k \le \exp\hspace{-2pt}\big(C' k^{1/10} (\log k)^2 \big) k \le \exp\hspace{-2pt}\big(k^{1/5}\big),
  \end{equation}
  where we used that $\sigma = k^{1/40}$ in the second inequality, and in the last one we used that $k$ is sufficiently large.
  We now confirm that we can apply \Cref{stmt:green-ruzsa-with-dim(P)<=dimF(X)} as we want, by verifying that every $X' \in \cX_1'$ satisfies \eqref{eq:assumptions-in-green-ruzsa}.
  The lower bound holds because
  \begin{equation*}
    |X'| \ge (1 - 2 k^{-1/40})k \ge k^{1/10} \ge C' \sigma^3 (\log \sigma)^2,
  \end{equation*}
  by \eqref{eq:ub-lower-range-size-of-X'} and using that $k$ is sufficiently large, whereas the upper bound in \eqref{eq:assumptions-in-green-ruzsa} follows from
  \begin{equation*}
    \exp(C' \sigma^4 (\log \sigma)^2) |X'| \le \exp(k^{1/5}) \le \exp( (\log n)^{2/5} ) < n,
  \end{equation*}
  where we used \eqref{eq:ub-lower-range-size-of-P-from-k}, and that $k \le (\log n)^2$ and $n$ is sufficiently large.

  We now claim that $X' \in \ccY(P)$, where $P \in \cP(\Zmodn)$ is the generalised arithmetic progression given by applying \Cref{stmt:green-ruzsa-with-dim(P)<=dimF(X)} to $X' \in \cX_1'$.
  To see this, simply note that the conditions in \eqref{eq:ub-lower-range-properties-of-Y} follow from \eqref{eq:ub-lower-range-size-of-X'} and \eqref{eq:ub-lower-range-doubling-of-X'}, and the $\eps$-robust \Freiman dimension bound $d_X \ge \dim(P)$ follows from \eqref{eq:ub-lower-range-properties-of-P-via-rect-chang} and \Cref{stmt:robust-subset}, so this completes the proof of \eqref{eq:union-bound-over-gaps}.

  In order to bound the term in the right-hand side of \eqref{eq:union-bound-over-gaps}, we analyse the contribution of each fixed $P \in \cP(\Zmodn)$.
  Notice that if $\ccY(P)$ is empty, then the probability term is equal to 0, so we may assume that $\ccY(P)$ is non-empty.
  Rather than directly taking a union bound over choices of $Y \in \ccY(P)$, our final union bound is over a collection of fingerprints $\cF(P)$:
  \begin{align}
    \pr( \exists \, Y \in \ccY(P) : Y \rsum Y \subseteq \complement{A_p} ) & \le \pr( \exists \, F \in \cF(P) : F \rsum F \subseteq \complement{A_p} ) \label{eq:ub-lower-range-union-bound-over-fingerprints} \\
                                                                        & \le |\cF(P)| \max_{F \in \cF(P)} \pr( F \rsum F \subseteq \complement{A_p} ). \label{eq:bound-on-each-term-of-fingerprint-union-bound}
  \end{align}
  That is true as long as, for each $Y \in \ccY(P)$, there exists $F \in \cF(P)$ such that $F \subseteq Y$; again we are using that $F \rsum F \subseteq Y \rsum Y$.
  We claim that applying \hyperlink{stmt:fingerprints-repeated}{\nameCref{stmt:fingerprints}~\ref*{stmt:fingerprints}} to $P$ with $\gamma = \gamma(\delta)$ (to be determined later) and $m = 2\sigma k$ yields such a collection of fingerprints $\cF(P)$.

  First, we define a candidate for $\cF(P)$ which proves our claim, and later we show that we can construct this candidate.
  Our candidate for $\cF(P)$ is defined as
  \begin{equation}\label{eq:application-of-fingerprints-to-P}
    \cF(P) = \bigcup_{Y \in \ccY(P)} \cF_{|Y|, m, \eps}(P)
  \end{equation}
  where $\cF_{|Y|, m, \eps}(P)$ is the collection of fingerprints given by \hyperlink{stmt:fingerprints-repeated}{\nameCref{stmt:fingerprints}~\ref*{stmt:fingerprints}}.
  For $\cF(P)$ to prove our claim, we must thus show that, for each $Y \in \ccY(P)$, there is an $F \in \cF_{|Y|, m, \eps}(P)$ such that $F \subseteq Y$.
  By \hyperlink{stmt:fingerprints-repeated}{property~\ref*{item:fingerprint-is-contained-in-X} of \nameCref{stmt:fingerprints}~\ref*{stmt:fingerprints}}, it suffices for each $Y$ to have $\eps$-robust \Freiman dimension $d_Y \ge \dim(P)$ and $|Y + Y| \le m = 2\sigma k$.
  These, however, hold for $Y$ by \eqref{eq:ub-lower-range-properties-of-Y}, and so we have that $\cF(P)$ is a valid candidate of fingerprints for $P$.

  To confirm that we can apply \hyperlink{stmt:fingerprints-repeated}{\nameCref{stmt:fingerprints}~\ref*{stmt:fingerprints}} to $P$ as in \eqref{eq:application-of-fingerprints-to-P}, we need to check that
  \begin{equation}\label{eq:ub-lower-range-requirements-for-fingerprints}
    m \ge \frac{|Y|(\dim(P) + 1)}{2}
  \end{equation}
  for all $Y \in \ccY(P)$.
  \eqref{eq:ub-lower-range-requirements-for-fingerprints} follows from
  \begin{equation*}
    \frac{|Y|(\dim(P) + 1)}{2} \le \frac{|Y|(\dimF(Y) + 1)}{2} \le |Y + Y| \le 2 \sigma k = m
  \end{equation*}
  where the first inequality uses that $\dim(P) \le \dimF(Y)$ by definition of $\ccY(P)$, the second inequality relies on \Cref{stmt:dim-at-most-twice-doubling} and the third one on \eqref{eq:ub-lower-range-properties-of-Y}.
  Therefore, \eqref{eq:application-of-fingerprints-to-P} is a valid definition for $\cF(P)$.

  We proceed to give an upper bound to the right-hand side of \eqref{eq:bound-on-each-term-of-fingerprint-union-bound}.
  With the goal of first bounding the size of $\cF(P)$, define $\Phi(P) = \max \{|F| : F \in \cF(P)\}$ and note that, trivially,
  \begin{equation*}
    |\cF(P)| \le \sum_{q = 0}^{\Phi(P)} {|P| \choose q} \le (|P| + 1)^{\Phi(P)}.
  \end{equation*}
  Since, by \hyperlink{stmt:fingerprints-repeated}{(\ref*{eq:fingerprint-reqs}) in \nameCref{stmt:fingerprints}~\ref*{stmt:fingerprints}}, $|F| \le C \eps^{-1} \sqrt{m \log m}$ for all $F \in \cF(P)$, we can use that $m = 2\sigma k$, $\eps = k^{-1/20}$ and $\sigma = k^{1/40}$ to obtain
  \begin{equation}\label{eq:bound-on-size-of-fingerprint-collection}
    |\cF(P)| \le \exp(k^{3/5} \, \log |P|) \le \exp(k^{4/5})
  \end{equation}
  where in the first inequality we used that $k$ is sufficiently large, and in the last we used \eqref{eq:ub-lower-range-size-of-P-from-k}.

  To obtain an upper bound on $\pr(F \rsum F \subseteq \complement{A_p})$ for all $F \in \cF(P)$, we use \hyperlink{stmt:fingerprints-repeated}{(\ref*{eq:fingerprint-reqs})}, the lower bound on $|F \rsum F|$ given by \hyperlink{stmt:fingerprints-repeated}{\nameCref{stmt:fingerprints}~\ref*{stmt:fingerprints}}:
  \begin{equation*}
    |F \rsum F| \ge \frac{(1 - \gamma)(\dim(P) + 1)}{2} \min_{Y \in \ccY(P)} |Y|.
  \end{equation*}
  Since $\gamma = \gamma(\delta)$ is a constant and $|Y| \ge (1 - 2 \eps \sigma)k$ for all $Y \in \ccY(P)$ by \eqref{eq:ub-lower-range-properties-of-Y}, we obtain
  \begin{equation}\label{eq:bound-on-fingerprint-prob}
    \max_{F \in \cF(P)} \pr( F \rsum F \subseteq \complement{A_p} ) \le (1 - p)^{(1 - 2 \gamma)(\dim(P) + 1)k/2},
  \end{equation}
  as $k$ is sufficiently large.
  Replacing \eqref{eq:bound-on-size-of-fingerprint-collection} and \eqref{eq:bound-on-fingerprint-prob} back into \eqref{eq:bound-on-each-term-of-fingerprint-union-bound} yields
  \begin{equation}\label{eq:bound-on-each-term-of-union-bound}
    \pr( \exists \, Y \in \ccY(P) : Y \rsum Y \subseteq \complement{A_p} ) \le \exp(k^{4/5}) (1 - p)^{(1 - 2 \gamma)(\dim(P) + 1)k/2}.
  \end{equation}
  This will be our bound for each term in the right-hand side of \eqref{eq:union-bound-over-gaps}.

  Observe that \eqref{eq:bound-on-each-term-of-union-bound} does not depend on the specific choice of $P$, only on its size and dimension.
  Recalling that $|P| \le \exp(k^{1/5}) =: s$ for every $P \in \cP(\Zmodn)$ by \eqref{eq:ub-lower-range-size-of-P-from-k}, we can group terms in the right-hand side of \eqref{eq:union-bound-over-gaps} based on $d = \dim(P)$ to deduce that, by \eqref{eq:bound-on-each-term-of-union-bound}, we have
  \begin{equation}\label{eq:sum-in-ub}
    \pr( \exists \, X \in \cX_1 : X \rsum X \subseteq \complement{A_p} ) \le \sum_{d = 1}^\infty s (n s)^{d + 1} \exp(k^{4/5}) (1 - p)^{(1 - 2 \gamma)(d + 1)k/2},
  \end{equation}
  where we have bounded the number of $d$-dimensional generalised arithmetic progressions in $\Zmodn$ with size at most $s$ by $s (n s)^{d + 1}$.

  It is therefore enough to prove that the right-hand side of \eqref{eq:sum-in-ub} goes to $0$ as $n \to \infty$.
  Note that, as $k = (2 + 4\delta) \log_{\frac{1}{1 - p}}n$, we have $$(1 - p)^{k/2} = n^{-(1 + 2\delta)},$$ which combined with a suitably small choice of $\gamma = \gamma(\delta)$, implies that
  \begin{equation}\label{eq:bound-on-term-that-helps-us}
    n^{d + 1} (1 - p)^{(1 - 2\gamma)(d + 1)k/2} = n^{d + 1 -(1 - 2\gamma)(1 + 2\delta)(d + 1)} \le n^{-\delta (d + 1)}.
  \end{equation}
  The final observation is that it follows from $s = \exp(k^{1/5})$ that there is $1 > \nu > 0$ such that
  \begin{equation}\label{eq:bound-on-term-that-hurts-us}
    s^{d + 2} \exp(k^{4/5}) \le \exp\big( (d + 1) (\log n)^{1 - \nu} \big)
  \end{equation}
  since $k \le 3 (\log n)/p$ for $\delta < 1/4$ and $n$ is sufficiently large.
  The value of $\nu$ depends on the exponent of the $\log n$ term in our choice of $p$, and $\nu = 0.18$ works for $p \ge (\log n)^{-1/80}$.

  Combining \eqref{eq:bound-on-term-that-helps-us} with \eqref{eq:bound-on-term-that-hurts-us}, we show, for sufficiently large $n$, that \eqref{eq:sum-in-ub} is at most
  \begin{equation*}
    \sum_{d = 1}^\infty s (n s)^{d + 1} \exp(k^{4/5}) (1 - p)^{(1 - 2 \gamma)(d + 1)k/2} \le \sum_{d = 1}^\infty n^{-\delta d / 2}
  \end{equation*}
  which goes to 0 as $n \to \infty$, as we wanted to show.
\end{proof}

\subsection{Bounding the probability over choices in $\cX_2$}\label{sec:bound-over-X2}

Our goal is to show that the term corresponding to choices in $\cX_2$ tends to $0$ as $n \to \infty$.
We emphasize that, in this \namecref{sec:bound-over-X2}, no new ideas, or even modification of previous, existing results, are needed.
We just need to use the following result of \citet{Gre05}.

\begin{prop}[\cite{Gre05}, see {\cite[Proposition 6.1]{GM16}}]\label{stmt:greens-estimates-on-sets-with-bounded-doubling}
  For every $k \in \N$ and $m \ge 2k - 1$, there exists $r \in \N$ with
  \begin{equation}\label{eq:bound-on-r}
    r \leq \min\{4 m/k, k\} \quad \text{ and } \quad r \leq \frac{2m}{k} +\frac{1}{k} \binom{r}{2},
  \end{equation}
  such that
  \begin{equation}\label{eq:bound-on-cX2}
    \big|\cX_2^{(m)}\big| \leq n^{r} k^{4k},
  \end{equation}
  where $\cX_2^{(m)} = \{X \in \cX_2 : |X \rsum X| = m\}$.
\end{prop}

With \Cref{stmt:greens-estimates-on-sets-with-bounded-doubling}, we can take a union bound over each $X \in \cX_2$ to obtain
\begin{equation}\label{eq:prob-of-X2}
    \pr\left(\exists \, X \in \cX_2 : X \rsum X \subseteq \complement{A_p} \right) \le \sum_{m = k^{1 + 1/40}}^{\delta k^2/10} \big|\cX_2^{(m)}\big| \, (1 - p)^m \to 0 \quad \text{ as } n \to \infty,
\end{equation}
and we prove the last step below.

\begin{proof}[Proof of \eqref{eq:prob-of-X2}]
  First, we bound the $r$ in \Cref{stmt:greens-estimates-on-sets-with-bounded-doubling} using \eqref{eq:bound-on-r}:
  \begin{equation*}
    r \le \frac{2 m}{k} + \frac{1}{k} \left( \frac{4m}{k} \right )^2 \le \frac{m}{k} \left(2 + 2 \delta\right)
  \end{equation*}
  since sets $X \in \cX_2^{(m)} \subseteq \cX_2$ satisfy $m/k = \sigma[X] \le \delta k / 10$.
  Now, applying \Cref{stmt:greens-estimates-on-sets-with-bounded-doubling} to each $\cX_2^{(m)}$, we obtain that
  \begin{equation}\label{eq:upper-bound-mid-range-first-step}
    \big|\cX_2^{(m)}\big| \le n^r k^{4k} \le n^{(2 + 2 \delta) m / k + 5 (\log n)^{1/50} \log \log n}
  \end{equation}
  and the last inequality follows from $k \le (\log n)^{51/50}$, which is implied by $p \ge (\log n)^{-1/80}$.

  Notice that it follows from $k = (2 + 4\delta) \log_{\frac{1}{1 - p}} n$ that
  \begin{equation}\label{eq:1-p-in-mid-range-upper-bound}
    (1 - p)^m = n^{-(2 + 4 \delta) m / k}.
  \end{equation}
  Together with \eqref{eq:upper-bound-mid-range-first-step} and the fact that $m/k = \sigma[X] \ge k^{1/40} \ge (\log n)^{1/40}$, \eqref{eq:1-p-in-mid-range-upper-bound} yields
  \begin{equation*}
    \big|\cX_2^{(m)}\big| \, (1 - p)^m \le n^{-\delta m / k}
  \end{equation*}
  from which the result follows by summing over all $k^{1 + 1/40} \le m \le \delta k^2/10$.
\end{proof}

\subsection{Bounding the probability over choices in $\cX_3$}

Similarly to the previous case, we start with a union bound
\begin{equation}\label{eq:upper-range-calc1}
    \pr\left(\exists \, X \in \cX_3 : X \rsum X \subseteq \complement{A_p} \right) \leq \sum_{m = \delta k^2/10}^{k^2} \big|\cX_3^{(m)}\big| \, (1 - p)^m,
\end{equation}
letting $\cX_3^{(m)}$ be the sub-collection of $\cX_3$ consisting of subsets $X \subseteq \Zmodn$ with $|X \rsum X| = m$.
We will need the following slight strengthening of Proposition~5.1 of \cite{GM16}.
\begin{restatable}{prop}{appendixprop}\label{stmt:appendix}
  Let $\delta > 0$ be sufficiently small and let $\eta > 0$.
  If $k \le (\log n)^{2 - \eta}$ and $m \ge \delta k^2/10$, then
  \begin{equation*}
    \big|\cX_3^{(m)}\big| \le n^{(2 + \delta + o(1)) m / k},
  \end{equation*}
  as $k \to \infty$, where $\cX_3^{(m)} = \{X \in \cX_3 : |X \rsum X| = m\}$.
\end{restatable}
The proof of \Cref{stmt:appendix} is essentially the same as that in \cite{GM16}; we only optimize constants and exponents.
Therefore, we defer its presentation to \Cref{sec:appendix}.

\begin{proof}[Proof that \eqref{eq:upper-range-calc1} $\to 0$ as $n \to \infty$]
  Observe that $$k \le \frac{3 \log n}{p} \le (\log n)^{41/40} \le (\log n)^{2 - \eta},$$ by our choice of $p \ge (\log n)^{-1/80}$.
  Hence, we can apply \Cref{stmt:appendix} to bound, for every $m \ge \delta k^2/10$, the size of $\cX_3^{(m)}$ by
  \begin{equation*}
    \big|\cX_3^{(m)}\big| \leq n^{(2 + \delta + o(1))m/k} \le n^{(2 + 2 \delta) m / k}
  \end{equation*}
  where the last inequality holds for large $k$.
  We can bound each term of \eqref{eq:upper-range-calc1} by:
  \begin{equation*}
    \big|\cX_3^{(m)}\big| \, (1 - p)^m  \le n^{(2 + 2 \delta) m / k} n^{-(2 - 4\delta) m / k} \le n^{- \delta m / k}.
  \end{equation*}
  Summing over $m \ge \delta k^2/10$ yields the desired result.
\end{proof}

\section{The lower bound}\label{sec:lower-bound}
In this \namecref{sec:lower-bound}, we prove the lower bound in \Cref{stmt:main-result}.
\begin{thm}\label{stmt:lb}
  Let $n$ be a prime number and let $p = p(n)$ satisfy $1/2 \ge p \ge n^{-o(1)}$.
  The random Cayley sum graph $\Gamma_p$ of $\Zmodn$ satisfies
  \begin{equation*}
    \alpha(\Gamma_p) \geq \big(2 + o(1)\big) \log_{\frac{1}{1 - p}} n
  \end{equation*}
  with high probability as $n \to \infty$.
\end{thm}

The proof of this result is significantly easier than the upper bound.
In fact, using only the pseudorandom properties of $\Gamma_p$ is enough to obtain $\alpha(\Gamma_p) \ge \big(1/2 + o(1)\big)(\log n)/p$ (see \cite{Alo07} and \cite[Corollary 2.2]{AKS99}).
To improve the leading constant to $2$, we use both the randomness (as opposed to only the pseudorandomness) and the fact that we can restrict our attention to any sub-collection of potential independent sets in $\Gamma_p$.

More precisely, for each $k \in \N$, we define $Z_k$ to be the random variable counting all independent $k$-sets in $\Gamma_p$ with maximal doubling, that is
\begin{equation}\label{eq:def-Z_k}
  Z_k = \left|\left\{X \in \cZ_k: X \rsum X\subset A_p^c\right\}\right|
\end{equation}
where
\begin{equation}\label{eq:def-cZ_k}
  \cZ_k = \left\{X \subseteq \Zmodn : |X| = k, \, |X \rsum X| = \binom{k}{2}\right\}.
\end{equation}
If $Z_k > 0$, then $\alpha(\Gamma_p) \geq k$ regardless of the potential independent $k$-sets that $Z_k$ overlooks.

In order to prove the lower bound, it is enough to show that $\Var(Z_k) = o(\E[Z_k]^2)$ since
\begin{equation*}
  \pr(\alpha(\Gamma_p) \geq k) \ge \pr(Z_k > 0) \ge 1 - \frac{\Var(Z_k)}{\E[Z_k]^2},
\end{equation*}
by Chebyshev's inequality.
The first step is to estimate $\E[Z_k]$, which we do by showing that $\cZ_k$ is large and using linearity of expectation.
\begin{lem}\label{stmt:size-of-Zk}
  For each $k \in \N$ such that $k = o(n^{1/4})$, we have
  \begin{equation*}
    |\cZ_k| = \big(1 - o(1)\big) \binom{n}{k}.
  \end{equation*}
\end{lem}

\begin{proof}
  We will (equivalently) show that almost all $X \subseteq \Zmodn$ with $|X| = k$ satisfy ${|X \rsum X| = \binom{k}{2}}$.
  Observe that if $|X \rsum X| < \binom{k}{2}$, then there are distinct $x_1, x_2, x_1', x_2' \in X$ such that
  \begin{equation*}
    x_1 + x_2 = x_1' + x_2'.
  \end{equation*}
  Motivated by that observation, define
  \begin{equation*}
    \cQ = \big\{\{x_1, x_2, x_1', x_2'\} \subseteq \Zmodn : x_1 + x_2 = x_1' + x_2', \, \text{ and } x_1, x_2, x_1', x_2' \text{ are distinct} \big\},
  \end{equation*}
  and let $Y$ be a uniformly random $k$-set in $\Zmodn$.
  Taking a union bound over $\cQ$ yields
  \begin{equation*}
    \pr(Y \not \in \cZ_k) \le \sum_{Q \in \cQ} \pr(Q \subseteq Y) \le n^3 \bigg(\frac{k}{n}\bigg)^4 = \frac{k^4}{n},
  \end{equation*}
  where the second inequality is due to the choices of $x_1, x_2$ and $x_1'$ determining $x_2' = x_1 + x_2 - x_1'$ for $\{x_1, x_2, x_1', x_2'\} \in \cQ$.
  Since $k = o(n^{1/4})$, the \namecref{stmt:size-of-Zk} follows.
\end{proof}

The second \namecref{stmt:var} that we need to prove \Cref{stmt:lb} gives a bound on $\Var(Z_k)$.
\begin{lem}\label{stmt:var}
  For every $k \in \N$ and $p = p(n) \in (0, 1)$, we have
  \begin{equation*}
    \Var(Z_k) \le \E[Z_k] + \binom{n}{k} \sum_{s = 1}^k k^{3 s} \binom{n}{k - s} (1 - p)^{2 \binom{k}{2} - k s/2}.
  \end{equation*}
\end{lem}

The main step in the proof of \Cref{stmt:var} is to show, for every $X \in \cZ_k$, that
\begin{equation}\label{eq:goal-for-var-proof}
  \sum_{\substack{Y \in \cZ_k \\ Y \sim X}} (1 - p)^{|(X \rsum X) \cup (Y \rsum Y)|} \le \sum_{s = 1}^k k^{3 s} \binom{n}{k - s} (1 - p)^{2 \binom{k}{2} - k s/2}
\end{equation}
where $Y \sim X$ if
\begin{equation}\label{eq:definition-of-sim}
  Y \neq X \qquad \text{ and } \qquad (X \rsum X) \cap (Y \rsum Y) \neq \emptyset.
\end{equation}
In order to do that, define
\begin{equation*}
  \cI(X, Y) = \big\{Y' \subseteq Y : (X \rsum X) \cap (Y' \rsum Y') = \emptyset\big\},
\end{equation*}
and
\begin{equation}\label{eq:definition-of-cal-I}
  \cI^*(X, Y) = \big\{Y' \in \cI(X, Y) : (y + Y') \cap (X \rsum X) \neq \emptyset \, \text{ for all } \, y \in Y \setminus Y'\big\}.
\end{equation}

The definition of $\cI^*(X, Y)$ is motivated by the following \namecref{stmt:size-of-X+X-cap-Y+Y}, which gives an upper bound on $\big| (X \rsum X) \cap (Y \rsum Y) \big|$:

\begin{lem}\label{stmt:size-of-X+X-cap-Y+Y}
  Let $k \in \N$ and $X, Y \in \cZ_k$.
  If $Y \sim X$, then
  \begin{equation}\label{eq:size-of-X+X-cap-Y+Y}
    \big|(X \rsum X) \cap (Y \rsum Y)\big| \le \frac{k (k - t)}{2},
  \end{equation}
  where $t = \min\big\{|Y'|: Y' \in \cI^*(X, Y)\big\}.$
\end{lem}

To prove \Cref{stmt:size-of-X+X-cap-Y+Y}, we will need the following simple observation about graphs.

\begin{obs}\label{stmt:edges-in-graph-with-large-maximal-independent-sets}
  Let $G$ be a graph with $k$ vertices.
  If all maximal independent sets of $G$ have at least $t$ vertices, then $G$ has at most $k (k - t) / 2$ edges.
\end{obs}

\begin{proof}
  Take $v \in V(G)$ to be a vertex of maximum degree, and $I_v \subseteq V(G)$ to be a maximal independent set containing $v$.
  Then,
  \begin{equation*}
    t \le |I_v| \le k - d(v) = k - \Delta(G).
  \end{equation*}
  Thus, $\Delta(G) \le k - t$, and the claimed bound follows.
\end{proof}

To deduce \Cref{stmt:size-of-X+X-cap-Y+Y}, we will apply \Cref{stmt:edges-in-graph-with-large-maximal-independent-sets} to $\Gamma_{X \rsum X}[Y]$, the Cayley sum graph over $X \rsum X$ restricted to the vertex set $Y$.

\begin{proof}[Proof of \Cref{stmt:size-of-X+X-cap-Y+Y}]
  We claim that
  \begin{equation}\label{eq:intersection-equal-number-of-edges}
    \big|(X \rsum X) \cap (Y \rsum Y)\big| = e(\Gamma_{X \rsum X}[Y]),
  \end{equation}
  which not only implies that the collection $\cI^*(X, Y)$ is exactly the collection of maximal independent sets in $\Gamma_{X \rsum X}[Y]$, but also reduces the proof to applying \Cref{stmt:edges-in-graph-with-large-maximal-independent-sets} with $t$ to $\Gamma_{X \rsum X}[Y]$.
  Therefore, we first establish \eqref{eq:intersection-equal-number-of-edges}, and then the characterization of $\cI^*(X, Y)$.

  To show \eqref{eq:intersection-equal-number-of-edges}, recall that, by the definition of the Cayley sum graph, $y_1 y_2 \in \binom{\,Y}{2}$ is an edge of $\Gamma_{X \rsum X}[Y]$ if and only if $y_1 + y_2 \in X \rsum X$.
  We obtain equality by observing that each sum in $Y \rsum Y$ corresponds to exactly one pair $y_1 y_2 \in \binom{\,Y}{2}$, since $|Y \rsum Y| = \binom{k}{2}$ by $Y \in \cZ_k$.

  It follows from \eqref{eq:intersection-equal-number-of-edges} that $\cI(X, Y)$ corresponds to independent sets in $\Gamma_{X \rsum X}[Y]$.
  Moreover, for every $Y^* \in \cI^*(X, Y)$, we know that there are no $Y' \in \cI(X, Y)$ such that $Y^* \subsetneq Y'$ by the condition in \eqref{eq:definition-of-cal-I}.
  Our definition of $t$ then corresponds to all maximal independent sets in $\Gamma_{X \rsum X}[Y]$ having at least $t$ vertices, so we can apply \Cref{stmt:edges-in-graph-with-large-maximal-independent-sets} as desired.
\end{proof}

With \Cref{stmt:size-of-X+X-cap-Y+Y} in hand, the proof of \eqref{eq:goal-for-var-proof} now relies on an efficient count of $Y \in \cZ_k$ with $Y \sim X$, for each fixed $X \in \cZ_k$.
Notice that $\cI^*(X, Y)$ is not empty, as every graph has at least one maximal independent set, so we can fix a set $Y^* \in \cI^*(X, Y)$ of minimum size.
Our counting strategy is to consider the choices for elements in $Y^*$ and then the choices for $Y \setminus Y^*$; in fact, a trivial count of all sets $Y^* \subseteq \Zmodn$ suffices, so we only need to count the possible elements in $Y \setminus Y^*$ efficiently.

In order to bound the choices for $Y \setminus Y^*$, notice that for every $y \in Y \setminus Y^*$, we have $(y + Y^*) \cap (X \rsum X) \neq \emptyset$ by \eqref{eq:definition-of-cal-I}.
It follows that there are $y^* \in Y^*$ and distinct $x_1, x_2 \in X$ such that $y + y^* = x_1 + x_2$, or, equivalently,
\begin{equation*}
  Y \setminus Y^* \subseteq X \rsum X - Y^*.
\end{equation*}
We can therefore choose the elements of $Y \setminus Y^*$ from a set of size at most
\begin{equation*}
  |X \rsum X - Y^*| \le |X|^2 |Y^*| \le k^3,
\end{equation*}
so there are at most $k^{3s}$ choices for $Y \setminus Y^*$ if $s = |Y \setminus Y^*|$.
Bounding the number of choices for $Y^*$ by $\binom{n}{k - s}$ yields the bound in \eqref{eq:goal-for-var-proof}, except for the $(1 - p)^{2 \binom{k}{2} - k s / 2}$ term, which we obtain by using \Cref{stmt:size-of-X+X-cap-Y+Y}.

We now have all the ingredients to prove \Cref{stmt:var}.

\begin{proof}[Proof of \Cref{stmt:var}]
  Observe that, via standard calculations, we have that
  \begin{equation}\label{eq:var-bound}
    \Var(Z_k) \le \E[Z_k] + \sum_{X \in \cZ_k} \sum_{\substack{Y \in \cZ_k \\ Y \sim X}} (1 - p)^{|(X \rsum X) \cup (Y \rsum Y)|},
  \end{equation}
  where, recall, $Y \sim X$ was defined in \eqref{eq:definition-of-sim}.
  We therefore need to show that
  \begin{equation}\label{eq:var-inner-term-upper-bound}
    \sum_{\substack{Y \in \cZ_k \\ Y \sim X}} (1 - p)^{|(X \rsum X) \cup (Y \rsum Y)|} \le \sum_{s = 1}^k k^{3 s} \binom{n}{k - s} (1 - p)^{2 \binom{k}{2} - k s/2}
  \end{equation}
  for each $X \in \cZ_k$.

  To prove \eqref{eq:var-inner-term-upper-bound}, fix $X \in \cZ_k$ and, for each $Y \in \cZ_k$ with $Y \sim X$, choose a set $Y^* = Y^*(X, Y)$ of minimum size.
  If we group the sets $Y$ by the size of their corresponding $Y^*$, we can count them by first enumerating the choices for $Y^*$ and then the choices for $Y \setminus Y^*$.
  For fixed $s = |Y \setminus Y^*| = k - |Y^*|$, there are at most $\binom{n}{k - s}$ choices for $Y^*$, and there are at most $k^{3 s}$ choices for $Y \setminus Y^*$ since $Y \setminus Y^* \subseteq X \rsum X - Y^*$.

  We then bound the size of the union in \eqref{eq:var-inner-term-upper-bound} by applying \Cref{stmt:size-of-X+X-cap-Y+Y} to the pair $(X, Y)$
  \begin{equation*}
    \big|(X \rsum X) \cap (Y \rsum Y)\big| \le \frac{k s}{2}
  \end{equation*}
  which means that we can ignore the case $s = 0$ because it would contradict $Y \sim X$.
  A trivial inclusion-exclusion now yields the bound we need for the size of the union:
  \begin{equation}\label{eq:cup-of-sumsets}
    \big|(X \rsum X) \cup (Y \rsum Y)\big| = 2 \binom{k}{2} - \big|(X \rsum X) \cap (Y \rsum Y)\big| \ge 2 \binom{k}{2} - \frac{k s}{2}.
  \end{equation}

  Replacing \eqref{eq:cup-of-sumsets} and the above count of $Y$ for fixed $s$ into the left-hand side of \eqref{eq:var-inner-term-upper-bound} gives
  \begin{equation*}
    \sum_{X \in \cZ_k} \sum_{\substack{Y \in \cZ_k \\ Y \sim X}} (1 - p)^{|(X \rsum X) \cup (Y \rsum Y)|} \le \sum_{X \in \cZ_k} \sum_{s = 1}^k k^{3 s} \binom{n}{k - s} (1 - p)^{2\binom{k}{2} - k s/2}.
  \end{equation*}
  Trivially bounding the number of choices for $X \in \cZ_k$ by $\binom{n}{k}$ and plugging the result into \eqref{eq:var-bound} completes the proof.
\end{proof}

With \Cref{stmt:size-of-Zk} and \Cref{stmt:var}, the proof of \Cref{stmt:lb} is just checking that the bounds match those of the statement.

\begin{proof}[Proof of \Cref{stmt:lb}]
  Fix $p = p(n)$ satisfying $1/2 \ge p \ge n^{-\delta/8}$ for some $\delta > 0$, and let $k = (2 - 2\delta) \log_{\frac{1}{1 - p}} n$.
  It suffices to show that
  \begin{equation}\label{eq:goal-for-lower-bound}
    \frac{\Var(Z_k)}{\E[Z_k]^2} \to 0 \qquad \text{ as } \qquad n \to \infty.
  \end{equation}

  First, we compute the expected value of $Z_k$ using \Cref{stmt:size-of-Zk} and linearity of expectation:
  \begin{equation}\label{eq:expec}
    \E[Z_k] = \big(1 - o(1)\big) \binom{n}{k} (1 - p)^{\binom{k}{2}} \ge \frac{1}{2} \binom{n}{k} n^{-(1 - \delta) (k - 1)} \to \infty
  \end{equation}
  as $n \to \infty$, by our choice of $k$.
  Now, if we assume that
  \begin{equation}\label{eq:expectation-dominates-variance-term}
    \E[Z_k] \ge \binom{n}{k} \sum_{s = 1}^k k^{3 s} \binom{n}{k - s} (1 - p)^{2\binom{k}{2} - k s/2},
  \end{equation}
  then, by \Cref{stmt:var}, we have $\Var(Z_k) \le 2 \E[Z_k]$ and
  \begin{equation*}
    \frac{\Var(Z_k)}{\E[Z_k]^2} \le \frac{2}{\E[Z_k]} \to 0 \qquad \text{ as } \qquad n \to \infty,
  \end{equation*}
  by \eqref{eq:expec}.
  We therefore assume that the converse of \eqref{eq:expectation-dominates-variance-term} holds.

  Before we proceed, observe that applying the standard binomial inequality
  \begin{equation*}
    \binom{n}{k - s} \le \left(\frac{k}{n}\right)^s \binom{n}{k}
  \end{equation*}
  to the right-hand side of \eqref{eq:expectation-dominates-variance-term} yields
  \begin{align}
    \binom{n}{k} \sum_{s = 1}^k k^{3 s} \binom{n}{k - s} (1 - p)^{2\binom{k}{2} - k s/2} & \le \binom{n}{k}^2 (1 - p)^{2 \binom{k}{2}} \sum_{s = 1}^k k^{3 s} \left(\frac{k}{n}\right)^s (1 - p)^{-k s/2} \notag \\
                                                                                               & \le 4 \E[Z_k]^2 \sum_{s = 1}^k \left(\frac{k^4}{n^\delta}\right)^s, \label{eq:bound-for-var-term-in-terms-of-expec}
  \end{align}
  where in the last inequality we used \eqref{eq:expec} with $n$ sufficiently large and also
  \begin{equation*}
    (1 - p)^{- k s / 2} = n^{(1 - \delta) s}
  \end{equation*}
  because $k = (2 - 2\delta) \log_{\frac{1}{1 - p}} n$.

  By \Cref{stmt:var}, our assumption that the converse of \eqref{eq:expectation-dominates-variance-term} holds, and \eqref{eq:bound-for-var-term-in-terms-of-expec}, we have
  \begin{equation}
    \Var(Z_k) \le 8 \E[Z_k]^2 \sum_{s = 1}^k \left(\frac{k^4}{n^\delta}\right)^s. \label{eq:bound-for-variance-when-expectation-is-dominated}
  \end{equation}
  Replacing \eqref{eq:bound-for-variance-when-expectation-is-dominated} into \eqref{eq:goal-for-lower-bound}, we conclude that the proof is complete if we show that
  \begin{equation*}
    \sum_{s = 1}^k \left(\frac{k^4}{n^\delta}\right)^s \to 0 \qquad \text{ as } \qquad n \to \infty.
  \end{equation*}
  This is easily seen to be true when $k = o(n^{\delta/4})$, which holds for our choice of $p \ge n^{-\delta/8}$.
\end{proof}

\section{Concluding remarks}\label{sec:concluding-remarks}

The most important open question left by our work is extending the upper bound in \Cref{stmt:main-result} to $p$ as small as possible.
However, already when $p \le (\log n)^{-1}$, there is an obstacle that prevents any approach similar to ours from working.
To understand the barrier, consider the following approximate summary of our strategy.
We find a family $\cS = \{F(X) + F(X) : X \in {\Zmodn \choose k}\}$ of subsets of $\Z_n$ of size $s$ (for some $s \in \N$) with the following two properties:
\begin{enumerate}
  \item For each set $X \in {\Zmodn \choose k}$, there exists $S \in \cS$ with $S \subseteq X + X$.
  \item The family is small, that is, $\left|\cS\right| \leq (1 - p)^{-s}$.
\end{enumerate}
Observe that each $S \in \cS$ is of the form $F(X) + F(X)$ and has size $s$, so we must trivially have $|F(X)| \ge \sqrt{s}$ for all $X \in {\Zmodn \choose k}$.
Naively counting every set $F \subseteq \Zmodn$ with $|F| = \sqrt{s}$ when bounding the size of $\cS$, we obtain
\begin{equation}\label{eq:lower-bound-on-size-of-S-from-all-sumsets}
  |\cS| \ge \binom{n}{\sqrt{s}} \approx \exp(\sqrt{s} \log n).
\end{equation}
In our proof, we show that we can choose $F$ inside a small generalised arithmetic progression $P$, thus replacing the $\log n$ term in \eqref{eq:lower-bound-on-size-of-S-from-all-sumsets} with $\log |P|$.
However, even if we could find such a set $P$ with $|P| = O(\sqrt{s})$, we could not improve the bound in \eqref{eq:lower-bound-on-size-of-S-from-all-sumsets} beyond $\exp(\sqrt{s})$.

Combining this lower bound on the size of $\cS$ with the upper bound that we require in property (2) gives
\begin{equation*}
  \exp(\sqrt{s}) \le |\cS| \le (1 - p)^{-s},
\end{equation*}
which implies that $s \ge p^{-2}$.
Now, consider any $X \in {\Zmodn \choose k}$ with $\sigma[X] = O(1)$: the corresponding set $S \in \cS$ satisfies $S \subseteq X + X$ and therefore
\begin{equation*}
  p^{-2} \le s = |S| \le |X + X| \le O(k).
\end{equation*}
As $k$ is the upper bound we are trying to prove for $\alpha(\Gamma_p)$, it follows that the best we can hope for this approach is, for some constant $C > 2$,
\begin{equation}\label{eq:best-possible-upper-bound-using-fingerprints}
  \alpha(\Gamma_p) \le C \max\{p^{-2}, p^{-1}\log n\},
\end{equation}
where the second term in the maximum is the lower bound that we proved in \Cref{sec:lower-bound}.

While \eqref{eq:best-possible-upper-bound-using-fingerprints} would still be far from what we believe is true for $p$ much smaller than $(\log n)^{-1}$, proving it would still be very interesting; even if we optimized our approach, we could not get close to proving \eqref{eq:best-possible-upper-bound-using-fingerprints} for all $p \ge (\log n)^{-1}$.

Improving the upper bound is related to the following problem due to Shachar Lovett: is there a list of $2^{n^{O(1)}}$ subsets of $\F_2^n$ with density at least $1/100$ such that, whenever $X \subseteq \mathbb{F}_2^n$ has density at least $1/3$, $X + X$ contains one of these sets?
This question is relevant because we can embed sets $X \subseteq \Zmodn$ with $\sigma[X] = O(1)$ as dense subsets of an Abelian group using Ruzsa's model lemma.
Now, if the positive answer to this problem does not rely too much on the details of the setting (the structure of $\mathbb{F}_2^n$ and the particular densities of the subsets) and a similar list exists for any Abelian group, then we could take a union bound over its elements and hope to circumvent \eqref{eq:best-possible-upper-bound-using-fingerprints} in the bounded doubling setting.

We would also like to highlight that for very small $p$, close to $n^{-1} \log n$, upper and lower bounds for $\alpha(\Gamma_p)$ follow from a theorem of \citet{Bou87}; this is the only other result that we are aware of in the regime $p = o(1)$.
Another way to obtain an upper bound in this range is as follows: observe that the first eigenvalue of $\Gamma_p$ is highly concentrated around $n p$, and that we can use elementary Fourier analysis to show that its second eigenvalue is, with high probability, at most $O( \sqrt{n p \log n} )$ (see, for example, \cite[Sections~3.1~and~3.2]{Alo07}).
Using Hoffman's ratio bound, we then obtain
\begin{equation*}
  \alpha(\Gamma_p) \le \bigg( \frac{C n \log n}{p} \bigg)^{1/2},
\end{equation*}
which matches $\alpha(G(n, p))$ up to a logarithmic factor if $p \approx n^{-1} \log n$.

Another interesting open question is to determine the minimum number of translates necessary to obtain the correct leading constant of $d + 1$ in \Freiman's lemma.
This is closely related to the question asked (and partially settled) by \citet*{BLT22+a} on whether three translates suffice to obtain the Cauchy--Davenport lower bound for two sets $A, B \subseteq \Zmodn$.
Even though their question was recently resolved by \citet*{FLPZ24} in the affirmative, it is not clear what is the truth in higher dimensions, even in the simpler case of $A = B$.
We conjecture the following:

\begin{conjecture}\label{stmt:optimal-freimans-lemma-via-few-trans}
  There exists $C > 0$ such that the following holds.
  For all $d \in \N$ and all finite sets $X \subseteq \R^d$ of full rank, $X$ contains a subset $T$ such that $|T| \le C d$ and
  \begin{equation}\label{eq:conj-optimal-freimans-lemma-via-few-trans}
    |X + T| \ge (d + 1)|X| - d^C.
  \end{equation}
\end{conjecture}

The best negative result we have for this problem is given by the following simple construction.
Let $d \ge 3$, and let $\{e_1, e_2, \ldots, e_d\}$ be the canonical basis of $\R^d$.
Consider
\begin{equation*}
  X = P + \{0, e_1, \ldots, e_{d - 1}\} \qquad \text{ where } \qquad P = \{0, e_d, 2e_d, \ldots, ke_d\}
\end{equation*}
for some $k \in \N$.
It is not hard to show that the following holds for this example: for every $\gamma > 0$, there is $c = c(\gamma) > 0$ such that if $T \subseteq X$ with $|T| < (2 - \gamma)d$, then
\begin{equation*}
  |X + T| \le (1 - c) (d + 1)|X|.
\end{equation*}

\Cref{stmt:freiman-lemma-via-few-translates} is actually not the best result that we know towards \Cref{stmt:optimal-freimans-lemma-via-few-trans}.
While we were finishing the writing of this paper, \citet{GGMT23+} made a breakthrough and proved Marton's conjecture, also known as the Polynomial \Freiman--Ruzsa (PFR) conjecture for finite fields.
In a previous work, \citet[Corollary 1.12]{GMT25} proved that a positive answer to Marton's conjecture would imply what is known as the ``weak'' Polynomial \Freiman--Ruzsa conjecture for $\Z^d$: for every $X \subseteq \Z^d$ with $\sigma[X] \le \sigma$, there exists $X' \subseteq X$ such that
\begin{equation*}
  \rank(X') \le C \log \sigma \qquad \text{ and } \qquad |X'| \ge \sigma^{-C} |X|
\end{equation*}
for some absolute constant $C > 0$.

Combining this weak version of PFR in $\Z^d$ with a variation of the argument we developed in \Cref{sec:weighted-freiman}, we can prove a version of \Cref{stmt:freiman-lemma-via-few-translates} that attains a bound of the form $|X + T| \ge (d + 1 - O(\log d))|X|$ at the cost of $d^C$ translates, for $C > 0$ an absolute constant.
We develop these ideas in a separate work, dedicated to answering a question of \citet{GM16} about the number of subsets of $\{1, \ldots, n\}$ with prescribed doubling, where we require that both the number of translates is a polynomial in $d$, and the leading constant is at least $1 - o(1)$.
Nonetheless, we think that \Cref{stmt:optimal-freimans-lemma-via-few-trans} should not depend on results as deep as variants of PFR.

The final future research direction we would like to highlight is the extension of our result to other groups.
During the preparation of this paper, it came to our knowledge that \citet*{CFPY23++} proved an upper bound of $O(\log N \log \log N)$ for the clique/independence number in the uniform random Cayley graph of any group $G$, where $N$ is the order of $G$.
As there are some groups, like $\F_2^n$, where this is tight up to the constant factor, their result can be seen as a generalisation of \citeauthor{Gre05}'s theorem~\cite{Gre05} about $\alpha(\Gamma_{1/2})$.

Moreover, over certain groups, like $\F_5^n$, they can show that there exist Cayley graphs which have both clique and independence number $(2 + o(1)) \log N$, even though a uniform random Cayley graph has clique number $\Theta( \log N \log \log N)$.
It would be interesting if our methods, combined with their techniques, can extend these results to sparser graphs.

\section*{Acknowledgements}

We would like to greatly thank Rob Morris for carefully reading the paper, and for the many improvements and corrections he suggested.
Without his support, this work would not have its current form.
We would also like to thank Zach Hunter and an anonymous referee for carefully reading the paper, pointing out mistakes and giving suggestions.
Finally, we are grateful to Lucas Arag\~ao for helpful discussions, and to David Conlon, Huy Tuan Pham and Shachar Lovett for comments on the manuscript.

This study was financed in part by the Coordenação de Aperfeiçoamento de Pessoal de Nível Superior, Brasil (CAPES).

\pagebreak
\bibliographystyle{abbrvnat}

\def\bibfont{\footnotesize}
\bibliography{bib}

\appendix

\section{Chang's theorem for $\Zmodn$}\label{sec:changs-theorem-for-Zn}

This \namecref{sec:changs-theorem-for-Zn} is devoted to the proof of \Cref{stmt:green-ruzsa-with-dim(P)<=dimF(X)}, the variant of \citeauthor{Cha02}'s theorem that we used in the proof of \Cref{stmt:main-result}.
\changgreenruzsa*

We rely on two theorems to prove \Cref{stmt:green-ruzsa-with-dim(P)<=dimF(X)}: \citeauthor{CS13}'s strengthening of the Green--Ruzsa theorem (\Freiman's theorem for general Abelian groups)~\cite{CS13} and the discrete John's theorem of \citet{TV08}.
To state these two results, we need some auxiliary definitions.
As in \cite{CS13}, define a coset progression to be a set of the form $P + H$, where $H$ is a subgroup of $G$ and $P$ is a generalised arithmetic progression.
Moreover, the dimension of $P + H$ is $\dim(P)$, we say that a coset progression is proper if (a) $P$ is proper, and (b) $|P + H| = |P| \, |H|$.
We say that $P + H$ is $2$-proper if $(P + P) + H$ is proper.

\begin{thm}[{\citet[Theorem 4]{CS13}}]\label{stmt:green-ruzsa-linear-bound-dim(P)}
  There exists $C' > 0$ such that the following holds.
  Let $G$ be an Abelian group, and let $X \subseteq G$ be a finite set with $\sigma[X] \le \kappa$.
  Either there exists a proper coset progression $P + H$ such that
  \begin{equation*}
    X \subseteq P + H, \quad \dim(P + H) \le 2 \kappa + 1 \quad \text{ and } \quad |P + H| \le \exp(C' \kappa^4 (\log (\kappa + 2))^2) |X|,
  \end{equation*}
  or $X$ is fully contained in at most $C' \kappa^3 (\log \kappa)^2$ cosets, whose total cardinality is bounded by $\exp(C' \kappa^4 (\log (\kappa + 2))^2)|X|$, of some subgroup of $G$.
\end{thm}

As our group of interest is $\Zmodn$, a proper coset progression $P + H$ is either all of $\Zmodn$, or simply a proper generalised arithmetic progression $P + 0 = P$.
We also need a \namecref{stmt:2-proper-from-proper} by the same authors to obtain, from a $d$-dimensional coset progression $P + H$, a $2$-proper coset progression that contains $P + H$, has dimension at most $d$, and whose size is not much larger than $|P + H|$.

\begin{lem}[{\citet[Lemma 6]{CS13}}]\label{stmt:2-proper-from-proper}
  There exists $C' > 0$ such that the following holds.
  Suppose that $P + H$ is a $d$-dimensional coset progression in an Abelian group $G$.
  Then there exists a $2$-proper coset progression $P' + H'$ such that
  \begin{equation*}
    P + H \subseteq P' + H', \qquad \dim(P' + H) \le d, \qquad \text{ and } \qquad |P' + H'| \le d^{C' d^2} |P + H|.
  \end{equation*}
\end{lem}

These two results enable us to give an overview of the (somewhat standard) proof of \Cref{stmt:green-ruzsa-with-dim(P)<=dimF(X)}.
First, we apply \Cref{stmt:green-ruzsa-linear-bound-dim(P)} to $X$.
Using \eqref{eq:assumptions-in-green-ruzsa} and \Cref{stmt:2-proper-from-proper}, we will show that this application yields a small $2$-proper generalised arithmetic progression $P \subseteq \Zmodn$ such that $X \subseteq P$.
However, the bound on $d = \dim(P)$ is not strong enough for our purposes, so we use it instead to define a \Freiman isomorphism $\phi : P \to P_d \subseteq \Z^d$ in the natural way.

Restricting $P_d$ to an $r = \dimF(X)$ dimensional subspace, then, yields an $r$-dimensional \emph{convex} progression, which is a set of the form $\cC = K \cap \Z^d$.
In this definition, $K \subseteq \R^d$ is a convex set of rank $r \le d$.
At this point, we use \citeauthor{TV08}'s discrete John's theorem~\cite{TV08} (\Cref{stmt:discrete-john} below) to obtain a small generalised arithmetic progression $P' \subseteq \Z^d$ such that $\cC \subseteq P'$ and $\dim(P') \le r$.
We complete the proof by taking the progression $\phi^{-1}(P') \subseteq \Zmodn$.

The above sketch is not quite right for two reasons.
The first is that $\phi^{-1}$ may not be defined for every element in $P'$, but the linearity of the mapping allows us to extend it to all of $\Z^d$.
The second reason is that the statement of \Cref{stmt:discrete-john} requires the convex progression to be (centrally) symmetric, \ie $\cC = -\cC$.
We amend this by applying it instead to $\cC - \cC$, bounding its size by $|P - P| \le 2^d |P|$, and its dimension by that of the same $d'$-dimensional subspace that contains $\cC$.

\begin{thm}[{\citet[Theorem 1.6, Lemma 3.3]{TV08}}]\label{stmt:discrete-john}
  There exists $C' > 0$ such that the following holds.
  Let $d, r \in \N$ with $r \le d$.
  For any $r$-dimensional symmetric convex progression $\cC \subseteq \Z^d$, there exists a generalised arithmetic progression $P \subseteq \Z^d$ such that
  \begin{equation*}
    \cC \subseteq P, \qquad |P| \le r^{C' r^2} |\cC| \qquad \text{ and } \qquad \dim(P) \le r = \dim(\cC).
  \end{equation*}
\end{thm}

The bound for the size of the generalised arithmetic progression in \Cref{stmt:discrete-john} has since been improved \cite{BH19, vHK23+}, but the one by \citet{TV08} suffices for our purposes.
Their statement crucially defines the convex progression on an arbitrary lattice $\Lambda$ over $\R^d$, rather than $\Z^d$.
This allows us to avoid the full rank requirement $r = d$ in their original result, and take $r \le d$ as in \Cref{stmt:discrete-john} by first setting $\Lambda = \cH \cap \Z^d$, where $\cH$ is an $r$-dimensional subspace containing $\cC$.
With this lattice, we then take a suitable linear isomorphism $f$ and apply their original result to $\cC' = f(\cC)$ with $\Lambda' = f(\Lambda)$ and $\R^r = f(\cH)$.

Using these results, we can now prove \Cref{stmt:green-ruzsa-with-dim(P)<=dimF(X)}.

\begin{proof}[Proof of \Cref{stmt:green-ruzsa-with-dim(P)<=dimF(X)}]
  As all the subgroups of $\Zmodn$ are trivial when $n$ is a prime, we claim that we cannot obtain the second case when applying \Cref{stmt:green-ruzsa-linear-bound-dim(P)} to $X$.
  Observe that none of the cosets obtained in the second outcome can be $\Zmodn$ itself.
  The reason for that is that a single one of those cosets would, by \eqref{eq:assumptions-in-green-ruzsa}, exceed the total cardinality bound given in \Cref{stmt:green-ruzsa-linear-bound-dim(P)} if $C > C'$ is sufficiently large: $$|\Zmodn| = n > \exp(C' \kappa^4 (\log (\kappa + 2))^2)|X|.$$
  The claim follows by noting that the cosets also cannot all be of the form $a \in \Zmodn$, since their total number is bounded by $C' \kappa^3 (\log \kappa)^2$, and, by \eqref{eq:assumptions-in-green-ruzsa}, this is insufficient to cover $X$.

  We have shown that we always get the first outcome of \Cref{stmt:green-ruzsa-linear-bound-dim(P)} under our assumptions, so applying it to $X$ yields a proper coset progression $P'' + H \subseteq \Zmodn$ such that
  \begin{equation*}
    X \subseteq P'' + H, \qquad \dim(P'') \le 2 \kappa + 1 \qquad \text{ and } \qquad |P'' + H| \le \exp(C \kappa^4 (\log \kappa)^2) |X|,
  \end{equation*}
  where we used that $\kappa \ge 2$ and that $C > C'$ is appropriately large.
  Applying now \Cref{stmt:2-proper-from-proper} to $P'' + H$ gives a $2$-proper coset progression $P' + H' \subseteq \Zmodn$ such that
  \begin{equation}\label{eq:properties-of-P'-in-green-ruzsa}
    X \subseteq P' + H', \quad d = \dim(P') \le 2 \kappa + 1 \quad \text{ and } \quad |P' + H'| \le \exp(C \kappa^4 (\log \kappa)^2) |X|,
  \end{equation}
  again by $C > C'$.
  Moreover, $H' = \{0\}$ because $H' = \Zmodn$ is impossible by \eqref{eq:assumptions-in-green-ruzsa}: $$\exp(C \kappa^4 (\log \kappa)^2) |X| < n = |P' + \Zmodn|,$$
  so we will write simply $P'$ for $P' + H'$.

  Let $\phi : P' \to P'_d \subseteq \Z^d$ be the function defined by
  \begin{equation*}
    \phi(y) = (w^{(y)}_1, \ldots, w^{(y)}_d) \quad \text{ where } \quad y = a_0 + \sum_{i = 1}^d w_i^{(y)} a_i
  \end{equation*}
  and
  \begin{equation*}
    P' = \bigg\{ a_0 + \sum_{i = 1}^d w_i a_i : w_i \in \Z, \, 0 \le w_i < \ell_i \bigg\}.
  \end{equation*}
  Since $P'$ is $2$-proper, we know that $\phi$ is a \Freiman isomorphism; this and $X \subseteq P'$ imply that so is $\phi|_X : X \to \phi(X)$.
  Let $X_d = \phi(X)$ and assume without loss of generality (by translating $\phi$ if necessary) that $0 \in X_d$.
  If we define $\cH$ to be the minimal subspace that contains $X_d$, then $\cC' = \cH \cap P'_d$ is a convex progression and $\cC = \cC' - \cC'$ is a symmetric convex progression satisfying $\cC' \subseteq \cC$.
  We can therefore apply \Cref{stmt:discrete-john} to $\cC$ and obtain a generalised arithmetic progression $P_d \subseteq \Z^d$ which, we claim, satisfies
  \begin{equation}\label{eq:properties-of-P_d-in-green-ruzsa}
    X_d \subseteq P_d, \qquad \dim(P_d) \le \dimF(X) \qquad \text{ and } \qquad |P_d| \le \exp(C \kappa^4 (\log \kappa)^2) |X|.
  \end{equation}

  The containment $X_d \subseteq P_d$ follows from $X_d \subseteq \cC' \subseteq \cC \subseteq P_d$, where the last step is given by \Cref{stmt:discrete-john}.
  In order to bound the dimension of $P_d$, let $d_\cC = \dim(\cC)$ and observe that
  \begin{equation}\label{eq:dimension-of-P_d}
    \dim(P_d) \le d_\cC \le \dim(\cH) \le \dimF(X)
  \end{equation}
  by \Cref{stmt:discrete-john}, $\cC \subseteq \cH$ and the definition of $\cH$ as the minimal subspace containing $X_d$.
  Finally, to obtain the claimed bound on the size of $P_d$, we first use that, by \Cref{stmt:discrete-john},
  \begin{equation}\label{eq:first-step-in-bounding-size-of-P_d}
    |P_d| \le d_\cC^{C' d_\cC^2} |\cC| \le \exp(C \kappa^2 \log \kappa) |\cC|
  \end{equation}
  where we used that $d_\cC \le d \le 2\kappa + 1$ by \eqref{eq:properties-of-P'-in-green-ruzsa} and that $C > C'$ is sufficiently large.
  The second step is showing that the size of $\cC$ is at most
  \begin{equation}\label{eq:second-step-in-bounding-size-of-P_d}
    |\cC| = |\cC' - \cC'| \le |P_d' - P_d'| \le 2^d |P_d'| \le \exp(C \kappa^4 (\log \kappa)^2) |X|,
  \end{equation}
  which is due first to $\cC' \subseteq P_d'$, and then by taking a slightly larger $C$ than previous occurrences and using \eqref{eq:properties-of-P'-in-green-ruzsa}.
  Combining \eqref{eq:first-step-in-bounding-size-of-P_d} and \eqref{eq:second-step-in-bounding-size-of-P_d} yields the desired bound on $|P_d|$ for a yet slightly larger value of $C$, and completes the proof of \eqref{eq:properties-of-P_d-in-green-ruzsa}.

  It remains to define a generalised arithmetic progression $P \subseteq \Zmodn$ from $P_d \subseteq \Z^d$.
  Observe that $\phi^{-1}$ is linear (or affine), so we can extend it to obtain a $\psi : \Z^d \to \Zmodn$ which satisfies
  \begin{equation}\label{eq:properties-of-psi-in-green-ruzsa}
    \psi|_{X_d} = (\phi|_X)^{-1}.
  \end{equation}
  Hence, $P = \psi(P_d)$ is a (not necessarily proper) generalised arithmetic progression such that $$X \subseteq P, \qquad |P| \le \exp(C \kappa^4 (\log \kappa)^2) |X| \qquad \text{ and } \qquad \dim(P) \le \dimF(X),$$ by \eqref{eq:properties-of-P_d-in-green-ruzsa}, the linearity of $\psi$ and \eqref{eq:properties-of-psi-in-green-ruzsa}.
  This is exactly \eqref{eq:green-ruzsa-properties-of-P}, so the proof is complete.
\end{proof}

\section{Counting sets with large doubling}\label{sec:appendix}

Recall that $n$ is a large prime and that we are dealing with sets $X \subseteq \Zmodn$ of size $k$.
We restate the result which we aim to prove for the reader's convenience:
\appendixprop*

The proof of \Cref{stmt:appendix} is essentially identical to that of~\cite[Proposition 5.1]{GM16}, and to obtain the statement above we only need to optimize the dependencies between the parameters.
Before proceeding, we recall a key definition from \cite{GM16}:
\begin{defi}
  A set $\{x_1, \ldots, x_d\} \subseteq \Zmodn$ is $M$-dissociated if $\sum_{i = 1}^d \lambda_i x_i \neq 0$ for every $(\lambda_1, \ldots, \lambda_d) \in \Z^d \setminus \{0\}$ such that $\sum_{i = 1}^d |\lambda_i| \le M$.
\end{defi}

It is useful to understand this notion as an analogue of linear independence for elements of $\Zmodn$ taking coefficients in $\Z$, with a restriction on the sum of their magnitudes.
The first result that we need to optimize is \Cref{stmt:bound-on-M-dimension}, a slight strengthening of \cite[Lemma 5.3]{GM16}.
\begin{lem}\label{stmt:bound-on-M-dimension}
  For every sufficiently small $\delta > 0$ and $\eta > 0$, the following holds.
  Fix a large prime $n$ and $M = (\log n) / (\log \log n)^4$.
  If $k \le (\log n)^{2 - \eta}$ and $m \ge \delta k^2 / 10$, then any $M$-dissociated subset of $X \in \cX_3^{(m)}$ has size at most $\big(2 + \delta + o(1)\big)m / k$, as $k \to \infty$.
\end{lem}

We will begin by giving an overview of the proof from \cite{GM16}, alongside some useful definitions of our own.
Whenever there is $(\lambda_1, \ldots, \lambda_d) \in \Z^d$ such that $x = \sum_{i = 1}^d \lambda_i x_i$ and $\sum_{i = 1}^d |\lambda_i| \le M$, we will say that $x$ is in the $M$-span of $\{x_1, \dots, x_d\}$.
In that same situation, we will say that $x$ can be written as an $M$-bounded combination of $\{x_1, \ldots, x_d\}$.

The proof starts by fixing the value of $M = (\log n)/(\log \log n)^4$ and any $M$-dissociated set $D = \{x_1, \ldots, x_d\} \subseteq X$.
We define $G'$ to be the graph with vertex set $\Zmodn$ and edges between a pair of vertices if their difference can be written as $\pm x_i \pm x_j$ for some $i, j \in \{1, \ldots, d\}$.
The graph $G'$ is useful because of the following \namecref{stmt:connected-components-of-G-are-spans}:

\begin{obs}\label{stmt:connected-components-of-G-are-spans}
  For each $a \in \Zmodn$, let $V_a \subseteq V(G')$ denote its connected component in $G'$.
  Every element $b \in V_a$ can be expressed as $b = a + \sum_{j = 1}^d \lambda_j x_j$ with $\sum_{j = 1}^d |\lambda_j| \le 2 \diam(V_a)$.
\end{obs}

\begin{proof}
  Take $b \in V_a$, and consider any shortest path $P$ connecting $a$ to $b$ in $G'$.
  Observe that
  \begin{equation*}
    b - a = \sum_{j = 1}^d (\lambda_j^+ - \lambda_j^-) x_j
  \end{equation*}
  where $\lambda_j^+, \lambda_j^-$ count respectively the number of times $x_j, -x_j$ appear as terms in $P$.
  Then,
  \begin{equation*}
    \sum_{j = 1}^d \big(|\lambda_j^+| + |\lambda_j^-|\big) \le 2 |P| \le 2 \diam(V_a)
  \end{equation*}
  where the last step follows from $P$ being a shortest path in $V_a$.
\end{proof}

Having thus related the span of $D$ with the diameter of connected components in $G'$, we will decompose $G = G'[X]$ using the following \namecref{stmt:weak-regularity-appendix}, also from \cite{GM16}, with $\Delta = k^{1/2} \log k$:

\begin{lem}[{\cite[Lemma 5.4]{GM16}}]\label{stmt:weak-regularity-appendix}
  Let $G$ be a graph with vertex set $V(G)$ and let $\Delta > 1$.
  There exists a partition $V(G) = X_* \cup X_1 \cup \cdots \cup X_t$ such that
  \begin{enumerate}
    \item $|X_*| \le 32 (v(G)/\Delta)^2$.
    \item $e(X_i, X_j) = 0$ for every $i \neq j$.
    \item The diameter of $G[X_i]$ is at most $\Delta$ for every $i \in \{1, \ldots, t\}$.
  \end{enumerate}
\end{lem}

Applying \Cref{stmt:weak-regularity-appendix} to $G$ yields a partition $X = X_* \cup X_1 \cup \cdots \cup X_t$ such that
\begin{enumerate}[(i)]
  \item if $X_i \neq X_j$, then $X_i + D$ and $X_j + D$ are disjoint, and
  \item every $X_i$ has a $y_i \in X_i$ such that $X_i - y_i$ is contained in the $2\Delta$-span of $D$.
\end{enumerate}
To see that property (i) holds, notice that if $X_i + D$ and $X_j + D$ intersect, then there is an edge of $G$ connecting $X_i$ and $X_j$, contradicting item (2) of \Cref{stmt:weak-regularity-appendix}.
Property (ii), on the other hand, follows directly from an application of \Cref{stmt:connected-components-of-G-are-spans} to $X_i$, the diameter of which is bounded by item (3) in \Cref{stmt:weak-regularity-appendix}.

To combine these two properties into a proof of \Cref{stmt:bound-on-M-dimension}, notice that the disjointness given by property (i) yields a lower bound for $X + X$ in terms of each $X_i$: $$|X + X| \ge \sum_i |X_i + D|.$$
We now want a lower bound for each $|X_i + D|$ in terms of $d$.
Towards that goal, we will define, for every $i \in \{1, \ldots, t\}$, a \Freiman isomorphism $\phi_i : X_i \to X_i' \subseteq \Z^d$.
Each will map (a translate of) $D$ to a structured set, where structured only means that we have a lower bound for its sumset with any sufficiently small set.
These isomorphisms are naturally defined by the $2\Delta$-bounded decomposition of $X_i - y_i$ given by property (ii).

\begin{obs}\label{stmt:appendix-freiman-isomorphism}
  The function $\phi_i : X_i \to X_i' \subseteq \Z^d$ defined by $\phi_i(x) = (\lambda_1, \ldots, \lambda_d)$, where $x - y_i = \sum_{j = 1}^d \lambda_j x_j$ is a $2\Delta$-bounded decomposition of $x - y_i$, is a \Freiman isomorphism.
\end{obs}

\begin{proof}
  The fact that $\phi^{-1} : X_i' \to X_i$ is a \Freiman homomorphism follows from its linearity, so we focus on the other direction of the implication.
  Take $a_1, a_2, a_3, a_4 \in X_i$ such that
  \begin{equation}\label{eq:freiman-isomorphism}
    a_1 + a_2 = a_3 + a_4,
  \end{equation}
  and we want to show that $\phi_i(a_1) + \phi_i(a_2) = \phi_i(a_3) + \phi_i(a_4)$.
  For $\ell \in \{1, \ldots, 4\}$, write
  \begin{equation}\label{eq:decomposition-of-ai}
    a_\ell - y_i = \sum_{i = 1}^d x_i \lambda^{(\ell)}_i,
  \end{equation}
  a $2 \Delta$-bounded combination of $D$.
  We now replace \eqref{eq:decomposition-of-ai} in \eqref{eq:freiman-isomorphism} and rearrange to obtain
  \begin{equation*}
    \sum_{j = 1}^d x_j (\lambda_j^{(1)} + \lambda_j^{(2)} - \lambda_j^{(3)} - \lambda_j^{(4)}) = 0.
  \end{equation*}
  Using the fact that each \eqref{eq:decomposition-of-ai} is $2\Delta$-bounded yields, as $k \to \infty$,
  \begin{equation*}
    \sum_{\ell = 1}^4 \sum_{j = 1}^d |\lambda_j^{(\ell)}| \le 4 (2\Delta) \le 8 k^{1/2} \log k \le (\log n)^{1 - \eta / 3} \le M,
  \end{equation*}
  where the inequalities follow by the definitions/assumptions which we now recall
  \begin{equation*}
    \Delta = k^{1/2} \log k, \qquad k \le (\log n)^{2 - \eta} \qquad \text{ and } \qquad M = (\log n)/(\log \log n)^4.
  \end{equation*}
  However, $D$ is $M$-dissociated, so we conclude that, for all $j \in \{1, \ldots, d\}$,
  \begin{equation*}
    (\lambda_j^{(1)} + \lambda_j^{(2)}) - (\lambda_j^{(3)} + \lambda_j^{(4)}) = 0,
  \end{equation*}
  which, by the definition of $\phi_i$, is just another way to write $\phi_i(a_1) + \phi_i(a_2) = \phi_i(a_3) + \phi_i(a_4)$.
\end{proof}

Observe that $\phi_i(D + y_i) = \{e_1, \ldots, e_d\}$, where $\{e_1, \ldots, e_d\}$ is the canonical basis of $\Z^d$.
Hence, to obtain a lower bound for the size of $X_i + D$ in terms of $d$ and complete the proof of \Cref{stmt:bound-on-M-dimension}, we will bound $|\phi_i(X_i) + \phi_i(D + y_i)|$ using the isoperimetric inequality of \citet{WW77}, as stated in~\cite{GM16} (see also~\cite{BL91}).

\begin{prop}[{\cite{WW77}, see~{\cite[``The Isoperimetric Inequality'']{GM16}}}]\label{stmt:isoperimetric-inequality}
  For every $\gamma > 0$ and $C > 0$, there exists $d_0 = d_0(\gamma, C)$ such that the following holds for every $d \ge d_0$.
  If $S \subseteq \Z^d$ is a set of size at most $C d$, then
  \begin{equation*}
    \big|S + \{e_1,\ldots,e_d\}\big| \ge \left(\frac{1}{2} - \gamma\right) d |S|.
  \end{equation*}
\end{prop}

\begin{proof}[Proof of \Cref{stmt:bound-on-M-dimension}]
  Fix $X \in \cX_3^{(m)}$ and let $D = \{x_1, \ldots, x_d\}$ be any $M$-dissociated subset of $X$.
  Towards our aim of showing that $|D| = d \le \big(2 + \delta + o(1)\big) m / k$, recall the definition of the graph $G'$: its vertices are $\Zmodn$ and there are edges between vertices $a, b \in \Zmodn$ only when some of $\pm (a - b)$ can be written as either $x_i - x_j$ or $x_i + x_j$ for $i, j \in \{1, \ldots, d\}$.

  We apply \Cref{stmt:weak-regularity-appendix} to $G = G'[X]$ with $\Delta = k^{1/2} \log k$ and obtain $X_* \cup X_1 \cup \cdots \cup X_t$, a partition of $X$ such that $e(X_i, X_j) = 0$ for all $i \neq j$,
  \begin{equation}\label{eq:bound-on-diam-Xi}
    \diam(G[X_i]) \le k^{1/2} \log k
  \end{equation}
  and
  \begin{equation}\label{eq:bound-on-X-trash-set}
    |X_*| \le 32 \left(\frac{k}{k^{1/2} \log k}\right)^2 = o(k).
  \end{equation}

  In order to obtain a lower bound for $X + X$, we use property (i) of $G$ to show that
  \begin{equation}\label{eq:bound-on-X+X-by-parts}
    |X + X| \ge \sum_{i = 1}^t |X_i + D|.
  \end{equation}

  Now, we define, for each $i \in \{1, \ldots, t\}$, the function $\phi_i$; by property (ii) and \Cref{stmt:appendix-freiman-isomorphism}, we have that $\phi_i$ is a \Freiman isomorphism from $X_i$ to $X_i' \subseteq \Z^d$.
  By the definition of \Freiman isomorphisms, we obtain
  \begin{equation}\label{eq:freiman-iso-for-Xi}
    |X_i + D| = |X_i + (D + y_i)| = |X_i' + \phi_i(D + y_i)|.
  \end{equation}

  Recall that the definition of $\phi_i$ implies that $\phi_i(D + y_i) = \{e_1, \ldots, e_d\}$.
  Hence, to apply \Cref{stmt:isoperimetric-inequality} to the right-hand side of \eqref{eq:freiman-iso-for-Xi}, we need to bound $|X_i'|$ in terms of $d$ and show that $d$ is sufficiently large.
  We will show both of these by observing that if $k > 10d/\delta$, then we are already done, since $$d \le \frac{\delta k}{10} \le \frac{m}{k} \le \big(2 + \delta + o(1)\big)\frac{m}{k},$$ which is what we wanted to show.
  We can therefore assume that $d \ge \delta k / 10$, which simultaneously implies that $d$ is large when $k \to \infty$, and that $|X_i'| \le k \le 10d/\delta$.
  Applying \Cref{stmt:isoperimetric-inequality} to $X_i'$ with $\gamma = \delta^2$ and $C = 10/\delta$, we obtain
  \begin{equation}\label{eq:bound-on-Xi'-sumset}
    \big|X_i' + \{e_1, \ldots, e_d\}\big| \ge \left(\frac{1}{2} - \delta^2 \right) d |X_i'|
  \end{equation}
  for each $i \in \{1, \ldots, t\}$.

  Replacing \eqref{eq:bound-on-Xi'-sumset} in \eqref{eq:bound-on-X+X-by-parts}, and using \eqref{eq:freiman-iso-for-Xi}, we conclude that
  \begin{equation*}
    |X + X| \ge d \left( \frac{1}{2} - \delta^2 \right ) \sum_{i = 1}^t |X_i'|.
  \end{equation*}
  Thus, using the fact that $|X_i'| = |X_i|$, which follows from each $\phi_i$ being a \Freiman isomorphism, and that $X = X_* \cup X_1 \cup \cdots X_t$ is a partition, we have
  \begin{equation*}
    d \left( \frac{1}{2} - \delta^2 \right ) \sum_{i = 1}^t |X_i'| = d \left( \frac{1}{2} - \delta^2 \right ) \big(|X| - |X_*|\big) = d \left( \frac{1}{2} - \delta^2 \right) (1 - o(1)) k
   \end{equation*}
   where we used \eqref{eq:bound-on-X-trash-set} to bound $|X_*|$.
   Finally, it follows from $|X + X| \le m$ and $\delta$ being sufficiently small that $d \le \big(2 + \delta + o(1)\big)m / k$ as $k \to \infty$.
\end{proof}

With \Cref{stmt:bound-on-M-dimension}, we need only one more piece, the following \namecref{stmt:count-of-dissociated-coefficients}.
It is a simple count of the number of choices for coefficients in a dissociated set, and we will use it to repeat the proof of Proposition~5.1 in \citet{GM16} and obtain \Cref{stmt:appendix}.

\begin{lem}[{\cite[Lemma 5.5]{GM16}}]\label{stmt:count-of-dissociated-coefficients}
  For every $M, d \in \N$, the number of choices for $\lambda_1, \ldots, \lambda_d \in
  \Z$ such that $\sum_{i = 1}^d |\lambda_i| \le M$ is at most $(4d)^M$.
\end{lem}

\begin{proof}[Proof of \Cref{stmt:appendix}]
  Fix $M = (\log n) / (\log \log n)^4$ and take $X \in \cX_3^{(m)}$.
  We will count the choices for $X$ by first choosing a maximal $M$-dissociated subset $D = \{x_1, \ldots, x_d\} \subseteq X$ and then choosing the remaining elements of $X \setminus D$ using the properties of $D$.

  First, we count the choices for $D$ naively, and rely on \Cref{stmt:bound-on-M-dimension} to bound its size.
  That is, we apply \Cref{stmt:bound-on-M-dimension} to $D$, obtain ${|D| \le d \le \big(2 + \delta + o(1)\big)m / k}$, and thus deduce that the number of choices for $D$ is at most:
  \begin{equation}\label{eq:count-of-D}
    \sum_{t = 1}^d \binom{n}{t} \le (n + 1)^d \le n^{(2 + \delta + o(1)) m / k}.
  \end{equation}

  The second step is counting the choices for $X \setminus D$, and we do so by counting the possible ways to write each of its elements as an $M$-bounded combination of $D$.
  Fix $x' \in X \setminus D$, and note that the maximality of $D$ implies that there is $\Lambda = \{\lambda_0, \lambda_1, \ldots, \lambda_d\} \subseteq \Z$ such that $$\lambda_0 x' + \sum_{j = 1}^d \lambda_j x_j = 0$$ and the elements of $\Lambda$ satisfy $$\sum_{j = 0}^d |\lambda_j| \le M, \qquad \lambda_0 \neq 0 \qquad \text{ and } \qquad \lambda_i \neq 0 \text{ for some } i > 0.$$
  We can therefore use \Cref{stmt:count-of-dissociated-coefficients} to count the number of choices for $\Lambda$ and observe that
  \begin{equation*}
    (4d + 4)^M \le \exp(3 M \log k) \le \exp\left(\frac{\log n}{(\log \log n)^2} \right) \le n^{1/\log k}
  \end{equation*}
  where the first inequality follows from the (trivial) observation that $d \le k$, and the rest is a consequence of our choice of $M = (\log n) / (\log \log n)^4$ and our assumption that $k \le (\log n)^{2 - \eta}$.

  We can now choose each of the $k - d$ elements of $X \setminus D$ by the above procedure, and obtain that there are at most
  \begin{equation}\label{eq:count-of-X-minus-D}
    n^{(k - d)/(\log k)} = n^{o(k)}
  \end{equation}
  such elements.
  Combining \eqref{eq:count-of-D} and \eqref{eq:count-of-X-minus-D} with $m \ge \delta k^2 / 10$ thus yields
  \begin{equation*}
    \big| \cX_3^{(m)} \big| \le n^{(2 + \delta + o(1)) m / k + o(k)} = n^{(2 + \delta + o(1)) m / k},
  \end{equation*}
  as required.
\end{proof}

\end{document}